\documentclass[12pt]{article}

\usepackage{style}

\usepackage{nine, braket, nicematrix}

\theoremstyle{definition}
\newtheorem{mycase}{Case}
\newtheorem{step}{Step}

\dateline{XXX X, 202X}{XXX X, 202X}{TBD}



\title{On linear lexicographic codes: \\ Ninth column construction \\ of the ternary Golay code}

\author[1]{Yuki Irie\thanks{Supported by JSPS KAKENHI Grant Number JP20K14277 and JP24K16892.}}

\affil[1]{%
Graduate Faculty of Interdisciplinary Research, University of Yamanashi, Yamanashi, Japan

\email{}%
}

\newcommand{\resetmycase}{\setcounter{mycase}{0}}
\newcommand{\resetstep}{\setcounter{step}{0}}
\begin{document}

\maketitle

\begin{abstract}
We characterize linear lexicographic \(p\)-ary codes.
Using this characterization, when $\ccp \ge 3$, we determine the dimensions of linear lexicographic codes
obtained from several bases including the standard basis, except for those of certain minimum distances.
In these excluded cases, we may obtain linear codes of higher dimensions;
for instance, when \(p = 3\) and \(d = 6\), the ternary Golay code is obtained.

\end{abstract}

\section{Introduction}
\label{sec:org8b4adb6}
Lexicographic codes were introduced by Levenshtein \cite{levenshtein-class-1960}
 and independently by Conway and Sloane \cite{conway-lexicographic-1986}.
Although a binary lexicographic code is linear,
a non-binary lexicographic code is not always linear,
so several algorithms for producing linear codes have been proposed \cite{bonn-forcing-1996, zanten-construction-2005}.
In this paper, we use a slightly different approach.
For a prime \(p\), we investigate linear 
lexicographic \(p\)-ary codes produced by the original algorithm developed by Levenshtein and Conway-Sloane.
We first give a characterization for linear lexicographic 
\(p\)-ary codes.
Using this characterization, when \(p > 2\),
we determine the dimensions of linear lexicographic codes
obtained from several bases including the standard basis,
except for cases with specific minimum distances. 
In these excluded cases,
we may obtain linear codes of higher dimensions; 
for instance, when \(p = 3\) and \(d = 6\), we can obtain
the ternary Golay code.

\subsection{Lexicographic codes}
\label{sec:orgae3f54a}
In this paper, we only consider lexicographic codes over 
a prime field \(\FF_{\ccp}\). 
We identify \(\FF_{\ccp }\) with \(\set{0, 1, \ldots, \ccp - 1}\).

Let \(\FF_\ccp^\NN\) denote the set of \([\ccx_\cci]_{\cci \in \NN}\) such that
\(\ccx_\cci \in \FF_p\) and \(\ccx_\cci = 0\) for all but only finitely many \(\cci \in \NN\),
where \(\NN\) is the set of non-negative integers.
Let \(\ccF\) be an (ordered) basis \((\bdf_{\cci})_{\cci \in \NN}\) of \(\FF_\ccp^\NN\).
For \(\bdx \in \FF_\ccp^\NN\), we write 
\[\bdx = \sum_{\cci \in \NN} \ithComp{\ccx}[\cci; \ccF]  \bdf_{\cci}.\]
When no confusion can arise, 
we simply write \(\ithComp{x}[\cci]\) instead of \(\ithComp{x}[\cci; \ccF]\).
Let \(\ccE\) denote the standard basis 
\((\bde_0, \bde_1, \ldots)\) of \(\FF_\ccp^\NN\),
where \(\bde_i = \mWord{0 \cdots,  0  1  0  \cdots}\) with 1 in the \(\cci\)-th coordinate;
for example, \(\bde_0 = \mWord{1  0  0  \cdots}\)
and \(\bde_1 = \mWord{0 1 0  \cdots}\).
We write \(\ithComp{\ccx}<\cci> = \ithComp{\ccx}[\cci, \ccE]\).

A basis \(F\) of \(\FF_\ccp^\NN\) induces a total order on \(\FF_p^\NN\) as follows.
Let \(\bdx\) and \(\bdy\) be two distinct elements in \(\FF_\ccp^\NN\).
Then there exists \(\ccN \in \NN\) such that
\(\ithComp{\ccx}[\ccN]\relax \neq \ithComp{\ccy}[\ccN]\) and \(\ithComp{\ccx}[\nIi] = \ithComp{\ccy}[\nIi]\) for \(\nIi > \ccN\).
We write \(\bdx <_{\ccF} \bdy\) if \(\ithComp{\ccx}[\ccN]\relax < \ithComp{\ccy}[\ccN]\).
For example, if \(\ccp = 3\) then
\[
 \mWord{0 0 0 \cdots} <_{\ccE} \mWord{1 0 0 \cdots}  <_{\ccE} \mWord{2 0 0 \cdots} <_{\ccE} \mWord{0 1 0 \cdots} <_{\ccE}  \cdots
\]

Let \(d \in \NN\) with \(d \ge 2\).
For \(\cca \in \NN\), 
let \(\word|\ccF, \ccd|(\cca)\) or simply \(\word(\cca)\) denote the minimum of \(\bdz \in \FF_\ccp^\NN\) with respect to \(<_{\ccF}\) such that
\(\ccd(\bdz, \word(\ccb)) \ge \ccd\) for \(0 \le \ccb < \cca\),
where \(\ccd(\bdx, \bdy)\) is the Hamming distance between \(\bdx\) and \(\bdy\),
that is, \(\ccd(\bdx, \bdy) = \Size{\set{\cci \in \NN : \ccx_\cci \neq \ccy_\cci}}\).
Define \(\word<\cci>(\cca)\) and \(\word[\cci](\cca)\) by
\[
 \word(\cca) = \sum_{\cci \in \NN} \word<\cci>(\cca) \bde_i =  \sum_{\cci \in \NN} \word[\cci](\cca) \bdf_i.
\]
For \(\nIk \in \NN\), let
\[
 \Lex|\ccF, \ccd|{\nIk} = \set{\word(0), \word(1), \ldots, \word(\ccp^{\nIk} - 1)}.
\]
We call \(\Lex|\ccF, \ccd|{\nIk}\) a \emph{lexicographic code}.

Let \(\bdx = [x_\cci]_{\cci \in \NN} \in \FF_p^\NN\).
For a subset \(S\) of \(\NN\), let \(\mRes[S](\bdx)\) denote 
the vector obtained from \(\bdx\) by deleting coordinates in \(\NN \setminus S\), that is,
\(\mRes[S](\bdx) = [\ithComp{\ccx}<\cci>]_{\cci \in S}\).
Let \(\msupp|\ccF|(\bdx) = \set{\cci \in \NN : \ithComp{x}[\cci; \ccF] \neq 0}\).
For example, if \(\bdx = \bde_0 + \bde_2 = \mWord{1 0 1 0 0 \cdots}\),
then \(\msupp|\ccE|(\bdx) = \set{0, 2}\) and \(\mRes[\set{0, 1, 2}](\bdx) = \mWord{1 0 1}\).
If \(\Lex{}\) is a subset of \(\FF_p^\NN\), then let
\[
 \msupp|\ccF|(\Lex{}) = \bigcup_{\bdx \in \Lex{}} \msupp|\ccF|(\bdx)
\]
and
\[
 \mRes(\Lex{}) = \Set{\mRes[\msupp|\ccE|(\Lex{})](\bdx) : \bdx \in \Lex{}}.
\]

\begin{example}
 \comment{Exm.}
\label{sec:org44bb357}
\label{orgcb2e1cd}
Let \(\ccp = 3\), \(\ccF = \ccE\), and \(\ccd = 2\).
Then
\begin{align*}
 \word(0) &= \mWord{0 0 0 0 0 \cdots},\ \
 \word(1)  = \mWord{1 1 0 0 0 \cdots},\ \
 \word(2)  = \mWord{2 2 0 0 0 \cdots},\\
 \word(3) &= \mWord{1 0 1 0 0 \cdots},\ \
 \word(4)  = \mWord{0 1 1 0 0 \cdots},\ \
 \word(5)  = \mWord{2 0 2 0 0 \cdots},\\ 
 \word(6) &= \mWord{0 2 2 0 0 \cdots},\ \
 \word(7)  = \mWord{1 0 0 1 0 \cdots},\ \
 \word(8)  = \mWord{0 1 0 1 0 \cdots}.
\end{align*}
Therefore 
\begin{align*}
 \mRes(\Lex{1}) &= \Set{\mWord{0 0},\ \mWord{1 1},\ \mWord{2 2}},\\
 \mRes(\Lex{2}) &= \Set{\mWord{0 0 0 0},\ \mWord{1 1 0 0},\ \ldots,\ \mWord{0 1 0 1}}.
\end{align*}
Note that \(\mRes(\Lex{1})\) is a linear code, but \(\mRes(\Lex{2})\) is not.
 
\end{example}

\subsection{Linearity of lexicographic codes}
\label{sec:org928374d}

\begin{theorem}[\hspace{0.1ex}\cite{conway-lexicographic-1986, levenshtein-class-1960}] \comment{Thm. [\hspace{0.1ex}\cite{conway-lexicographic-1986, levenshtein-class-1960}]}
\label{sec:orgce1444d}
\label{org3f1ceb3}
If \(\ccp = 2\), then \(\Lex|\ccF, \ccd|{\nIk}\) is a linear code.

\end{theorem}

\begin{example}[\hspace{0.1ex}\cite{conway-lexicographic-1986}]
 \comment{Exm. [\hspace{0.1ex}\cite{conway-lexicographic-1986}]}
\label{sec:orgfa29451}
If \(\ccp = 2\), then \(\mRes(\Lex|\ccE, 8|{12})\) is the (extended) binary Golay code.
 
\end{example}

\comment{connect}
\label{sec:orga36bae7}
In contrast to binary lexicographic codes,
a \(p\)-ary lexicographic code may not be linear when \(\ccp \ge 3\)
as we have seen in Example \ref{orgcb2e1cd}.

Our first result gives the conditions for \(\Lex|\ccF, \ccd|{\nIk}\) to be linear.
For \(\cca \in \NN\), let \(\ithComp{\cca}<\nIi>\) denote the
 \(\nIi\)-th digit of \(\cca\) in the \(\ccp\)-ary expansion, that is,
\(\cca = \sum \ithComp{\cca}<\nIi> \ccp^{\nIi}\) and 
\(\ithComp{\cca}<\nIi> \in \set{0, 1, \ldots, \ccp - 1}\).
We define \(\word+|\ccF, \ccd|(\cca)\) as follows:
\[
\word+|\ccF, \ccd|(\cca) =  \sum_{\nIi \in \NN} \ithComp{\cca}<\nIi> \word|\ccF, \ccd|(\ccp^{\nIi}).
\]
Note that \(\word+|\ccF, \ccd|(\ccp^\nIi) = \word|\ccF, \ccd|(\ccp^\nIi)\).

\begin{theorem} \comment{Thm.}
\label{sec:orgcd4d6d8}
\label{org100bcd6}
The following three conditions are equivalent.
\begin{enumerate}
\item \(\Lex|\ccF, \ccd|{\nIk}\) is a linear code.
\item \(\word|\ccF, \ccd|(\cca) = \word+|\ccF, \ccd|(\cca)\) for \(0 \le \cca < \ccp^{\nIk}\).
\item \(\word|\ccF, \ccd|(\cca) = \word+|\ccF, \ccd|(\cca)\) for \(0 \le \cca < \ccp^{\nIk}\) with \(\sum_{\nIi \in \NN} \cca_{\nIi} \le \ccp - 1\).
\end{enumerate}

\end{theorem}

\begin{remark}
 \comment{Rem.}
\label{sec:org81eba8d}
Theorem \ref{org100bcd6} can be considered as a generalization of
Theorem \ref{org3f1ceb3} because (3) in Theorem \ref{org100bcd6} 
holds trivially when \(p = 2\). 
Indeed, if \(p = 2\) and \(\sum \cca_{\nIi} \le 1\), then
\(\cca\) must be a power of 2, and hence \(\word|\ccF, \ccd|(\cca) = \word+|\ccF, \ccd|(\cca)\).
 
\end{remark}

\subsection{Dimensions of linear lexicographic codes}
\label{sec:orgedb17e0}
Using Theorem \ref{org100bcd6}, we will determine the dimensions of several linear lexicographic codes.

Recall that when \(p = 2\), the lexicographic code \(\mRes(\Lex|E, 8|{12})\)
is the Golay code. However,
when \(p = 3\), the lexicographic code \(\mRes(\Lex|\ccE, 6|{6})\) is
not the ternary Golay code, and is not even a linear code;
in \cite{conway-atlas-1986}, it is pointed out that a modified algorithm produces 
the ternary Golay code (see Example \ref{org0fc9cf0}).
We here give another construction of the ternary Golay code.
For distinct non-negative integers \(\xi\) and \(\eta\), let \(\ccF(\xi, \eta)\) be the basis \((\bdf_i)_{\cci \in \NN}\) of \(\FF_p^\NN\) defined by
\[
 \bdf_i = \begin{cases} \bde_\cci & \tif \cci \neq \xi, \\ 
   \bde_\xi + \bde_\eta & \tif \cci = \xi. \end{cases}
\]
The basis \(\ccF(\xi, \eta)\) differs from standard basis \(E\) only in the \(\xi\)-th element.
Let 
\[\sF = \set{\ccE} \cup \set{\ccF(\xi, \eta) : \xi, \eta \in \NN, \xi \neq \eta}.
\]
For convenience, let \(\ccE = \ccF(\infty, \infty)\).

For a basis \(\ccF\) of \(\FF_{\ccp}^{\NN}\),
let \(\MaxDimLinearLex|\ccF, \ccd|\)
denote the maximum \(\cck\) such that
\(\Lex|\ccF, \ccd|{\nIk}\) is a linear code;
if \(\Lex|\ccF, \ccd|{\nIk}\) is a linear code for all \(\cck\),
then let \(\MaxDimLinearLex|\ccF, \ccd| =  \infty\).
The following theorem says that a slight modification of the standard basis allows the lexicographic code algorithm to produce the ternary Golay code.

\begin{theorem} \comment{Thm.}
\label{sec:org17b2b3d}
\label{orgeadf707}
Let \(p = 3\) and \(\ccF \in \sF\). Then
\[
\MaxDimLinearLex|\ccF, 6| = \begin{cases}
6 & \tif \ccF = \ccF(\xi, 9),\ 0 \le \xi \le 8,\\
3 & \tif \ccF = \ccF(\xi, \eta),\ 0 \le \xi \le 8,\ \eta \ge 10,\\
2 & \totherwise.\\
\end{cases}
\]
In particular, if \(\ccF = \ccF(\xi, 9)\) and \(0 \le \xi \le 8\),
then \(\mRes(\Lex|\ccF, 6|{6})\) is the ternary Golay code.

\end{theorem}

\comment{connect}
\label{sec:orgf369433}
Note that, in Theorem \ref{orgeadf707}, the dimension \(\MaxDimLinearLex|\ccF, 6|\) is significantly larger
in the case when \(\ccF = \ccF(\xi, 9)\) and \(0 \le \xi \le 8\).
The following theorem shows that, except for certain \(d\), this phenomenon does not occur.

\begin{theorem} \comment{Thm.}
\label{sec:orgf3c5d1d}
\label{org9ad3185}
Let \(\ccF \in \sF\).
\begin{enumerate}
\item \(\MaxDimLinearLex|\ccF, \ccd|  = \infty\) if  \(\ccp = 2\).
\item \(\MaxDimLinearLex|\ccF, \ccd|  \ge \nIk\) if \(\ccp = 3\), \(\ccd = 3^{\nIk - 1} \ccd'\), and \(\ccd' \ge 2\);
moreover, if \(\nIk - 1\) is the \(3\)-adic order of \(\ccd\) and \(\ccd' \ge 8\), then
\(\MaxDimLinearLex|\ccF, \ccd| = \nIk\).
\item \(\MaxDimLinearLex|\ccF, \ccd| \le 2\) if \(\ccp \ge 5\).
\end{enumerate}

\end{theorem}

\comment{connect}
\label{sec:org2df7c01}
While proving Theorem \ref{org9ad3185}, 
we will determine, in certain cases, the structure of the following lexicographic-code variants
discussed in \cite{conway-atlas-1986, bonn-forcing-1996}. 
For a non-negative integer \(\cca\), 
define \(\word-|\ccF, \ccd|(\cca)\), or simply \(\word-(\cca)\), as follows.
If \(\cca\) is a power of \(\ccp\), then \(\word-(\cca)\) is
the minimum of \(\bdz \in \FF_\ccp^\NN\) such that
\(\ccd(\bdz, \word-(\ccb)) \ge \ccd\) for \(0 \le \ccb < \cca\);
otherwise, \(\word-(\cca) = \sum_{\nIi \in \NN} \ithComp{\cca}<\nIi> \word-(\ccp^\nIi)\).
For \(\nIk \in \NN\), let
\[
 \Lex-|\ccF, \ccd|{\nIk} = \set{\word-(0), \word-(1), \ldots, \word-(\ccp^{\nIk} - 1)}.
\]
We define \(\word-<\nIi>(\cca)\) and \(\word-[\nIi](\cca)\) 
in the same way as 
\(\word<\nIi>(\cca)\) and \(\word[\nIi](\cca)\), that is,
\(\word(\cca) = \sum \word-<\nIi>(\cca) \bde_\nIi = \sum \word-[\nIi](\cca) \bdf_\nIi\).
Note that, by Theorem \ref{org100bcd6}, we see that if \(\Lex|\ccF, \ccd|{\cck}\) is linear, then 
\(\word(\cca) = \word+(\cca) = \word-(\cca)\) for \(0 \le \cca < \ccp^{\cck}\), and \(\Lex|\ccF, \ccd|{\cck} = \Lex-|\ccF, \ccd|{\cck}\).

\begin{example}[\hspace{0.1ex}\cite{conway-atlas-1986}]
 \comment{Exm. [\hspace{0.1ex}\cite{conway-atlas-1986}]}
\label{sec:org830358a}
\label{org0fc9cf0}
If \(\ccp = 3\), then \(\mRes(\Lex-|\ccE, 6|{6})\) is the ternary Golay code.
 
\end{example}

\begin{example}
 \comment{Exm.}
\label{sec:org7ed8f61}
Let \(p = 3\). Then
\(\mRes(\Lex-|\ccE, 2|{1})\) is generated by \(\mWord{1 1}\)
and \(\mRes(\Lex-|\ccE, 2|{2})\) is generated by \(\mWord{1 1 0}\) and \(\mWord{1 0 1}\).
Note that \(\mRes(\Lex-|\ccE, 2|{1}) = \mRes(\Lex|\ccE, 2|{1})\)
and \(\mRes(\Lex-|\ccE, 2|{2}) \neq \mRes(\Lex|\ccE, 2|{2})\).
 
\end{example}

\begin{theorem} \comment{Thm.}
\label{sec:orga37c967}
\label{org56f1397}
If \(\ccd = \ccp^{\nIk - 1} \ccd'\) with \(\ccd' \ge 2\) and \(\ccF \in \sF\), 
then \(\mRes(\Lex-|\ccF, \ccd|{\nIk})\) is a \(\bigl[\frac{\ccp^{\nIk} - 1}{\ccp - 1} \ccd',\ \nIk,\ \ccp^{\nIk - 1} \ccd' \bigr]\) code,
which is obtained by repeating the \(\bigl[\frac{\ccp^{\nIk} - 1}{\ccp - 1},\ \nIk,\ \ccp^{\nIk - 1} \bigr]\) simplex code \(\ccd'\) times.

\end{theorem}

\begin{theorem} \comment{Thm.}
\label{sec:org3a15845}
\label{org3214412}
Let \(\ccd = \ccp^{\nIk - 1} (\ccp \mQuo + \mRem)\), where \(0 < \mRem < \ccp\).
If \(\mQuo + \mRem - \ccp + 1 \ge 2\), then, 
\(\mRes(\Lex-|\ccF, \ccd|{\nIk + 1})\) is 
a \(\bigl[\frac{\ccp^{\nIk} - 1}{\ccp - 1} (\ccp \mQuo + \mRem) + \mQuo + 1,\ \nIk + 1,\ \ccd \bigr]\) 
Solomon-Stiffler code.

\end{theorem}

\subsection{Organization}
\label{sec:org2b10730}

This paper is organized as follows.
In Section \ref{org0746dba},
we prove Theorem \ref{org100bcd6}.
After reviewing the Griesmer bound
in Section \ref{org05ef9f9},
we prove Theorem \ref{org56f1397}
in Section \ref{orgb6599eb}.
In Section \ref{orgd92423c},
we show that
\(\mLex|\ccF, \ccd|{\nIk - 1}\) is linear
if \(\ccp = 3\) and \(\ccd = 3^{\nIk - 1} \ccd'\) with \(\ccd' \ge 2\).
Section \ref{org67955c0} proves Theorem \ref{org3214412}.
Finally, we  prove Theorem \ref{org9ad3185} in Section \ref{org2aaef62}.

\section{A characterization of linear lexicographic codes}
\label{sec:org209446d}
\label{org0746dba}

We fix a prime \(p\), an integer \(d\) with \(d \ge 2\), and
a basis \(\ccF\) of \(\FF_p^\NN\).
For \(\cca, \ccb \in \NN\), let \(\cca \oplus_p \ccb = \sum_{\nIi \in \NN} (\cca_\nIi \oplus_p \ccb_\nIi) \ccp^{\nIi}\)
and \(\cca \ominus_p \ccb = \sum_{\nIi \in \NN} (\cca_\nIi \ominus_p \ccb_\nIi) \ccp^{\nIi}\),
where \(\cca_\nIi \oplus_p \ccb_\nIi\) and \(\cca_\nIi \ominus_p \ccb_\nIi\) are the addition and subtraction of \(\cca_\nIi\) and \(\ccb_\nIi\) in \(\FF_{p}\), respectively.
For example, if \(p = 3\), then 
\begin{align*}
 7 \oplus_3 13 &= (1 + 2 \cdot 3 + 0 \cdot 3^2) \oplus_3 (1 + 1 \cdot 3 + 1 \cdot 3^2) \\
 &= 
(1 \oplus_3 1) + (2 \oplus_3 1) \cdot 3 + (0 \oplus 1) \cdot 3^2 =  2 + 0 \cdot 3 + 1 \cdot 3^2 = 11
\end{align*}
and 
\[
7 \ominus_3 13 = (1 + 2 \cdot 3) \ominus_3 (1 + 3 + 3^2) = 3 + 2 \cdot 3^2 = 21.
\]
Note that
\begin{align*}
\word+(\cca \oplus_{\ccp} \ccb) &= 
\sum (\ithComp{\cca}<\nIi> \oplus_{\ccp} \ithComp{\ccb}<\nIi>) 
\word(\ccp^{\nIi})
 = \sum \ithComp{\cca}<\nIi> \word(\ccp^{\nIi}) + \sum \ithComp{\ccb}<\nIi> \word(\ccp^{\nIi}) = \word+(\cca) + \word+(\ccb),\\
\word-(\cca \oplus_{\ccp} \ccb) &= 
\sum (\ithComp{\cca}<\nIi> \oplus_{\ccp} \ithComp{\ccb}<\nIi>) 
\word-(\ccp^{\nIi})
 = \sum \ithComp{\cca}<\nIi> \word-(\ccp^{\nIi}) + \sum \ithComp{\ccb}<\nIi> \word-(\ccp^{\nIi}) = \word-(\cca) + \word-(\ccb).
\end{align*}

We will repeatedly use the following simple lemma.

\begin{lemma}
 \comment{Lem.}
\label{sec:org76a0d82}
\label{org6f41f57}
Let \(\bdy \in \FF_\ccp^\NN\).

\begin{enumerate}
\item If \(d(\bdy,\ \word(\ccb)) \ge \ccd\) for \(0 \le \ccb < \cca\), then \(\bdy \ge_{\ccF} \word(\cca)\).
\item If \(d(\bdy,\ \word-(\ccb)) \ge \ccd\) for \(0 \le \ccb < \ccp^{\nIk}\), then \(\bdy \ge_{\ccF} \word-(\ccp^{\nIk})\).
\end{enumerate}
 
\end{lemma}

\begin{proof}
 \comment{Proof.}
\label{sec:org4f5484b}
By definition, \(\word(\cca)\) is the minimum of \(\bdz\) such that \(d(\bdz, \word(\ccb)) \ge \ccd\) for \(0 \le \ccb < \cca\).
Thus \(\word(\cca) \le_{\ccF} \bdy\). The proof of (2) is the same. 
\end{proof}

\comment{connect}
\label{sec:org665d2c6}
For \(\nIi, \cca \in \NN\), we define \(\word+[\cci](\cca^{})\) by \(\word+(\cca) = \sum_{\cci} \word+[\cci](\cca) \bdf_{\cci}\).
The next lemma is a key to prove Theorem \ref{org100bcd6}.
\begin{lemma}
 \comment{Lem.}
\label{sec:org1b18e78}
\label{org13a9541}
Suppose that \(\cca\) is the smallest integer satisfying
\(\word(\cca) \neq \word+(\cca)\). Let
\[
 \ccN = \max \set{\cci \in \NN : \word[\cci](\cca) \neq \word+[\cci](\cca)}.
\]
If \(\ccb\) is a positive integer such that 
\(\ithComp{\ccb}<\nIi> \le \ithComp{\cca}<\nIi>\) for \(\nIi \in \NN\), then
\begin{equation}
\label{equ:lem-linearity}
 \word+[\ccN](\ccb) \neq 0.
\end{equation}
 
\end{lemma}

\begin{proof}
 \comment{Proof.}
\label{sec:org96b84d6}
We first show that \(\word(\cca) <_{\ccF} \word+(\cca)\).
Let \(\ccp^{\nIk - 1} \le \cca < \ccp^{\nIk}\).
Since \(\word(\cca) \neq \word+(\cca)\), it follows that \(\ccp^{\nIk - 1} < \cca\). 
For \(0 \le \ccb < \cca\),
\begin{align*}
  \mDist{\word+(\cca)}{\word(\ccb)} &= \mDist{\word+(\cca)}{\word+(\ccb)} \\
  &= \mBigDist{\, \sum_{\nIi = 0}^{\nIk - 1}\ithComp{\cca}<\nIi> \word(\ccp^\nIi)}{\sum_{\nIi = 0}^{\nIk - 1}\ithComp{\ccb}<\nIi> \word(\ccp^\nIi)}\\
 &= \mBigDist{(\ithComp{a}<\nIk - 1> \ominus_{\ccp} \ithComp{b}<\nIk - 1>)\word(\ccp^{\nIk - 1})}{\ \sum_{\nIi = 0}^{\nIk - 2} (\ithComp{b}<\nIi> \ominus_{\ccp} \ithComp{a}<\nIi>) \word(\ccp^{\nIi})}.
\end{align*}
Let \(\ccc = \sum_{\nIi = 0}^{\nIk - 2} (\ithComp{b}<\nIi> \ominus_{\ccp} \ithComp{a}<\nIi>) \ccp^{\nIi}\). Since \(\ccc < \cca\), we see that
\[
\word(\ccc) = \word+(\ccc) =  \sum_{\nIi = 0}^{\nIk - 2} (\ithComp{b}<\nIi> \ominus_{\ccp} \ithComp{a}<\nIi>) \word(\ccp^{\nIi}).
\]
If \(\ithComp{a}<\nIk - 1> \ominus_{\ccp} \ithComp{b}<\nIk - 1> = 0\), then \(\mDist{\word+(\cca)}{\word(\ccb)} = \mDist{\word(0)}{\word(\ccc)} \ge \ccd\).
If \(\ithComp{a}<\nIk - 1> \ominus_{\ccp} \ithComp{b}<\nIk - 1> \neq 0\), then
\begin{align*}
&\mBigDist{(\ithComp{a}<\nIk - 1> \ominus_{\ccp} \ithComp{b}<\nIk - 1>)\word(\ccp^{\nIk - 1})}{\ \sum_{\nIi = 0}^{\nIk - 2} (\ithComp{b}<\nIi> \ominus_{\ccp} \ithComp{a}<\nIi>) \word(\ccp^{\nIi})}\\
&= \mBigDist{\word(\ccp^{\nIk - 1})}{\frac{1}{(\ithComp{a}<\nIk - 1> \ominus_{\ccp} \ithComp{b}<\nIk - 1>)} \sum_{\nIi = 0}^{\nIk - 2} (\ithComp{b}<\nIi> \ominus_{\ccp} \ithComp{a}<\nIi>) \word(\ccp^{\nIi})} \ge \ccd.
\end{align*}
Therefore \(\mDist{\word+(\cca)}{\word(\ccb)} \ge \ccd\).
Since \(\word+(\cca) \neq \word(\cca)\), it follows from Lemma \ref{org6f41f57} that \(\word(\cca) <_{\ccF} \word+(\cca)\).
In particular,
\begin{equation}
\label{equ:lem-linearity-word-bword-M}
 \word[\ccN](\cca) < \word+[\ccN](\cca).
\end{equation}

We now show \eqref{equ:lem-linearity}.
Suppose that \(\ccb \neq 0\) and \(\ithComp{\ccb}[\nIi] \le \ithComp{\cca}[\nIi]\) for \(\nIi \in \NN\).
If \(\ccb = \cca\), then \eqref{equ:lem-linearity} follows from \eqref{equ:lem-linearity-word-bword-M}.
Suppose that \(\ccb < \cca\).
We claim that
\begin{equation}
\label{equ:lem-linearity-inequality2}
 \word+(\cca) - \word(\ccb) < \word(\cca) - \word(\ccb). 
\end{equation}
Assume \eqref{equ:lem-linearity-inequality2} for the moment.
By the definition of \(\ccN\), for \(\cci > \ccN\),
\[
 \word+[\cci](\cca) - \word[\cci](\ccb) = \word[\cci](\cca) - \word[\cci](\ccb).
\]
By \eqref{equ:lem-linearity-inequality2},
\[
 \word+[\ccN](\cca) - \word[\ccN](\ccb) < \word[\ccN](\cca) - \word[\ccN](\ccb).
\]
It follows from \eqref{equ:lem-linearity-word-bword-M} that
\[
 \word+[\ccN](\ccb) = \word[\ccN](\ccb) \neq 0.
\]

We finally show \eqref{equ:lem-linearity-inequality2}.
Since \(\ithComp{\ccb}[\nIi] \le \ithComp{\cca}[\nIi]\), we see that
\(a \ominus_{\ccp} b = a - b < a\). Hence
\[
 \word+(\cca) - \word(\ccb) = \word+(\cca) - \word+(\ccb) = 
 \word+(\cca \ominus_{\ccp} \ccb) = \word(\cca \ominus_{\ccp} \ccb).
\]
Hence to prove \eqref{equ:lem-linearity-inequality2},
it suffices to show that,
for \(c < a \ominus_{\ccp} b = a - b\), 
\[
 \ccd(\word(\cca) - \word(\ccb), \word(\ccc)) \ge \ccd.
\]
Since \(c \oplus b \le c + b < a\), we see that
\[
  \word(\ccc \oplus \ccb) = \word(\ccc) + \word(\ccb)
\]
and
\[
 \ccd(\word(\cca) - \word(\ccb), \word(\ccc)) = \ccd(\word(\cca), \word(\ccc) + \word(\ccb)) = \ccd(\word(\cca), \word(\ccc \oplus \ccb)) \ge \ccd.
\]
It follows from Lemma \ref{org6f41f57} 
that \eqref{equ:lem-linearity-inequality2} holds.
\end{proof}

\begin{lemma}
 \comment{Lem.}
\label{sec:org616b6c6}
\label{org5db7ec7}
For \(0 \le \nIh \le \nIk - 1\), let
\[
 \ccM_{\nIh} = \max \set{\cci \in \NN : \word[\cci](\ccp^{\nIh}) \neq 0}.
\]
Then the following assertions hold.
\begin{enumerate}
\item \(\ccM_0 < \ccM_1 < \dotsb < \ccM_{\nIk - 1}\).
\item For \(0 \le \nIh, \nIl < \nIk - 1\), \[\word-[\ccM_{\nIh}](\ccp^{\nIl}) = \delta_{\nIh, \nIl} = \begin{cases} 
 1 & \tif \nIh = \nIl, \\ 
 0 & \tif \nIh \neq \nIl.
 \end{cases}\]
\item If \(\mLex{\nIk - 1}\) is a linear code,
then \(\word(\cca) = \word+(\cca) = \word-(\cca)\) for \(0 \le \cca < \ccp^{\nIk}\).
\end{enumerate}
 
\end{lemma}

\begin{proof}
 \comment{Proof.}
\label{sec:orgb7d4359}
We show the lemma by induction on \(\nIk\).

Suppose that \(\nIk = 1\); then (1) is obvious.
By the definition of \(\ccM_0\), 
we see that \(\word[\ccM_{0}](\ccp^{0}) \neq 0\).
From the minimality of \(\word(\ccp^0)\), we see that \(\word[\ccM_{0}](\ccp^{0}) = 1\).
Thus (2) holds. This implies that 
\begin{equation}
\label{equ:lem:linearity-order-k-1}
\word+(1) < \word+(2) < \cdots < \word+(\ccp - 1).
\end{equation}
Suppose that \(\mLex{0}\) is a linear code. Then \(\mLex{0} = \set{\word+(0), \word+(1), \ldots, \word+(\ccp - 1)}\).
It follows from \eqref{equ:lem:linearity-order-k-1} that \(\word(\cca) = \word+(\cca) = \word-(\cca)\) for \(0 \le \cca < \ccp\).

Suppose that \(\nIk \ge 2\).
By the induction hypothesis, \(\ccM_0 < \ccM_1 \dotsb < \ccM_{\nIk - 2}\).
Since \(\word-(\ccp^{\nIk - 1}) >_{\ccF} \word-(\ccp^{\nIk - 2})\), we see that
\(\ccM_{\nIk - 1}\ge \ccM_{\nIk - 2}\).
Let \(0 \le \nIh < \nIk - 1\), \(\ccM = \ccM_{\nIh}\),
\(\alpha = \word-[\ccM](\ccp^{\nIk - 1})\), and \(\bdy = \word-(\ccp^{\nIk - 1}) - \alpha \word-(\ccp^{\nIh})\).
By the induction hypothesis, \(\word-[\ccM](\ccp^{\nIh}) = 1\),
so \(\ithComp{\ccy}[\ccM] = 0\) and \(\bdy \le_{\ccF} \word-(\ccp^{\nIh})\).
For \(0 \le \cca < \ccp^{\nIk - 1}\),
\begin{align*}
 \ccd(\bdy, \word-(\cca)) &= \ccd\bigl(\word-(\ccp^{\nIk - 1}) - \alpha \word-(\ccp^{\nIh}),\ \word-(\cca) \bigr)\\
&= \ccd\bigl(\word-(\ccp^{\nIk - 1}),\ \word-(\cca) + \alpha \word-(\ccp^{\nIh})\bigr) \\
&= \ccd \bigl(\word-(\ccp^{\nIk - 1}),\ \word-(\cca \oplus_\ccp \alpha \ccp^{\nIh})\bigr) \ge \ccd. 
\end{align*}
Lemma \ref{org6f41f57} implies that \(\bdy \ge_{\ccF} \word-(\ccp^{\nIk - 1})\).
Hence \(\bdy = \word-(\ccp^{\nIk - 1})\), that is, \(\word[\ccM](\ccp^{\nIk - 1}) = \alpha = 0\).
Therefore (1) and (2) hold.
From (1) and (2), if \(0 \le \ccb < \cca < \ccp^{\nIk}\), then \(\word+(\ccb) < \word+(\cca)\).
Therefore (3) holds.
\end{proof}

\begin{proof}[proof of Theorem \ref{org100bcd6}]
 \comment{Proof. [proof of Theorem \ref{org100bcd6}]}
\label{sec:orga4b7b7e}
By Lemma \ref{org5db7ec7}, we see that (1) implies (2).
It is obvious that (2) implies (1) and (3).
We show that (3) implies (2).
Assume that \(\cca\) is the smallest integer satisfying
\(\word(\cca) \neq \word+(\cca)\) and \(\cca < \ccp^{\nIk}\).
Let \(\ccN\) be as in Lemma \ref{org13a9541}
and \(\alpha_{\nIi} = \word[\ccN](\ccp^{\nIi})\).
We show that
\begin{equation}
\label{equ:lem:linearity-order-k-ge-2}
\ithComp{\ccb}<0> \alpha_0 \oplus_\ccp \cdots \oplus_\ccp \ithComp{\ccb}<\nIk - 1> \alpha_{\nIk - 1} = 0 \tforsome 0 < \ccb < \cca \text{\ satisfying \ } \ithComp{\ccb}<\nIi> \le \ithComp{\cca}<\nIi> \text{\ for \ } \nIi \in \NN,
\end{equation}
which contradicts Lemma \ref{org13a9541} since
\(\word+[\ccN](\ccb) = \ithComp{\ccb}<0> \alpha_0 \oplus_\ccp \cdots \oplus_\ccp \ithComp{\ccb}<\nIk - 1> \alpha_{\nIk - 1}\).
Consider the sequence
\[
 (\beta_i)_{i = 1, 2, \ldots, n}= (\underbrace{\alpha_0, \ldots, \alpha_0}_{\ithComp{\cca}<0>}, \underbrace{\alpha_1, \ldots, \alpha_1}_{\ithComp{\cca}<1>}, \ldots, \underbrace{\alpha_{\nIk - 1}, \ldots, \alpha_{\nIk - 1}}_{\ithComp{\cca}<\nIk - 1>})
\]
and let \(\gamma_j = \beta_1 \oplus_\ccp \beta_2 \oplus_\ccp \cdots \oplus_\ccp \beta_\ccj\). 
If \(\set{\gamma_j : 1 \le j \le \ccn} = \FF_p\), then \eqref{equ:lem:linearity-order-k-ge-2} is obvious.
Suppose that \(\set{\gamma_j : 1 \le j \le \ccn} \neq \FF_p\).
Since \(\ccn \ge \ccp\), it follows that \(\gamma_j = \gamma_h\) for some \(j < h\),
so \(\gamma_h - \gamma_j = \beta_{\ccj + 1} \oplus_\ccp \cdots \oplus_\ccp \beta_h = 0\), which yields \eqref{equ:lem:linearity-order-k-ge-2}.
Therefore \(\word+(\cca) = \word(\cca)\) for \(0 \le \cca < \ccp^{\nIk}\).
\end{proof}

\begin{remark}
 \comment{Rem.}
\label{sec:orgfa1d313}
\label{orgfd98e8a}
Let \(\ccp = 3\), and let \(\cca\) and \(\ccN\) be the same as in Lemma \ref{org13a9541}.
Theorem \ref{org100bcd6} implies that \(\cca = 3^{\nIh} + 3^{\nIl}\) for some \(\nIh\) and \(\nIl\).
Moreover, \((\word[\ccN](3^{\nIh}), \word[\ccN](3^{\nIl})) \in \set{(1,1), (2,2)}\).
Indeed, it follows from Lemma \ref{org13a9541} that \(\word[\ccN](3^{\nIh})\),
\(\word[\ccN](3^{\nIl})\),
and \(\word[\ccN](3^{\nIl}) + \word[\ccN](3^{\nIh})\) are all nonzero.
Therefore \((\word[\ccN](3^{\nIh}), \word[\ccN](3^{\nIl})) \in \set{(1,1), (2,2)}\).
 
\end{remark}

\comment{connect}
\label{sec:orgf22e1cf}
We introduce some notation before moving on to the next section.
For \(\nIk \in \NN\), let \(\marray{\nIk} = \set{0, 1, \ldots, \nIk}\).
For a \(\marray{\nIk} \times \NN\) matrix \(\mnGen\) and \(\nIi \in \NN\), let \(\mnGen<\nIi>\) denote the \(\nIi\)-th column of \(\mnGen\).
We write \(\mnGen<\nIi> \in \FF_\ccp^{\marray{\nIk} \times 1}\) or \(\mnGen<\nIi> \in \FF_\ccp^{\marray{\nIk}}\).
For \(\mvecA \in \FF_\ccp^{\marray{\nIk}}\),
let \(\Cdn[\mnGen](\mvecA)\) denote the set of column indexes \(\cci\) such that
the \(\nIi\)-th column \(\mnGen<\nIi>\) equals \(\mvecA\), that is,
\[
 \Cdn[\mnGen](\mvecA) = \set{\cci \in \NN : \mnGen<\nIi> = \mvecA}.
\]
Let \(\cdn[\mnGen](\mvecA) = \Size{\Cdn[\mnGen](\mvecA)}\).

\begin{example}
 \comment{Exm.}
\label{sec:orgd506153}
\label{org3080d68}
Let \(\ccp = 3\), \(\ccd = d\), and \(\ccF = \ccE\). Let
\[
 \mGen{1} = \begin{bmatrix} \word(3^0) \\ \word(3^1) \end{bmatrix}
= \begin{bNiceMatrix}[first-row, last-col] 0 & 1 & 2 & 3 & 4 & & \\ 1 & 1 & 1 & 0 & 0 & \cdots & 0  \\
                  2 & 1 & 0 & 1 & 0 & \cdots & 1 \end{bNiceMatrix}\ \ .
\]
Then
\[
\mGen{1}<0> = \mVec{1 \\ 2}, \quad
\mGen{1}<1> = \mVec{1 \\ 1},\quad
\mGen{1}<2> = \mVec{1 \\ 0},\quad
\mGen{1}<3> = \mVec{0 \\ 1},
\]
\[
\Cdn[\mGen{1}](\mVec{1 \\ 2}) = \set{0},\quad
\Cdn[\mGen{1}](\mVec{1 \\ 1}) = \set{1},
\]
\[
\Cdn[\mGen{1}](\mVec{1 \\ 0}) = \set{2},\quad
\Cdn[\mGen{1}](\mVec{0 \\ 1}) = \set{3},\quad
\Cdn[\mGen{1}](\mVec{0 \\ 0}) = \set{4, 5, \ldots}.
\]
 
\end{example}

\section{Griesmer bound}
\label{sec:org12f2076}
\label{org05ef9f9}
In this section, we recall the Griesmer bound.
Let
\[
 \ccg_\ccp(\cck, \ccd) = \sum_{\cci = 0}^{\nIk - 1} \Bigl \lceil \frac{\ccd}{\ccp^\cci} \Bigr \rceil
\]

\begin{theorem}[Griesmer Bound \cite{griesmer-bound-1960, solomon-algebraically-1965}] \comment{Thm. [Griesmer Bound \cite{griesmer-bound-1960, solomon-algebraically-1965}]}
\label{sec:org69366c6}
\label{org496dbdd}
If \(\cC\) is an \([\ccn, \nIk, \ccd]\) code over \(\FF_{\ccp}\), then
\(\ccn \ge \ccg_\ccp(\cck, \ccd)\).

\end{theorem}

\begin{corollary}
 \comment{Cor.}
\label{sec:org1f3f645}
\label{orgb141d90}
Let \(\cC\) be an \([\ccn, \nIk, \ccd]\) code over \(\FF_{\ccp}\) with generator matrix \(\mnGen\).
If \(\cC\) meets the Griesmer bound and \(\ccd \le \ccd' \ccp^{\nIk - 1}\), 
then \(\cdn[\mnGen](\mvecA) \le \ccd'\) for \(\mvecA \in \FF_{\ccp}^{\marray{\nIk - 1}}\).
 
\end{corollary}

\begin{proof}
 \comment{Proof.}
\label{sec:orgffccc03}
Suppose that \(\cdn[\mnGen](\mvecB) \ge \ccd' + 1\) for some \(\mvecB \in \FF_{\ccp}^{\marray{\nIk - 1}}\).
Since \(\cC\) meets the Griesmer bound, \(\mvecB \neq \allzero\).
Hence we may assume that

\smallskip
\medskip
\[
\mnGen =\begin{bNiceMatrix}
  & 1 & \dotso & 1 & \dotso &\\
  & 0 & \dotso & 0 & & \\
  & \vdots & \ddots & \vdots & \mnGen' &\\
  & 0 & \dotso & 0 &  &\\
\CodeAfter
\OverBrace[shorten,yshift=3pt]{1-2}{2-4}{d' + 1}
\end{bNiceMatrix}.
\]
The matrix \(\mnGen'\) generates an \([\ccn - \ccd' - 1, \nIk - 1, \ccd]\) code.
By the Griesmer bound, \(\ccn - \ccd' - 1 \ge \ccg_\ccp(\nIk - 1, \ccd)\).
Thus \(\ccn \ge \ccd' + 1 + \ccg_\ccp(\nIk - 1, \ccd) > \ccg_\ccp(\nIk, \ccd)\),
a contradiction.
\end{proof}

\comment{connect}
\label{sec:orga621cc1}
Let
\begin{align*}
\mVct{\nIk} &= \set{\mvecA \in \FF_\ccp^{\marray{\nIk}} : \mvecA<0> = \cdots = \mvecA<\cci - 1> = 0, \mvecA<\cci> = 1 \text{\ for some\ } \cci},\\
 \mVct*{\nIk} &= \mVct{\nIk} \cup \set{\allzero},\qquad
\mEvec{\nIk} = \mVec{0 \\ 0 \\ \svdots \\ 0 \\ 1} \in \mVct{\nIk}.
\end{align*}
For example, if \(\ccp = 3\), then
\[
 \mVct{0} = \set{[1]},\quad \mVct{1} = \set{\mVec{1\\2},\ \mVec{1\\1},\ \mVec{1\\0}, \ \mVec{0\\1}}.
\]
Note that
\[
 \size{\mVct{\nIk - 1}} = \frac{\ccp^{\nIk} - 1}{\ccp - 1}.
\]
As we have seen in Example \ref{org3080d68}, a nonzero column of \(\mGen{1}\) is in \(\mVct{1}\).
We will show this is always true in Section 4.

\begin{corollary}
 \comment{Cor.}
\label{sec:org35b185c}
\label{org09650b1}
Let \(\cC\) be an \([\ccn, \nIk, \ccp^{\nIk - 1} \ccd']\) code over \(\FF_{\ccp}\) with generator matrix \(\mnGen\).
Let \(\ccV\) be a subset of \(\FF_\ccp^{\marray{\nIk}}\) of size \(\frac{\ccp^{\nIk} - 1}{\ccp - 1}\).
Suppose that 
\begin{equation}
\label{equ:cor-griesmer-bound-size}
 \Cdn[\mnGen](\mvecA) = \emptyset \tfor \mvecA \in \FF_p^{\marray{\nIk - 1}} \setminus \ccV.
\end{equation}
Then the following two conditions are equivalent.
\begin{enumerate}
\item \(\cC\) meets the Griesmer bound.
\item \(\cdn[\mnGen](\mvecA) = \ccd'\) for \(\mvecA \in \ccV\).
\end{enumerate}
 
\end{corollary}

\begin{proof}
 \comment{Proof.}
\label{sec:orgadbb69b}
By the Griesmer bound,
\(\ccn \ge \ccg_\ccp(\nIk, \ccp^{\nIk - 1} \ccd') = \frac{\ccp^{\nIk} - 1}{\ccp - 1} \ccd'\).
Suppose that \(\cC\) meets the Griesmer bound.
Then \(\ccn = \frac{\ccp^{\nIk} - 1}{\ccp - 1} \ccd'\) and \(\cdn[\mnGen](\mvecA) \le \ccd'\) for \(\mvecA \in \ccV\) by Corollary \ref{orgb141d90}.
Since \(\size{\ccV} = \frac{\ccp^{\nIk} - 1}{\ccp - 1}\),
it follows that \(\cdn[\mnGen](\mvecA) = \ccd'\).
Conversely, suppose that \(\cdn[\mnGen](\mvecA) = \ccd'\) for \(\mvecA \in \ccV\).
Then \(\ccn = \frac{\ccp^{\nIk} - 1}{\ccp - 1} \ccd'\), so \(\cC\) meets the Griesmer bound.
\end{proof}

\comment{connect}
\label{sec:org25a14cd}
Note that
we can apply Corollary \ref{org09650b1} for \(\ccV = \mVct{\nIk - 1}\)
since \(\mVct{\nIk - 1} = \frac{\ccp^{\nIk} - 1}{\ccp - 1}\).

\section{Lexicographic codes meeting the Griesmer bound}
\label{sec:orgfa28b1b}
\label{orgb6599eb}
Fix \(\ccF = \ccF(\xi, \eta) \in \sF\) and an integer \(\ccd\) such that \(\ccd \ge 2\).
Let 
\[
 \Theta = \set{\xi, \eta}.
\]
In this section, we prove Theorem \ref{org56f1397},
which states that \(\mLex-{\nIk - 1}\) meets the Griesmer bound,  by using Corollary \ref{org09650b1}. 

Let
\[
 \mGen{\nIk}  = \mGen|\ccF, \ccd|{\nIk} = \begin{bmatrix} 
 \word(\ccp^0) \\
 \word(\ccp^1) \\
 \vdots\\
 \word(\ccp^{\nIk})
\end{bmatrix} \quad
\tand \mRes(\mGen{\nIk}) = 
 \begin{bmatrix} 
 \mRes[\ccS](\word(\ccp^0)) \\
  \mRes[\ccS](\word(\ccp^1)) \\
 \vdots\\
 \mRes[\ccS](\word(\ccp^{\nIk}))
 \end{bmatrix},
\]
where \(\ccS =  \msupp|\ccE|(\mGen{\nIk}) = \bigcup_{\nIi \in \marray{\nIk}} \msupp|\ccE|(\word(\ccp^\nIi))\).
Similarly, let \(\mGen-{\nIk} = \begin{bmatrix} \word-(\ccp^{\nIi}) \end{bmatrix}_{\nIi \in \marray{\nIk}}\) and
\(\mRes(\mGen-{\nIk}) = \begin{bmatrix} \mRes[\overline{\ccS}](\word(\ccp^{\nIi})) \end{bmatrix}_{\nIi \in \marray{\nIk}}\),
where \(\overline{\ccS} = \msupp|\ccE|(\mGen-{\nIk})\).
For \(\cca \in \NN\) and \(\mvecA \in \FF_{\ccp}^{\marray{\nIk}}\), let
\[
 \innerproduct{\cca}{\mvecA} = \bigoplusp[\nIi \in \marray{\nIk}][][\ccp] \ithComp{\cca}<\nIi> \mvecA<\nIi> \in \FF_{\ccp}.
\]
For example, if \(\ccp = 3\), \(\cca = 7 = 1 + 2 \cdot 3\), and \(\mvecA = \mVec{1 \\ 1 \\ 2} \in \FF_{\ccp}^{\marray{2}}\), then
\[
\innerproduct{\cca}{\mvecA} = 1 \cdot 1 \oplus_3 2 \cdot 1 \oplus_3 0 \cdot 2 = 0.
\]

We begin with an easy lemma to calculate distance.

\begin{lemma}
 \comment{Lem.}
\label{sec:orge9a622a}
\label{org4a0db5b}
Let \(\mnGen\) be a \(\marray{\nIk} \times \NN\) matrix. For \(\cca \in \NN\),
\begin{equation}
\label{equ:lem-simplex-distance}
\mBigDist{\bdx_{\nIk}}{\sum_{\nIi \in \marray{\nIk - 1}} \cca_{\nIi} \bdx_{\nIi}}  = 
\sum_{\mvecA \in \FF_{\ccp}^{\marray{\nIk}}, \mvecA<\nIk> \neq \innerproduct{\mvecA}{\cca}} \cdn[\mnGen](\mvecA).
\end{equation}
where \(\mnGen = \mVec{\bdx_0 \\ \vdots \\ \bdx_{\nIk}}\).
 
\end{lemma}

\begin{proof}
 \comment{Proof.}
\label{sec:orgc698c2e}
Let \(\bdy = \sum_{\nIi \in \marray{\nIk - 1}} \cca_{\nIi} \bdx_{\nIi}\).
We count the number of coordinates in which \(\bdx_{\nIk}\) and \(\bdy\) agree.
Let \(\mvecA = \mnGen<\cci>\), that is, \(\cci \in \Cdn[\mnGen](\mvecA)\).
Then the \(\cci\)-th component of \(\bdx_{\nIk}\) is \(\mvecA<\nIk>\)
and that of \(\bdy\) is \(\innerproduct{\cca}{\mvecA}\).
Therefore \eqref{equ:lem-simplex-distance} follows.
\end{proof}

\comment{connect}
\label{sec:org83c41dd}
The next lemma enables us to show \eqref{equ:cor-griesmer-bound-size} for \(\mRes(\mGen-{\nIk})\).

\begin{lemma}
 \comment{Lem.}
\label{sec:org4364dd1}
\label{orgedc6ece}
Let \(F = F(\xi, \eta) \in \sF\).
\begin{enumerate}
\item If \(\xi \in \msupp|\ccE|(\mLex-{\nIk})\), then \(\eta \in \msupp|\ccE|(\mLex-{\nIk})\).
\item If \(\eta \in \msupp|\ccE|(\mLex-{\nIk})\) and \(\xi < \eta\), then \(\xi \in \msupp|\ccE|(\mLex-{\nIk})\).
\end{enumerate}
 
\end{lemma}

\begin{proof}
 \comment{Proof.}
\label{sec:org5f45f4f}
(1) Assume that \(\eta \not \in \msupp|\ccE|(\mLex-{\nIk})\).
Let \(\nIh\) be the smallest integer with \(\xi \in \supp_{\ccE} (\word-(\ccp^{\nIh}))\). Then \(\nIh \le \nIk\).
Note that
\[
 \word-<\nIi>(\ccp^{\nIh}) = 
\begin{cases}
 \word-[\nIi](\ccp^{\nIh}) & \tif \nIi \neq \eta, \\
 \word-[\eta](\ccp^{\nIh}) + \word-[\xi](\ccp^{\nIh}) & \tif \nIi = \eta. \\
\end{cases}
\]
Since \(\eta \not \in \msupp|\ccE|(\mLex-{\nIk})\),
it follows that \(0 = \word-<\eta>(\ccp^{\nIh}) = \word-[\eta](\ccp^{\nIh}) + \word-[\xi](\ccp^{\nIh})\),
and hence \(\word-[\eta](\ccp^{\nIh}) = - \word-[\xi](\ccp^{\nIh}) \neq 0\).
Let \(\bdy = \word-(\ccp^{\nIh}) - \word-[\eta](\ccp^{\nIh}) \bdf_{\eta}\); then \(\bdy <_{\ccF} \word-(\ccp^{\nIh})\).
We show that \(\mDist{\bdy}{\word-(\cca)} \ge \ccd\) for \(0 \le \cca < \ccp^{\nIh}\).
For \(\cci \neq \eta\), we see that \(\ithComp{\ccy}<\cci> = \word-<\cci>(\ccp^{\nIh})\).
Since
\[
 \ithComp{\ccy}<\eta> = \word-[\eta](\ccp^{\nIh}) + \word-[\xi](\ccp^{\nIh}) - \word-[\eta](\ccp^{\nIh}) = \word-[\xi](\ccp^{\nIh}) \neq 0 \quad \tand
\]
\[
\word-<\eta>(\ccp^{\nIl}) = 0 \tfor 0 \le \nIl \le \nIk,
\]
it follows that
\[
 \mDist{\bdy}{\word-(\cca)} = \ccd(\word-(\ccp^{\nIh}), \word-(\cca)) + 1 > \ccd \tfor 0 \le \cca < \ccp^{\nIh}.
\]
Thus Lemma \ref{org6f41f57} implies that \(\bdy \ge_{\ccF} \word-(\ccp^{\nIh})\), which is impossible.
Therefore \(\eta \in \msupp|\ccE|(\mLex-{\nIk})\).

\medskip
\noindent
(2) Assume that \(\xi \not \in \msupp|\ccE|(\mLex-{\nIk})\).
Let \(\nIh\) be the smallest integer such that \(\eta \in \msupp|\ccE|(\word-(\ccp^{\nIh}))\);
then \(\nIh \le \nIk\) and \(0 \neq \word-<\eta>(\ccp^{\nIh}) = 
\word-[\eta](\ccp^{\nIh}) + \word-[\xi](\ccp^{\nIh}) = \word-[\eta](\ccp^{\nIh})\).
Let
\[
 \bdy = \word-(\ccp^{\nIh}) - \bdf_\eta + \bdf_\xi = \word-(\ccp^{\nIh}) + \bde_\xi.
\]
Since \(\xi < \eta\) and \(\word-[\eta](\ccp^{\nIh}) > 0\), it follows that
\(\bdy <_{\ccF} \word-(\ccp^{\nIh})\).
Moreover, since \(\ithComp{\ccy}<\xi> = \word-<\xi>(\ccp^{\nIh}) + 1 = 1\)
and \(\word-<\xi>(\ccp^{\nIl}) = 0\) for \(0 \le \nIl \le \nIk\), we see that
\(\ccd(\bdy, \word-(\cca)) \ge \ccd(\word-(\ccp^{\nIh}), \word-(\cca)) \ge \ccd\)
for \(0 \le \cca < \ccp^{\nIh}\).
Therefore \(\bdy \ge_{\ccF} \word-(\ccp^{\nIh})\), a contradiction.
\end{proof}

\begin{remark}
 \comment{Rem.}
\label{sec:orgbd67778}
\label{org04e980f}
From Lemma \ref{orgedc6ece},
we see that \(\msupp|\ccF|(\mLex-{\nIk}) \subseteq \msupp|\ccE|(\mLex-{\nIk})\).
Indeed, let \(\ccN \in \NN \setminus \msupp|\ccE|(\mLex-{\nIk})\).
If \(\ccN \neq \eta\), then 
\(\ccN \in \NN \setminus \msupp|\ccF|(\mLex-{\nIk})\).
By Lemma \ref{orgedc6ece}, \(\xi \in \NN \setminus \msupp|\ccE|(\mLex-{\nIk})\).
Hence \(\word-[\eta](\ccp^{\nIi}) = \word-<\eta>(\ccp^{\nIi}) = 0\) for \(0 \le \nIi \le \nIk\).
Therefore \(\ccN = \eta \in \NN \setminus \msupp|\ccF|(\mLex-{\nIk})\).
 
\end{remark}

\begin{lemma}
 \comment{Lem.}
\label{sec:orgb2f2edc}
\label{orgf9637b2}
If \(\mvecA \in \FF_{\ccp}^{\marray{\nIk}} \setminus \mVct*{\nIk}\),
then \(\Cdn[\mGen-{\nIk}](\mvecA) = \emptyset\). In particular,
\eqref{equ:cor-griesmer-bound-size} holds for \(\mRes(\mGen-{\nIk})\).
 
\end{lemma}

\begin{proof}
 \comment{Proof.}
\label{sec:org568f891}
Since \(\mvecA \not \in \mVct*{\nIk}\),
we see that there exists \(\nIh\) such that
\(\mvecA<\nIh> \ge 2\) and \(\mvecA<\nIl> = 0\) for \(0 \le \nIl < \nIh\).
Assume that \(\nIi \in \Cdn[\mGen-{\nIk}](\mvecA)\).
Then \(\word-<\nIi>(\ccp^{\nIl}) = \mvecA<\nIl>\) for \(0 \le \nIl \le \nIk\).

\begin{mycase}[\(\cci \neq \xi\)]
 \comment{Case. [\(\cci \neq \xi\)]}
\label{sec:org7748ad8}
Suppose that \(\word-[\cci](\ccp^{\nIh}) \neq 0\).
Let \(\bdy = \word-(\ccp^{\nIh}) - \bdf_\cci\);
then \(\bdy <_{\ccF} \word-(\ccp^{\nIh})\) and
\(\ithComp{\ccy}<\cci> = \word-<\cci>(\ccp^{\nIh}) - 1 = \mvecA<\nIh> - 1 \neq 0\).
Since \(\word-<\cci>(\ccp^{\nIl}) = \mvecA<\nIl> = 0\) for \(0 \le \nIl < \nIh\),
it follows that
\(\mDist{\bdy}{\word-(\cca}) = \mDist{\word(\ccp^{\nIh})}{\word-(\cca)} \ge \ccd\)
for  \(0 \le \cca < \ccp^{\nIh}\), which is impossible.

Suppose that \(\word-[\cci](\ccp^{\nIh}) = 0\).
Since \(\word-<\cci>(\ccp^{\nIh}) = \mvecA<\nIh> \ge 2\),
it follows that \(\cci = \eta\) and 
\(\word-<\xi>(\ccp^{\nIh}) = \word-<\eta>(\ccp^{\nIh}) \ge 2\).
Let \(\bdy = \word-(\ccp^{\nIh}) - \bdf_\xi\); then 
\(\bdy <_{\ccF} \word-(\ccp^{\nIh})\).
Since \(\eta \not \in \msupp|\ccE|(\Lex-{\nIh - 1})\),
it follows from Lemma \ref{orgedc6ece} that \(\xi \not \in \msupp|\ccE|(\mLex-{\nIh - 1})\).
Hence
\(\word-<\eta>(\ccp^{\nIl}) = \word-<\xi>(\ccp^{\nIl}) = 0\) for \(0 \le \nIl < \nIh\).
Since \(\ithComp{\ccy}<\eta> = \ithComp{\ccy}<\xi> = \word-<\xi>(\ccp^{\nIh}) - 1 \neq 0\),
it follows that \(\mDist{\bdy}{\word-(\cca)}\ge \ccd\) for \(0 \le \cca < \ccp^{\nIh}\), a contradiction.
 
\end{mycase} 

\begin{mycase}[\(\cci = \xi\)]
 \comment{Case. [\(\cci = \xi\)]}
\label{sec:org3b3317d}
Suppose that \(\xi > \eta\).
Let \(\bdy = \word-(\ccp^{\nIh}) - \bdf_\xi + \bdf_\eta = \word(\ccp^{\nIh}) - \bde_\xi\);
then \(\bdy <_{\ccF} \word-(\ccp^{\nIh})\).
Since \(\word-<\xi>(\ccp^{\nIl}) = 0\) for \(0 \le \nIl < \nIh\),
it follows that
\(\mDist{\bdy}{\word-(\cca)} = \mDist{\word-(\ccp^{\nIh})}{\word-(\cca)} \ge \ccd\)
for \(0 \le \cca < \nIh\), which is impossible.

Suppose that \(\xi < \eta\).
Let
\(\bdy = \word-(\ccp^{\nIh}) - \bdf_\xi - \word-[\eta](\ccp^{\nIh})\bdf_\eta\);
then \(\bdy <_{\ccF} \word-(\ccp^{\nIh})\) and
\(\ithComp{\ccy}<\eta> = \ithComp{\ccy}<\xi> = \word-<\xi>(\ccp^{\nIh}) - 1 = 
\mvecA<\nIh> - 1 \neq 0\).
Since \(\xi \not \in \msupp|\ccE|(\mLex-{\nIh - 1})\) and \(\xi < \eta\),
it follows Lemma \ref{orgedc6ece} that \(\eta \not \in \msupp|\ccE|(\mLex-{\nIh - 1})\).
Therefore \(\mDist{\bdy}{\word-(\cca)} \ge \ccd\) for \(0 \le \cca < \nIh\), a contradiction.\qedhere
 
\end{mycase} 
\end{proof}

\comment{connect}
\label{sec:orgd98ef07}
For \(\mvecA \in \FF_{\ccp}^{\marray{\nIk - 1}}\) and \(\alpha \in \FF_\ccp\),
let 
\begin{equation}
\mvecA[\alpha] = \begin{bmatrix}\mvecA \\ \alpha\end{bmatrix} \in \FF_{\ccp}^{\marray{\nIk}}.
\end{equation}
For example,
if \(\mvecA = \mVec{0 \\ 2}\) and \(\alpha = 1\), then \(\mvecA[\alpha] = \mVec{0 \\ 2 \\ 1}\).
Similarly, for a \(\marray{\nIk - 1} \times \NN\) matrix \(\mnGen\) and \(\bdy \in \FF_\ccp^{\NN}\),
let 
\begin{equation}
\mnGen[\bdy] = \begin{bmatrix} \mnGen \\ \bdy\end{bmatrix} \in \FF_{\ccp}^{\marray{\nIk} \times \NN}.
\end{equation}

\begin{lemma}[Order lemma]
 \comment{Lem. [Order lemma]}
\label{sec:org83d1b27}
\label{orgfc26eca}
Let \(\ccM \in \Cdn[\mGen{\nIk}](\mvecA[\alpha]) \setminus \set{\xi}\) and
\(\ccN \in \Cdn[\mGen{\nIk}](\mvecA[\beta]) \setminus \set{\xi}\),
where \(\mvecA \in \FF_\ccp^{\marray{\nIk - 1}}\), and  \(\alpha\), \(\beta \in \FF_\ccp\).
If \(\alpha < \beta\) and
\(\word-[\ccN](\ccp^{\nIk}) = \word-<\ccN>(\ccp^{\nIk})\) \((= \beta)\), then
\(\ccM > \ccN\).
 
\end{lemma}

\begin{proof}
 \comment{Proof.}
\label{sec:org051f218}
Let
\begin{align*}
 \bdy &= \word-(\ccp^{\nIk}) + (\beta - \alpha) \bdf_\ccM + (\alpha - \beta) \bdf_\ccN\\
 &= \word-(\ccp^{\nIk}) + (\beta - \alpha) \bde_\ccM + (\alpha - \beta) \bde_\ccN.
\end{align*}
Note that if \(\ccM < \ccN\), then \(\bdy <_{\ccF} \word-(\ccp^{\nIk})\)
since \(\ithComp{\ccy}[\ccN] = \alpha < \beta = \word-[\ccn](\ccp^{\nIk})\).
By the definition of \(\bdy\),
\[
\Cdn[\mGen-{\nIk - 1}[\bdy]](\mvecA[\gamma]) = 
\begin{cases}
\Cdn[\mGen-{\nIk}](\mvecA[\alpha]) \cup \set{\ccN} \setminus \set{\ccM} & \tif \gamma = \alpha\\
\Cdn[\mGen-{\nIk}](\mvecA[\beta]) \cup \set{\ccM} \setminus \set{\ccN} & \tif \gamma = \beta\\
\Cdn[\mGen-{\nIk}](\mvecA[\gamma]) & \totherwise.
\end{cases}
\]
In particular, \(\cdn[\mGen-{\nIk - 1}[\bdy]](\mvecA[\gamma]) = \cdn[\mGen-{\nIk}](\mvecA[\gamma])\).
It follows from Lemma \ref{org4a0db5b} that 
\(\mDist{\bdy}{\word-(\cca)} = \mDist{\word-(\ccp^{\nIk})}{\word-(\cca)} \ge \ccd\)
for \(0 \le \cca < \ccp^{\nIk}\). Therefore \(\ccM > \ccN\).
\end{proof}

\comment{connect}
\label{sec:org4820bc3}
As we have seen in the proof of Lemma \ref{orgfc26eca},
the distance \(\ccd(\bdy, \word*(\cca))\)
is determined by \(\Cdn[\mGen-{\nIk - 1}[\bdy]](\mvecA[\gamma])\).
The following lemma ensures that there exists a suitable \(\bdy\).

\begin{lemma}
 \comment{Lem.}
\label{sec:orgc37c30b}
\label{org2f35cc7}
Let \(\mnGen\) be a \(\marray{\nIk} \times \NN\) matrix
and \(\mvecA \in \FF_{\ccp}^{\marray{\nIk - 1}}\).
Let \(\ccm_0, \ccm_1, \dotsc, \ccm_{\ccp - 1}\) be integers such that \(\sum_{\alpha \in \FF_\ccp} \ccm_\alpha = \cdn[\mnGen](\mvecA)\) and
\(\ccm_\alpha \ge 2\).
Then, for \(\bdy \in \FF_p^{\NN}\), there exists \(\tilde{\bdy} \in \FF_p^{\NN}\) satisfying the following two conditions:

\begin{enumerate}
\item \(\cdn[\mnGen[\tilde{\bdy}]](\mvecA[\alpha]) = \ccm_{\alpha}\) for \(\alpha \in \FF_p\).
\item \(\msupp|\ccE|(\bdy - \tilde{\bdy}) = \msupp|\ccF|(\bdy - \tilde{\bdy}) \subseteq \Cdn[\mnGen](\mvecA) \setminus \Theta\). 
In particular, \(\Cdn[\mnGen[\tilde{\bdy}]](\mvecB[\alpha]) = \Cdn[\mnGen[\bdy]](\mvecB[\alpha])\) for \(\alpha \in \FF_p\) and \(\mvecB \in \FF_p^{\marray{\nIk - 1}} \setminus \set{\mvecA}\).
\end{enumerate}
 
\end{lemma}

\begin{proof}
 \comment{Proof.}
\label{sec:org3a80700}
Let \(\ccl_\alpha = \cdn[\mnGen[\bdy]](\mvecA[\alpha])\)
and \(\ccD = \sum_{\alpha \in \ccF_p} \abs{\ccl_{\alpha} - \ccm_{\alpha}}\).
We show the lemma by induction on \(\ccD\).
If \(\ccD = 0\), then \(\bdy\) itself satisfies the two conditions.
Suppose that \(\ccD \ge 1\).
Since \(\sum_{\alpha \in \FF_\ccp} \ccl_\alpha = \cdn[\mnGen](\mvecA) = \sum_{\alpha \in \FF_\ccp} \ccm_\alpha\),
it follows that there exist \(\beta\), \(\gamma \in \FF_\ccp\) such that
\(\ccl_\beta > \ccm_{\beta}\) and \(\ccl_\gamma < \ccm_{\gamma}\).
Since \(\ccl_\beta > \ccm_{\beta} \ge 2\), the set \(\Cdn[\mnGen[\bdy]](\mvecA[\beta]) \setminus \Theta\) 
contains an element \(\ccN\).
Let \(\bdz = \bdy + (\gamma - \beta) \bdf_{\ccN} = \bdy + (\gamma - \beta) \bde_{\ccN}\)
and \(\ccl'_\alpha = \cdn[\mnGen[\bdz]](\mvecA[\alpha])\). We see that 
\[
 \ccl'_\alpha = \begin{cases}
 \ccl_\beta - 1 & \tif \alpha = \beta,\\
 \ccl_\gamma + 1 & \tif \alpha = \gamma,\\
 \ccl_\alpha & \totherwise.\\
 \end{cases}
\]
By the induction hypothesis, there exists \(\tilde{\bdz} \in \FF_{\ccp}^{\NN}\) satisfying (1) and (2) for \(\bdz\), that is,
\[
 \cdn[\mnGen[\tilde{\bdz}]](\mvecA[\alpha]) = \ccm_{\alpha} \tfor \alpha \in \FF_p
\]
and
\[
\supp_{\ccE} (\bdz - \tilde{\bdz}) = \supp_{\ccF} (\bdz - \tilde{\bdz}) \subseteq \Cdn[\mnGen](\mvecA) \setminus \Theta.
\]
Since
\[
 \supp_{\ccF} (\bdy - \tilde{\bdz}) = \supp_{\ccF} (\bdz + (\alpha - \beta) \bdf_{\ccN} - \tilde{\bdz}) \subseteq \supp_{\ccF} (\bdz - \tilde{\bdz}) \cup \set{\cci} \subseteq \Cdn[\mnGen](\mvecA) \setminus \Theta,
\]
it follows that \(\supp_{\ccF} (\bdy - \tilde{\bdz}) = \supp_{\ccE} (\bdy - \tilde{\bdz})\),
and that \(\tilde{\bdz}\) satisfies (1) and (2) also for \(\bdy\).
\end{proof}

\begin{example}
 \comment{Exm.}
\label{sec:org752686d}
Let \(\ccp = 3\), \(\ccd = 7\), and \(\ccF = \ccF(1, 3)\).
Then
\[\word(1) = \mWord{1 1 1 1 1 1 1 0 0 \cdots}
\]
and \(\cdn[\mGen{0}](1) = 7\).
Consider \(\ccm_0 = 3\), \(\ccm_1 = 2\), \(\ccm_2 = 2\),
and \(\bdy = \mWord{2 2 2 2 2 1 1 0 0 \cdots}\).
Since \(\mGen{0}[\bdy] = \begin{bNiceMatrix}[small, light-syntax, first-row]0 1 2 3 4 5 6 7 8 \cdots; 1 1 1 1 1 1 1 0 0 \cdots; 2 2 2 2 2 1 1 0 0 \cdots\end{bNiceMatrix}\), 
we see that
\[\Cdn[\mGen{0}[\bdy]](\mVec{1 \\ 0}) = \emptyset,\ \Cdn[\mGen{0}[\bdy]](\mVec{1 \\ 1}) = \set{5, 6},\ \Cdn[\mGen{0}[\bdy]](\mVec{1 \\ 2}) = \set{0,1,2,3,4}.
\]
Let \(\tilde{\bdy} = \mWord{0 2 0 2 0 1 1 0 0 \cdots}\). Then 
\(\mGen{0}[\tilde{\bdy}] = \begin{bNiceMatrix}[small, light-syntax, first-row]0 1 2 3 4 5 6 7 8 \cdots; 1 1 1 1 1 1 1 0 0 \cdots; 0 2 0 2 0 1 1 0 0 \cdots\end{bNiceMatrix}\), so
\[\Cdn[\mGen{0}[\tilde{\bdy}]](\mVec{1 \\ 0}) = \set{0, 2, 4},\
\Cdn[\mGen{0}[\tilde{\bdy}]](\mVec{1 \\ 1}) = \set{5, 6},\
\Cdn[\mGen{0}[\tilde{\bdy}]](\mVec{1 \\ 2}) = \set{1, 4}.
\]
Moreover \(\bdy - \tilde{\bdy} = 2 \bde_0 + 2 \bde_2 + 2\bde_4 = 2 \bdf_0 + 2 \bdf_2 + 2 \bdf_4\).
Thus \(\tilde{\bdy}\) satisfies the two conditions (1) and (2).
 
\end{example}

\begin{proof}[Proof of Theorem \ref{org56f1397}]
 \comment{Proof. [Proof of Theorem \ref{org56f1397}]}
\label{sec:org31a644b}
Let \(-1 \le \nIh < \nIk\).
We prove the theorem by induction on \(\nIh\).
Since the length of \(\mRes(\mLex-{\nIh})\) equals
\(\sum_{\mvecA \in \mVct{\nIh}} \cdn[\mGen-{\nIh}](\mvecA)\)
and \(\ccg_\ccp(\nIh, \ccd) = \ccd' (\ccp^{\nIk - 1} + \ccp^{\nIk - 2} + \cdots + \ccp^{\nIk - \nIh - 1})\),
it suffices to show that
\begin{align*}
\sum_{\mvecA \in \mVct{\nIh}} \cdn[\mGen-{\nIh}](\mvecA)
 = \ccd' (\ccp^{\nIk - 1} + \ccp^{\nIk - 2} + \cdots + \ccp^{\nIk - \nIh - 1}).
\end{align*}
If \(\nIh = -1\), then both sides equal zero.
Suppose that \(\nIh \ge 0\).
By the induction hypothesis,
\begin{align*}
\sum_{\mvecA \in \mVct{\nIh - 1}} \cdn[\mGen-{\nIh - 1}](\mvecA)
 = \ccd' (\ccp^{\nIk - 1} + \ccp^{\nIk - 2} + \cdots + \ccp^{\nIk - \nIh}).
\end{align*}
Hence
\begin{align*}
\sum_{\mvecA \in \mVct{\nIh}} \cdn[\mGen-{\nIh}](\mvecA)
&= \sum_{\mvecA \in \mVct{\nIh - 1}} \cdn[\mGen-{\nIh - 1}](\mvecA)
 + \cdn[\mGen-{\nIh}](\mEvec{\nIh}) \\
 &= \ccd' (\ccp^{\nIk - 1} + \ccp^{\nIk - 2} + \cdots + \ccp^{\nIk - \nIh})
 + \cdn[\mGen-{\nIh}](\mEvec{\nIh}).
\end{align*}
Assume that \(\cdn[\mGen-{\nIh}](\mEvec{\nIh}) > \ccd' \ccp^{\nIk - \nIh - 1}\).
Since \(\ccd' \ge 2\), 
there exists \(\ccM \in \Cdn[\mGen-{\nIh}](\mEvec{\nIh}) \setminus \Theta\).
Let
\[
 \bdy = \word-(\ccp^{\nIh}) - \bdf_\ccM = \word-(\ccp^{\nIh}) - \bde_\ccM.
\]
Then
\(\cdn[\mGen-{\nIh - 1}[\bdy]](\mEvec{\nIh}) \ge \ccd' \ccp^{\nIk - \nIh - 1}\) and
\(\bdy <_{\ccF} \word-(\ccp^{\nIh})\) because \(\word-[\ccM](\ccp^{\nIh}) = \word-<\ccM>(\ccp^{\nIh}) = 1\).
Since \(\mRes(\mLex-{\nIh - 1})\) meets the Griesmer bound,
it follows from Lemma \ref{orgf9637b2} and Corollary \ref{org09650b1} that
\[
\cdn[\mGen-{\nIh - 1}](\mvecA) = \frac{\ccd' \ccp^{\nIk - 1}}{\ccp^{\nIh - 1}} = \ccd' \ccp^{\nIk - \nIh} \tfor \mvecA \in \mVct{\nIh - 1}.
\]
Note that
\[
 \ccd' \ccp^{\nIk - \nIh} = \underbrace{\ccd' \ccp^{\nIk - \nIh - 1} + \cdots + \ccd' \ccp^{\nIk - \nIh - 1}}_{\ccp}.
\]
Lemma \ref{org2f35cc7} shows that there exists \(\tilde{\bdy}\) such that
\begin{equation}
\label{equ:sec4-proof-1}
\cdn[\mGen-{\nIh - 1}[\tilde{\bdy}]](\mvecA[\alpha]) = \ccd' \ccp^{\nIk - \nIh - 1} \tfor \mvecA \in \mVct{\nIh - 1},\ \alpha \in \FF_\ccp
\end{equation}
and
\[\msupp|\ccF|(\bdy - \tilde{\bdy}) = \msupp|\ccE|(\bdy - \tilde{\bdy}) \subseteq \bigcup_{\mvecA \in \mVct{\nIh - 1}} \Cdn[\mGen-{\nIh - 1}](\mvecA) \setminus \Theta.\]
Since \(\cdn[\mGen-{\nIh - 1}[\tilde{\bdy}]](\mEvec{\nIh}) = \cdn[\mGen-{\nIh - 1}[\bdy]](\mEvec{\nIh}) \ge \ccd' \ccp^{\nIk - \nIh - 1}\)
and \(\cdn[\mGen-{\nIh - 1}[\tilde{\bdy}]](\mvecB) = 0\) for \(\mvecB \not \in \mVct{\nIh}\),
it follows from \eqref{equ:sec4-proof-1} that
\(\ccd(\tilde{\bdy}, \word-(\cca)) \ge \ccd\) for \(0 \le \cca < \ccp^{\nIh}\),
so \(\tilde{\bdy} \ge_{\ccF} \word-(\ccp^{\nIh})\).
From Lemma \ref{orgfc26eca}, if \(\ccN \in \Cdn[\mGen-{\nIh - 1}](\mvecA) \setminus \Theta\), then
\(\ccN < \ccM\). Thus \(\max \msupp|\ccF|(\tilde{\bdy} - \word-(\ccp^{\nIh})) = \ccM\), and hence \(\tilde{\bdy} <_{\ccF} \word-(\ccp^{\nIh})\),
a contradiction.
\end{proof}

\section{Ternary linear lexicographic codes}
\label{sec:org5aab263}
\label{orgd92423c}

In this section, we show that \(\mLex{\nIk - 1}\) is a linear code when \(\ccp = 3\), \(\ccd = \ccp^{\nIk - 1} \ccd'\), and \(\ccd' \ge 2\).

\subsection{Stronger versions of lemmas in Section \texorpdfstring{\ref{orgb6599eb}}{4}}
\label{sec:org57d50b6}

\begin{lemma}
 \comment{Lem.}
\label{sec:org49c5936}
\label{org45cd32d}
Let \(\nIh \in \NN\).
Suppose that \(\cdn[\mGen-{\nIh}](\mvecA) \ge 2\) for \(\mvecA \in \mVct{\nIh}\).

\begin{enumerate}
\item If \(\word-<\xi>(\ccp^{\nIh}) \neq 0\), then \(\word-<\eta>(\ccp^{\nIh}) \neq 0\).
\item If \(\xi < \eta\), then \(\word-<\xi>(\ccp^{\nIh}) = \word-<\eta>(\ccp^{\nIh})\); in particular, \(\word-[\eta](\ccp^{\nIh}) = 0\).
\item If \(\xi > \eta\) and \(\word-<\xi>(\ccp^{\nIh}) \neq \word-<\eta>(\ccp^{\nIh})\), then \(\word-<\xi>(\ccp^{\nIh}) < \word-<\eta>(\ccp^{\nIh})\); 
in particular, \(\word-[\eta](\ccp^{\nIh}) \neq \ccp - 1\) when \(\word-<\xi>(\ccp^{\nIh}) \neq 0\).
\end{enumerate}
 
\end{lemma}

\begin{proof}
 \comment{Proof.}
\label{sec:orgc093882}

\noindent
(1) We first show (1) assuming (2) and (3).
Suppose that \(\word-<\xi>(\ccp^{\nIh}) \neq 0\).
Since \(\word-<\eta>(\ccp^{\nIh}) = \word-[\eta](\ccp^{\nIh}) + \word-[\xi](\ccp^{\nIh})\),
we may assume that \(\word-[\eta](\ccp^{\nIh}) \neq 0\).
It follows from (2) and (3) that \(\xi > \eta\) and \(\word-<\xi>(\ccp^{\nIh}) < \word-<\eta>(\ccp^{\nIh})\).
Thus \(\word-<\eta>(\ccp^{\nIh}) \neq 0\).

\medskip
\noindent
(2) 
Suppose that \(\xi < \eta\).
If \(\xi \not \in \msupp|\ccE|(\mLex-{\nIh})\),
then it follows from Lemma \ref{orgedc6ece} that
\(\eta \not \in \msupp|\ccE|(\mLex-{\nIh})\),
and hence \(\word-<\eta>(\ccp^{\nIh}) = \word-<\xi>(\ccp^{\nIh}) = 0\).

Suppose that \(\xi \in \msupp|\ccE|(\mLex-{\nIh})\).
Lemma \ref{orgedc6ece} yields \(\eta \in \msupp|\ccE|(\mLex-{\nIh})\).
Let \(\mvecA[\alpha] = \mGen-{\nIh}<\xi>\), that is, \(\xi \in \Cdn[\mGen-{\nIh}](\mvecA[\alpha])\).
We show that \(\eta \in \Cdn[\mGen-{\nIh}](\mvecA[\alpha])\) by induction on \(\nIh\).

If \(\nIh = 0\), then 
it follows from Lemma \ref{orgf9637b2} that 
\(\mvecA[\alpha] = 1\) and \(\eta \in \Cdn[\mGen-{0}](1)\).

Suppose that \(\nIh \ge 1\).
By the induction hypothesis, \(\word-<\xi>(\ccp^\nIl) = \word-<\eta>(\ccp^\nIl)\) for \(0 \le \nIl < \nIh\),
and hence \(\mGen-{\nIh - 1}<\eta> = \mGen-{\nIh - 1}<\xi> = \mvecA\). 
Thus \(\eta \in \Cdn[\mGen-{\nIh - 1}](\mvecA)\). 
Let \(\beta = \word-<\eta>(\ccp^\nIh)\), that is,  \(\eta \in \Cdn[\mGen-{\nIh}](\mvecA[\beta])\). 
Assume that \(\beta \neq \alpha\) (\(= \word-<\xi>(\ccp^{\nIh})\)).
Since \(\beta = \alpha + \word-[\eta](\ccp^{\nIh})\),
we see that \(\word-[\eta](\ccp^{\nIh}) \neq 0\).
Let \(\gamma = \max \set{\alpha, \beta}\) and \(\epsilon = \min \set{\alpha, \beta}\).
Since \(\cdn[\mGen-{\nIh}](\mvecA[\gamma]) \ge 2\)
and \(\Size{\Cdn[\mGen-{\nIh}](\mvecA[\gamma]) \cap \Theta} = 1\),
there exists \(\ccN \in \Cdn[\mGen-{\nIh}](\mvecA[\gamma]) \setminus \Theta\).
Let
\[
 \bdy = \word-(\ccp^{\nIh}) + (\gamma - \alpha) \bdf_\xi + (\alpha - \beta) \bdf_\eta + (\epsilon - \gamma) \bdf_\ccN.
\]
\begin{center}
\begin{tabular}{c|ccc}
 & \(\xi\) & \(\eta\) & \(\ccN\)\\
\hline
\(\word-[\cci](\ccp^{\nIh})\) & \(\alpha\) & \(\beta - \alpha\) & \(\gamma\)\\
\(\word-<\cci>(\ccp^{\nIh})\) & \(\alpha\) & \(\beta\) & \(\gamma\)\\
\(\ithComp{\ccy}[\cci]\) & \(\gamma\) & \(0\) & \(\epsilon\)\\
\(\ithComp{\ccy}<\cci>\) & \(\gamma\) & \(\gamma\) & \(\epsilon\)\\
\end{tabular}
\end{center}
\noindent
Since \(\xi < \eta\) and \(\epsilon < \gamma\), it follows that \(\bdy <_{\ccF} \word-(\ccp^\nIh)\).
Because \(\set{\gamma, \gamma, \epsilon}\) and \(\set{\alpha, \beta, \gamma}\) are equal as multisets, 
we see that \(\cdn[\mGen-{\nIh - 1}[\bdy]](\mvecB) = \cdn[\mGen-{\nIh}](\mvecB)\) for \(\mvecB \in \mVct{\nIh}\).
It follows from Lemma \ref{org4a0db5b} that \(\mbigDist{\bdy}{\word-(\cca)} = \mbigDist{\word-(\ccp^{\nIh})}{\word-(\cca)} \ge \ccd\).
This implies that \(\bdy >_{\ccF} \word-(\ccp^{\nIh})\), a contradiction.

\medskip
\noindent
(3) If \(\xi \not \in \msupp|\ccE|(\mLex-{\nIh})\), then the assertion is obvious.
Suppose that \(\xi \in \msupp|\ccE|(\mLex-{\nIh})\).
Lemma \ref{orgedc6ece} shows that \(\eta \in \msupp|\ccE|(\mLex-{\nIh})\).
Let \(\mvecA[\alpha] = \mGen-{\nIh}<\xi>\) and \(\mvecB[\beta] = \mGen-{\nIh}<\eta>\),
that is, \(\xi \in \Cdn[\mGen-{\nIh}](\mvecA[\alpha])\) and \(\eta \in \Cdn[\mGen-{\nIh}](\mvecB[\beta])\).
Assume that \(\mvecB = \allzero\). Lemma \ref{orgedc6ece} implies that \(\mvecA = \allzero\).
Since \(\xi, \eta \in \msupp|\ccE|(\mLex-{\nIh})\),
it follows from Lemma \ref{orgf9637b2} that \(\alpha = \beta = 1\), a contradiction.
Thus \(\mvecB \neq \allzero\), so \(\mvecB \in \mVct{\nIh - 1}\) and \(\mvecB[\alpha] \in \mVct{\nIh}\).
Since \(\cdn[\mGen-{\nIh}](\mvecB[\alpha]) \ge 2\),
there exists \(\ccM \in \Cdn[\mGen-{\nIh}](\mvecB[\alpha]) \setminus \Theta\).
Let
\[
\bdy = \word-(\ccp^{\nIh}) + (\alpha - \beta) \bdf_\eta + (\beta - \alpha) \bdf_\ccM.
\]
\begin{center}
\begin{tabular}{c|ccc}
 & \(\eta\) & \(\ccM\) & \(\xi\)\\
\hline
\(\word-[\cci](\ccp^{\nIh})\) & \(\beta - \alpha\) & \(\alpha\) & \(\alpha\)\\
\(\word-<\cci>(\ccp^{\nIh})\) & \(\beta\) & \(\alpha\) & \(\alpha\)\\
\(\ithComp{\ccy}[\cci]\) & \(0\) & \(\beta\) & \(\alpha\)\\
\(\ithComp{\ccy}<\cci>\) & \(\alpha\) & \(\beta\) & \(\alpha\)\\
\end{tabular}
\end{center}
\noindent
It follows from Lemma \ref{org4a0db5b} that 
\(\mbigDist{\bdy}{\word-(\cca)} = \mbigDist{\word-(\ccp^{\nIh}}{\word-(\cca)} \ge \ccd\)
for \(0 \le \cca < \ccp^{\nIh}\),
and hence  \(\bdy >_{\ccF} \word-(\ccp^{\nIh})\). Therefore \(\ccM > \eta\) and \(\beta > \alpha\).
\end{proof}

\begin{lemma}
 \comment{Lem.}
\label{sec:orge8cce5f}
\label{org06db361}
Let \(\nIh \in \NN\).
Suppose that \(\cdn[\mGen-{\nIh}](\mvecB) \ge 2\) for \(\mvecB \in \mVct{\nIh}\).
Let \(\ccM \in \Cdn[\mGen-{\nIh}](\mvecA[\alpha])\) and
\(\ccN \in \Cdn[\mGen-{\nIh}](\mvecA[\beta]) \setminus \set{\xi}\).
If \(\alpha < \beta\) and \(\word-[\ccN](\ccp^{\nIh}) = \word-<\ccN>(\ccp^{\nIh})\),
then \(\ccM > \ccN\).
 
\end{lemma}

\begin{proof}
 \comment{Proof.}
\label{sec:orgfba106d}
When \(\ccM \neq \xi\), the lemma follows from Lemma \ref{orgfc26eca}.
Suppose that \(\ccM = \xi\).

\resetmycase
\begin{mycase}[\(\eta < \xi\)]
 \comment{Case. [\(\eta < \xi\)]}
\label{sec:org7317fbb}
If \(\ccN = \eta\), then \(\ccN = \eta < \xi = \ccM\),
and hence we may assume that \(\ccN \neq \eta\).
Let 
\[
 \bdy = \word-(\ccp^\nIh) + (\beta -\alpha) \bdf_\xi + 
(\alpha - \beta) \bdf_\eta + (\alpha - \beta) \bdf_\ccN
\]
It follows from Lemma \ref{org4a0db5b} that \(\mbigDist{\bdy}{\word(\cca)} \ge \ccd\) for \(0 \le \cca < \ccp^{\nIh}\),
so \(\bdy >_{\ccF} \word(\ccp^{\nIh})\). Therefore \(\ccN < \xi = \ccM\).

\begin{center}
\begin{tabular}{llll}
 & \(\ccN\) & \(\eta\) & \(\xi\)\\
\hline
\(\word-[\cci](\ccp^{\nIh})\) & \(\beta\) & \(\gamma\) & \(\alpha\)\\
\(\word-<\cci>(\ccp^{\nIh})\) & \(\beta\) & \(\gamma + \alpha\) & \(\alpha\)\\
\(\ithComp{\ccy}[\cci]\) & \(\alpha\) & \(\gamma + \alpha - \beta\) & \(\beta\)\\
\(\ithComp{\ccy}<\cci>\) & \(\alpha\) & \(\gamma + \alpha\) & \(\beta\)\\
\end{tabular}
\end{center}
 
\end{mycase}

\begin{mycase}[\(\eta > \xi\)]
 \comment{Case. [\(\eta > \xi\)]}
\label{sec:org506ed81}
By Lemma \ref{org45cd32d}, we see that \(\mGen-{\nIh}<\eta> = \mGen-{\nIh}<\xi> = \mvecA[\alpha]\),
and hence \(\eta \in \Cdn[\mGen-{\nIh}](\mvecA[\alpha])\).
Thus \(\ccN \neq \eta\). 
Since \(\cdn[\mGen-{\nIh}](\mvecA[\beta]) \ge 2\),
there exists \(\ccL \in \Cdn[\mGen-{\nIh}](\mvecA[\beta]) \setminus \set{\ccN}\).
Note that \(\ccL \not \in \Theta\).
Let
\[\bdy = \word-(\ccp^\nIh) + (\beta - \alpha) \bdf_\xi + (\alpha - \beta) \bdf_\ccL + (\alpha - \beta) \bdf_\ccN.
\]
\begin{center}
\begin{tabular}{lllll}
 & \(\ccN\) & \(\ccL\) & \(\xi\) & \(\eta\)\\
\hline
\(\word-[\cci](\ccp^{\nIh})\) & \(\beta\) & \(\beta\) & \(\alpha\) & 0\\
\(\word-<\cci>(\ccp^{\nIh})\) & \(\beta\) & \(\beta\) & \(\alpha\) & \(\alpha\)\\
\(\ithComp{\ccy}[\cci]\) & \(\alpha\) & \(\alpha\) & \(\beta\) & 0\\
\(\ithComp{\ccy}<\cci>\) & \(\alpha\) & \(\alpha\) & \(\beta\) & \(\beta\)\\
\end{tabular}
\end{center}
\noindent
It follows from Lemma \ref{org4a0db5b} that \(\mbigDist{\bdy}{\word(\cca)} \ge \ccd\) for \(0 \le \cca < \ccp^{\nIh}\),
and hence \(\bdy >_{\ccF} \word-(\ccp^{\nIh})\).
Therefore \(\ccN < \xi = \ccM\). \hspace*{\fill} \(\qedhere\)
 
\end{mycase} 
\end{proof}

\begin{lemma}
 \comment{Lem.}
\label{sec:orgd80407f}
\label{orgb44e8dd}
Let \(\ccL < \ccj < \ccM\) and \(\nIh \in \NN \cup \set{-1}\).
Suppose that \(\cdn[\mGen-{\nIh}](\mvecA) \ge 2\) for \(\mvecA \in \mVct*{\nIh}\),
and that \(\ccL, \ccM \in \Cdn[\mGen-{\nIh}](\mvecC)\) for some \(\mvecC \in \mVct*{\nIh}\),
where \(\mVct*{-1} = \set{[\, ]}\) and \(\Cdn[\mGen-{-1}]([\, ]) = \NN\).
If one of the following two conditions holds, then \(\ccj \in \Cdn[\mGen-{\nIh}](\mvecC)\).
\begin{enumerate}
\item \(\ccj, \ccM \not \in \Theta\).
\item \(\ccj, \ccM \in \Theta\).
\end{enumerate}
 
\end{lemma}

\begin{proof}
 \comment{Proof.}
\label{sec:org833d937}
We show the lemma by induction on \(\nIh\). 
This is obvious when \(\nIh = -1\).
Suppose that \(\nIh \ge 0\),
and let \(\mvecB[\alpha] = \mvecC\),
where \(\mvecB \in \mVct*{\nIh - 1}\) and \(\alpha \in \FF_{\ccp}\).
By the induction hypothesis,
\(\ccj \in \Cdn[\mGen-{\nIh - 1}](\mvecB)\).
Let \(\beta = \word-<\ccj>(\ccp^{\nIh})\); then
\(\ccj \in \Cdn[\mGen-{\nIh}](\mvecB[\beta])\).

\noindent
(1) First, assume that \(\beta < \alpha\).
Since \(\ccj \in \Cdn[\mGen-{\nIh}](\mvecB[\beta])\) and
\(\ccM \in \Cdn[\mGen-{\nIh}](\mvecB[\alpha]) \setminus \Theta\),
it follows from Lemma \ref{org06db361} that \(\ccj > \ccM\), a contradiction.

Next, assume that \(\alpha < \beta\).
Since \(\ccL \in \Cdn[\mGen-{\nIh}](\mvecB[\alpha])\)
and \(\ccj \in \Cdn[\mGen-{\nIh}](\mvecB[\beta]) \setminus \Theta\),
it follows from Lemma \ref{org06db361} that \(\ccL > \ccj\), a contradiction.
Therefore \(\alpha = \beta\).

\noindent
(2) If \(\ccj = \xi\) and \(\ccM = \eta\),
then \(\xi < \eta\); it follows from Lemma \ref{org45cd32d} that \(\alpha = \beta\).
Suppose that \(\ccj = \eta\) and \(\ccM = \xi\).
Let
\[\bdy = \word-(\ccp^{\nIh}) + (\alpha - \beta) \bdf_\ccj +  (\beta - \alpha) \bdf_{\ccL}.\]
\begin{center}
\begin{tabular}{llll}
 & \(\ccL\) & \(\quad\ccj\quad\) & \(\ccM\)\\
\hline
\(\word-[\cci](\ccp^{\nIh})\) & \(\alpha\) & \(\beta - \alpha\) & \(\alpha\)\\
\(\word-<\cci>(\ccp^{\nIh})\) & \(\alpha\) & \(\beta\) & \(\alpha\)\\
\(\ithComp{\ccy}[\cci]\) & \(\beta\) & \(0\) & \(\alpha\)\\
\(\ithComp{\ccy}<\cci>\) & \(\beta\) & \(\alpha\) & \(\alpha\)\\
\end{tabular}
\end{center}
Note that \(\bdy \le_{\ccF} \word-(\ccp^{\nIh})\).
Since \(\mDist{\bdy}{\word-(\cca)} \ge \ccd\) for \(0 \le \cca < \ccp^{\nIh}\),
it follows that \(\bdy \ge_{\ccF} \word-(\ccp^{\nIh})\). Therefore \(\alpha = \beta\).
\end{proof}

\comment{connect}
\label{sec:orgaf81c86}
As seen in Lemma \ref{org06db361}, if \(\ccM \in \Cdn[\mGen-{\nIh}](\mvecA[\alpha])\),
\(\ccN \in \Cdn[\mGen-{\nIh}](\mvecA[\beta])\), and \(\alpha < \beta\), then \(\ccM > \ccN\) in most cases.
The following lemma provides necessary conditions for \(\ccM < \ccN\).

\begin{lemma}
 \comment{Lem.}
\label{sec:org5f81f27}
\label{org69880d5}
Let \(\nIh \in \NN\).
Suppose that 
\(\cdn[\mGen-{\nIh}](\mvecB) \ge 2\) for \(\mvecB \in \mVct{\nIh}\).
Let \(\ccM \in \Cdn[\mGen-{\nIh}](\mvecA[\alpha])\) and
\(\ccN \in \Cdn[\mGen-{\nIh}](\mvecA[\beta])\).
Suppose that \(\alpha < \beta\) and \(\ccM < \ccN\). Then the following statements hold.

\begin{enumerate}
\item \(\word-[\eta](\ccp^{\nIh}) = 0\) and \(\word<\eta>(\ccp^{\nIh}) = \word<\xi>(\ccp^{\nIh}) = \beta \neq 0\).
\item If \(\eta < \xi\), then \(\ccN = \eta\) and \(\beta = \word-<\eta>(\ccp^{\nIh}) = \word-<\xi>(\ccp^{\nIh})\).
\item If \(\xi < \eta\), then \(\ccN \in \Theta\); moreover, if \(\ccN = \xi\), then \(\alpha = \beta - 1\) and \(\ccN = \ccM + 1\).
\end{enumerate}
 
\end{lemma}

\begin{proof}
 \comment{Proof.}
\label{sec:orgf767128}
Lemma \ref{org06db361} implies that \(\ccN \in \Theta\).

\noindent
(1) We show \(\word-[\eta](\ccp^{\nIh}) = 0\) by assuming (2). 
If \(\eta < \xi\), then \(\beta = \word-<\eta>(\ccp^{\nIh}) = \word-<\xi>(\ccp^{\nIh})\) by (2), and hence \(\word-[\eta](\ccp^{\nIh}) = 0\).
Suppose that \(\xi < \eta\). Lemma \ref{org45cd32d} shows that \(\word-[\eta](\ccp^{\nIh}) = 0\).
Hence \(\word-<\eta>(\ccp^{\nIh}) = \word-<\xi>(\ccp^{\nIh}) = \word-<\ccN>(\ccp^\nIh) = \beta\) since \(\ccN \in \Theta\).

\medskip
\noindent
(2) Assume that \(\ccN = \xi\).
Let
\[
 \bdy = \word-(\ccp^{\nIh}) + (\alpha - \beta) \bdf_\xi + (\beta - \alpha) (\bdf_\eta + \bdf_\ccM).
\]
Note that \(\bdy <_{\ccF} \word-(\ccp^{\nIh})\).
Moreover, 
\(\mbigDist{\bdy}{\word-(\cca)} = \mbigDist{\word-(\ccp^{\nIh})}{\word-(\cca)} \ge \ccd\) for \(0 \le \cca < \ccp^{\nIh}\)
even if \(\ccM = \eta\). Thus \(\bdy >_{\ccF} \word-(\ccp^{\nIh})\), a contradiction.

\begin{minipage}{.45\textwidth}

\begin{center}
\begin{tabular}{@{}ccc@{}}
 & \(\quad\ccM = \eta\quad\) & \(\quad\ccN = \xi\quad\)\\
\hline
\(\word-[\cci](\ccp^{\nIh})\) & \(\alpha - \beta\) & \(\beta\)\\
\(\word-<\cci>(\ccp^{\nIh})\) & \(\alpha\) & \(\beta\)\\
\(\ithComp{\ccy}[\cci]\) & \(\beta - \alpha\) & \(\alpha\)\\
\(\ithComp{\ccy}<\cci>\) & \(\beta\) & \(\alpha\)\\
\end{tabular}
\end{center}

\end{minipage}\begin{minipage}{.51\textwidth}

\begin{center}
\begin{tabular}{@{}cccc@{}}
 & \(\quad\ccM\quad\) & \(\quad\eta\quad\) & \(\quad\ccN = \xi\quad\)\\
\hline
\(\word-[\cci](\ccp^{\nIh})\) & \(\alpha\) & \(\gamma\) & \(\beta\)\\
\(\word-<\cci>(\ccp^{\nIh})\) & \(\alpha\) & \(\gamma + \beta\) & \(\beta\)\\
\(\ithComp{\ccy}[\cci]\) & \(\beta\) & \(\gamma + \beta - \alpha\) & \(\alpha\)\\
\(\ithComp{\ccy}<\cci>\) & \(\beta\) & \(\gamma + \beta\) & \(\alpha\)\\
\end{tabular}
\end{center}

\end{minipage}

Therefore \(\ccN = \eta\).
Since \(\ccM < \ccN = \eta < \xi\),
it follows that \(\ccM \neq \xi\).
We show that there exists \(\ccL \in \Cdn[\mGen-{\nIh}](\mvecA[\beta  - 1]) \setminus \Theta\) such \(\ccL \le \ccM\).
Indeed, if \(\alpha = \beta - 1\), then \(\ccL = \ccM\) satisfies the condition.
If \(\alpha < \beta - 1\), then there exists \(\ccL \in \Cdn[\mGen-{\nIh}](\mvecA[\beta - 1]) \setminus \Theta\),
and we see that \(\ccL < \ccM\) by Lemma \ref{org06db361}.
Let \(\bdy = \word-(\ccp^{\nIh}) - \bdf_\ccN + \bdf_\ccL\).
\begin{center}
\begin{tabular}{lccc}
 & \(\quad\ccL\quad\) & \(\quad\ccN = \eta\quad\) & \(\quad\xi\quad\)\\
\hline
\(\word-[\cci](\ccp^{\nIh})\) & \(\beta - 1\) & \(\beta - \gamma\) & \(\gamma\)\\
\(\word-<\cci>(\ccp^{\nIh})\) & \(\beta - 1\) & \(\beta\) & \(\gamma\)\\
\(\ithComp{\ccy}[\cci]\) & \(\beta\) & \(\beta - \gamma - 1\) & \(\gamma\)\\
\(\ithComp{\ccy}<\cci>\) & \(\beta\) & \(\beta - 1\) & \(\gamma\)\\
\end{tabular}
\end{center}
It follows from Lemma \ref{org4a0db5b} that \(\ccd(\bdy, \word-(\cca)) \ge \ccd\) for \(0 \le \cca < \ccp^{\nIh}\),
and hence \(\bdy >_{\ccF} \word-(\ccp^{\nIh})\) and \(\beta - \gamma < \beta - \gamma - 1\).
Therefore \(\beta = \gamma = \word-<\xi>(\ccp^{\nIh})\).

\medskip
\noindent
(3) We have shown \(\ccN \in \Theta\). 
Suppose that \(\ccN = \xi\).
Since \(\xi < \eta\), it follows from Lemma \ref{org45cd32d} that \(\word-[\eta](\ccp^{\nIh}) = 0\)
and \(\eta, \xi \in \Cdn[\mGen-{\nIh}](\mvecA[\beta])\).
We first show that 
\[
\text{if\ } \ccL \in \Cdn[\mGen-{\nIh}](\mvecA[\alpha]) \setminus \set{\ccM},
\text{\ then\ } \ccL > \ccN. \tag{$*$}
\]
Let 
\[
 \bdy = \word-(\ccp^{\nIh}) + (\beta - \alpha)(\bdf_\ccM + \bdf_\ccL) + (\alpha - \beta)\bdf_\xi.
\]
\begin{center}
\begin{tabular}{lcccc}
 & \(\quad\ccM\quad\) & \(\quad\ccL\quad\) & \(\quad\ccN = \xi\quad\) & \(\quad\eta\quad\)\\
\hline
\(\word-[\cci](\ccp^{\nIh})\) & \(\alpha\) & \(\alpha\) & \(\beta\) & 0\\
\(\word-<\cci>(\ccp^{\nIh})\) & \(\alpha\) & \(\alpha\) & \(\beta\) & \(\beta\)\\
\(\ithComp{\ccy}[\cci]\) & \(\beta\) & \(\beta\) & \(\alpha\) & 0\\
\(\ithComp{\ccy}<\cci>\) & \(\beta\) & \(\beta\) & \(\alpha\) & \(\alpha\)\\
\end{tabular}
\end{center}
It follows from Lemma \ref{org4a0db5b} that \(\mbigDist{\bdy}{\word-(\cca)}\) for \(0 \le \cca < \ccp^{\nIh}\),
and hence \(\bdy >_{\ccF} \word-(\ccp^\nIh)\). Therefore \(\ccL > \ccN\).

We next show that \(\ccN = \ccM + 1\).
Assume that \(\ccN \neq \ccM + 1\).
Since \(\ccM < \ccN\), we see that
\(\ccM < \ccM + 1 < \ccN\).
Let \(\ccL \in \Cdn[\mGen-{\nIh}](\mvecA[\alpha]) \setminus \set{\ccM}\);
then \(\ccL > \ccN\) by \((*)\). 
Since \(\ccM < \ccM + 1 < \ccL\) and \(\ccM + 1, \ccL \not \in \Theta\),
it follows from Lemma \ref{orgb44e8dd} that
\(\ccM + 1 \in \Cdn[\mGen-{\nIh}](\mvecA[\alpha])\).
From (\(*\)), we see that \(\ccM + 1 > \ccN\), a contradiction.
Therefore \(\ccN = \ccM + 1\).

Finally, we show that \(\alpha = \beta - 1\).
Assume that \(\alpha < \beta - 1\),
and let \(\ccM' \in \Cdn[\mGen-{\nIh}](\mvecA[\beta - 1])\).
Since \(\ccM' \not \in \Theta\), it follows from Lemma \ref{org06db361} that \(\ccM' < \ccM\),
and hence \(\ccM' < \ccN\).
Therefore \(\ccM'\) satisfies the same conditions as \(\ccM\).
This implies that \(\ccM' = \ccN - 1\), a contradiction.
\end{proof}

\begin{example}
 \comment{Exm.}
\label{sec:orgebc46fd}
\label{org86ab854}
(1) Let \(\ccp = 3\), \(\ccd = 4\), \(\xi = 3\), and \(\eta = 7\).
Then
\[
 \word(1) = \mWord{1 1 0 1 0 0 0 1 0  \cdots}.
\]
Note that \(2,4,5,6 \in \Cdn[\mGen-{0}](0)\) and \(0,1,3,7 \in \Cdn[\mGen-{0}](1)\).
Thus, for example, the cases \((\ccM, \ccN) = (2, 3), (4, 7)\) correspond to Lemma \ref{org69880d5}(3).

\medskip
\noindent
(2) Let \(\ccp = 3\), \(\ccd = 6\), \(\xi = 6\), and \(\eta = 5\).
Then
\begin{align*}
 \word(1) &= \mWord{1 1 1 1 1 1 0 0 0  \cdots},\\
 \word(3) &= \mWord{2 2 1 0 0 1 1 1 0  \cdots}.
\end{align*}
Since \(3, 4 \in \Cdn[\mGen-{1}](\mVec{1 \\ 0})\) and \(\eta = 5 \in \Cdn[\mGen-{1}](\mVec{1\\1})\),
we see that the cases \((\ccM, \ccN) = (3,5), (4, 5)\) correspond to Lemma \ref{org69880d5}(2).
 
\end{example}

\begin{lemma}
 \comment{Lem.}
\label{sec:org7f13b66}
\label{org0bbad96}
Let \(\nIh \in \NN\).
Suppose that \(\cdn[\mGen-{\nIh}](\mvecB) \ge 2\) for \(\mvecB \in \mVct{\nIh}\).
Let
\(\ccM \in \NN \setminus \msupp|\ccE|(\mLex-{\nIh})\) and
\(\ccN \in \msupp|\ccE|(\mLex-{\nIh})\).
If \(\ccM < \ccN\), then \(\xi \le \ccM + 1\), \(\xi < \eta\), \(\ccM \not \in \Theta\), and \(\ccN \in \Theta\).
 
\end{lemma}

\begin{proof}
 \comment{Proof.}
\label{sec:orgba7835e}
We may assume that \(\ccN \in \Cdn[\mGen-{\nIh}](\mEvec{\nIh})\) and \(\ccM \in \Cdn[\mGen-{\nIh}](\allzero)\).
Since \(\ccM < \ccN\), it follows from Lemma \ref{org69880d5} that \(\ccN \in \Theta\)
and \(\word-<\eta>(\ccp^{\nIh}) = \word-<\xi>(\ccp^{\nIh}) = \word-<\ccN>(\ccp^{\nIh}) = 1\).
Hence \(\ccM \not \in \Theta\) since \(\word-<\ccM>(\ccp^{\nIh}) = 0\).

We show that \(\xi < \eta\).
Assume that \(\eta < \xi\).
It follows from Lemma \ref{org69880d5} that \(\ccN = \eta\).
Let \(\bdy = \word-(\ccp^{\nIh}) - \bdf_\xi + \bdf_\ccM + \bdf_\eta\); then \(\bdy <_{\ccF} \word-(\ccp^{\nIh})\).
\begin{center}
\begin{tabular}{lccc}
 & \(\quad\ccM\quad\) & \(\quad\ccN = \eta\quad\) & \(\quad\xi\quad\)\\
\hline
\(\word-[\cci](\ccp^{\nIh})\) & \(0\) & \(0\) & \(1\)\\
\(\word-<\cci>(\ccp^{\nIh})\) & \(0\) & \(1\) & \(1\)\\
\(\ithComp{\ccy}[\cci]\) & \(1\) & \(1\) & \(0\)\\
\(\ithComp{\ccy}<\cci>\) & \(1\) & \(1\) & \(0\)\\
\end{tabular}
\end{center}
Since \(\ccM, \ccN \in \Cdn[\mGen-{\nIh - 1}](\allzero)\),
it follows that \(\word-<\ccM>(\cca) = \word-<\ccN>(\cca) = 0\)
for \(0 \le \cca < \ccp^{\nIh}\), and
that \(\mDist{\bdy}{\word-(\cca)} \ge \ccd\).
This implies that \(\bdy >_{\ccF} \word-(\ccp^{\nIh})\), a contradiction.
Therefore \(\xi < \eta\).

It remains to show that \(\xi \le \ccM + 1\).
If \(\xi \le \ccM\), then it is obvious.
Suppose that \(\xi \ge \ccM + 1\).
By applying Lemma \ref{org69880d5} for \(\ccM\) and \(\xi\),
we see that \(\xi = \ccM + 1\).
Therefore \(\xi \le \ccM + 1\). 
\end{proof}

\comment{connect}
\label{sec:orgc3cfa31}

\subsection{Linearity breaking}
\label{sec:org48190d7}

In this subsection, we use the following notation.
Let \(\ccp = 3\) and \(\ccd = 3^{\nIk - 1} \ccd'\) with \(\ccd' \ge 2\).
Let \(\cca\) be the smallest integer \(\cca\) such that \(\word(\cca) \neq \word-(\cca)\).
It follows from Theorem \ref{org100bcd6} that \(\cca = 3^{\nIh} + 3^{\nIl}\) for some \(\nIh\) and \(\nIl\),
where \(\nIl \le \nIh\).
Let \(\bdy = \word(3^{\nIh} + 3^{\nIl})\), \(\bdz = \bdy - \word(3^{\nIh})\), and
\begin{align*}
 \ccN &= \max\, \msupp|\ccF|(\bdy - \word(3^\nIh) - \word(3^\nIl)) \\
     &= \max \Set{\cci \in \NN : \ithComp{\ccy}[\cci] \neq \word[\cci](3^{\nIh}) + \word[\cci](3^{\nIl})}.
\end{align*}
By Lemma \ref{org13a9541},
\begin{equation}
\label{equ:linear-breaking-support-1}
 (\word[\ccN](3^{\nIh}), \word[\ccN](3^{\nIl})) = (1, 1) \tor (2, 2).
\end{equation}

\begin{lemma}
 \comment{Lem.}
\label{sec:org7ced33d}
\label{org5e6c6e3}
Suppose that \(\nIh < \nIk\).
If \(\ccN = \eta\), then \(\eta < \xi\).
 
\end{lemma}

\begin{proof}
 \comment{Proof.}
\label{sec:org5d779d8}
It follows from \eqref{equ:linear-breaking-support-1} that
\(\word[\ccN](3^\nIh) = \word[\eta](3^\nIh) \neq 0\).
By Lemma \ref{org45cd32d}, we see that \(\eta < \xi\).
\end{proof}

\begin{lemma}
 \comment{Lem.}
\label{sec:org5c810b0}
\label{orgc29f204}
If \(\nIh < \nIk\), then the minimum distances of \(\mLex-{\nIh - 1}[\bdy]\) and \(\mLex-{\nIl - 1}[\bdz]\) are at least \(\ccd\),
where \(\mLex-{\nIh - 1}[\bdy]\) and \(\mLex-{\nIl - 1}[\bdz]\) are linear codes with generator matrices \(\mGen-{\nIh - 1}[\bdy]\) and \(\mGen-{\nIl - 1}[\bdz]\), respectively.
 
\end{lemma}

\begin{proof}
 \comment{Proof.}
\label{sec:org199f787}
Let \(0 \le \cca < 3^\nIh + 3^\nIl\).
Since \(\word(\cca) = \word-(\cca)\),
it follows that
\(\mDist{\bdy}{\word-(\cca)} = \mDist{\bdy}{\word(\cca)} \ge \ccd\) for \(0 \le \cca < 3^{\nIh} + 3^{\nIl}\).
Thus the minimum distance of \(\mLex-{\nIh - 1}[\bdy]\) is at least \(\ccd\).

If \(0 \le \cca < 3^{\nIl}\), then \(\cca \oplus_3 3^{\nIh} < 3^{\nIl} + 3^{\nIh}\), and hence
\begin{align*}
\mDist{\bdz}{\word-(\cca)} &= \mBigDist{\bdy - \word-(3^{\nIh})}{\word-(\cca)}\\
&= \mBigDist{\bdy}{\word-(\cca \oplus_3 3^{\nIh})} \ge \ccd.
\end{align*}
Therefore the minimum distance of \(\mLex-{\nIl - 1}[\bdz]\) is at least \(\ccd\).
\end{proof}

\begin{lemma}
 \comment{Lem.}
\label{sec:org0e31367}
\label{orgab95765}
If \(\nIh < \nIk\), then \(\msupp|\ccE|(\bdy) \subseteq \msupp|\ccE|(\mLex-{\nIh})\) and \(\msupp|\ccE|(\bdz) \subseteq \msupp|\ccE|(\mLex-{\nIl})\).
 
\end{lemma}

\begin{proof}
 \comment{Proof.}
\label{sec:org73907d4}
Theorem \ref{org56f1397} shows that
\(\mLex-{\nIk - 1}\) meets the Griesmer bound.
It follows from Corollary \ref{org09650b1}
that \(\cdn[\mGen-{\nIk - 1}](\mvecA) = \ccd'\) for
\(\mvecA \in \mVct{\nIk - 1}\).
Hence \(\cdn[\mGen-{\nIh}](\mvecA) = 3^{\nIk - \nIh - 1} \ccd'\)
for \(\mvecA \in \mVct{\nIh}\).

Note that \(\mGen-{\nIh} = \mGen{\nIh}\) and
\begin{equation}
\label{equ:linear-breaking-overN}
 \ithComp{\ccy}[\nIi] = \word[\nIi](3^{\nIh}) + \word[\cci](3^{\nIl}),\quad \ithComp{\ccz}[\nIi] = \word[\nIi](3^{\nIl})
\end{equation}
for \(\cci > \ccN\).
By \eqref{equ:linear-breaking-support-1},
\[
 (\word[\ccN](3^{\nIh}), \word[\ccN](3^{\nIl})) = (1, 1) \tor (2, 2).
\]
Thus \(\ccN \in \msupp|\ccF|(\mLex-{\nIl}) \subseteq \msupp|\ccE|(\mLex-{\nIl})\)
by Remark \ref{org04e980f}.
Now assume that there exists \(\ccM \in \msupp|\ccE|(\bdz) \setminus \msupp|\ccE|(\mLex-{\nIl})\).
We see that \(\ccM \neq \ccN\) and \(\ccM \not \in \msupp|\ccF|(\mLex-{\nIl})\).

\resetmycase
\begin{mycase}[\(\ccN < \ccM\)]
 \comment{Case. [\(\ccN < \ccM\)]}
\label{sec:org99a7b48}
Note that \(\ithComp{\ccz}<\ccM> \neq 0\).
Since \(\ccM > \ccN\), it follows that
\(\ithComp{\ccz}[\ccM] = \word[\ccM](3^{\nIl}) = 0\).
Hence \(\ccM = \eta\) and
\(\ithComp{\ccz}<\xi> = \ithComp{\ccz}<\eta> \neq 0\).
Since \(\eta = \ccM \not \in \msupp|\ccE|(\mLex-{\nIl})\)
it follows from Lemma \ref{orgedc6ece} that
\(\xi \not \in \msupp|\ccE|(\mLex-{\nIl})\), 
so
\(\word<\xi>(3^\nIl) = \word[\xi](3^\nIl) = 0 \neq \ithComp{\ccz}[\xi]\).
By the definition of \(\ccN\), we see that \(\xi < \ccN\).
Since \(\xi \not \in \msupp|\ccE|(\mLex-{\nIl})\) and \(\ccN \in \msupp|\ccE|(\mLex-{\nIl})\),
it follows from Lemma \ref{org0bbad96} that \(\xi \not \in \Theta\),
a contradiction.
 
\end{mycase} 

\begin{mycase}[\(\ccN > \ccM\)]
 \comment{Case. [\(\ccN > \ccM\)]}
\label{sec:org02f7caa}
It follows from Lemma \ref{org0bbad96} that \(\xi \le \ccM + 1\), \(\xi < \eta\), and \(\ccN \in \Theta\).
By Lemma \ref{org5e6c6e3}, \(\ccN = \xi\).

For \(\bdw = \bdy, \bdz\), we will compute
\begin{equation}
\label{equ:linear-breaking-target}
\Size{\msupp|\ccE|(\mLex-{\nIh - 1}[\bdw]) \setminus \msupp|\ccE|(\mLex-{\nIh - 1})}
\end{equation}
in two ways,
where \(\mLex-{\nIh - 1}[\bdw]\) is the code with generator matrix \(\mGen-{\nIh - 1}[\bdw]\). 
In the first way, the results will be at least \(3^{\nIk - \nIh - 1} \ccd'\), but in the second way, one of them will be less than \(3^{\nIk - \nIh - 1} \ccd'\).

\resetstep
 
 \begin{step}
 \comment{Step.}
\label{sec:orgd9a4ebe}
We show that \(\nIl = \nIh\) and
\begin{equation}
\label{equ:linear-breaking-N}
 \ccN \in \Cdn[\mGen{\nIh}](\mEvec{\nIh}).
\end{equation}
From Lemma \ref{org5db7ec7}, it suffices to show that
\[
 \ccN = \max \set{\cci \in \NN : \word[\cci](3^{\nIl}) \neq 0}.
\]
For \(\nIi > \ccN\),
we see that \(\word[\nIi](3^{\nIl}) = 0\).
Indeed, if \(\nIi = \eta\), then \(\word[\eta](3^{\nIl}) = 0\) since \(\xi < \eta\).
Suppose that \(\nIi \neq \eta\).
Since \(\ccM \not \in \msupp|\ccE|(\mLex-{\nIl})\),
it follows from Lemma \ref{org0bbad96} that \(\msupp|\ccE|(\mLex-{\nIl}) \setminus \Theta \subseteq \NN_{< \ccM}\),
and hence \(\word[\nIi](3^{\nIl}) = 0\).
Therefore \(\ccN = \max \set{\nIi \in \NN : \word[\nIi](3^{\nIl}) \neq 0}\).
Lemma \ref{org5db7ec7} implies that
\(\word[\ccN](3^{\nIm}) = \delta_{\nIl, \nIm}\).
Since \(\word[\ccN](3^{\nIh}) \neq 0\),
it follows that \(\nIh = \nIl\).
 
\end{step}

 \begin{step}
 \comment{Step.}
\label{sec:orgc7b8a06}
Let us compute \eqref{equ:linear-breaking-target}.
Lemma \ref{orgc29f204} shows that \(\mLex-{\nIh - 1}[\bdw]\) is
a linear codes of dimension \(\nIh + 1\) whose minimum distance is at least \(\ccd\).
By the Griesmer bound, we see that
\[
\Size{\msupp|\ccE|(\mLex-{\nIh - 1}[\bdw])} \ge 
 \ccg_3(\nIh + 1, \ccd)
 = (3^{\nIk - 1} + \cdots + 3^{\nIk - \nIh - 1}) \ccd'.
\]
Since \(\mLex-{\nIh - 1}\) meets the Griesmer bound, it follows that
\begin{equation}
\label{equ:linear-breaking-bound}
\Size{\msupp|\ccE|(\mLex-{\nIh - 1}[\bdw])
  \setminus \msupp|\ccE|(\mLex-{\nIh - 1})} \ge 3^{\nIk - \nIh - 1} \ccd'.\end{equation}
 
\end{step}

 \begin{step}
 \comment{Step.}
\label{sec:org3eda583}
From Lemma \ref{org69880d5}, we see that \(\ccN = \ccM + 1\).
Since \(\ccM\) is an arbitrary element of \(\msupp|\ccE|(\bdz) \setminus \msupp|\ccE|(\mLex-{\nIh})\),
it follows that
\begin{equation}
\label{equ:linear-breaking-supp-z}
\msupp|\ccE|(\bdz) \setminus \msupp|\ccE|(\mLex-{\nIh}) = \set{\ccN - 1}.
\end{equation}
Recall that \(\bdz = \bdy - \word(3^{\nIh})\).
This implies that \(\msupp|\ccE|(\bdy) \subseteq \msupp|\ccE|(\bdz) \cup \msupp|\ccE|(\mLex-{\nIh})\).
Therefore
\begin{equation}
\label{equ:linear-breaking-supp-y}
\msupp|\ccE|(\bdy) \setminus \msupp|\ccE|(\mLex-{\nIh}) \subseteq \set{\ccN - 1}.
\end{equation}
 
\end{step}

 \begin{step}
 \comment{Step.}
\label{sec:orgc2ffce8}
For \(\bdw \in \set{\bdy, \bdz}\), we show that
\begin{equation}
\label{equ:linear-breaking-supp-w}
  \msupp|\ccE|(\mLex-{\nIh - 1}[\bdw]) \setminus \msupp|\ccE|(\mLex-{\nIh - 1}) 
  \subseteq \Cdn[\mGen{\nIh}](\mEvec{\nIh}) \cup \set{\ccN - 1}.
\end{equation}
Note that
\[
  \msupp|\ccE|(\mLex-{\nIh - 1}[\bdw]) \setminus \msupp|\ccE|(\mLex-{\nIh - 1}) 
  = \msupp|\ccE|(\bdw) \setminus  \msupp|\ccE|(\mLex-{\nIh - 1}).
\]
By \eqref{equ:linear-breaking-supp-z} and \eqref{equ:linear-breaking-supp-y}, we see that
\[
\msupp|\ccE|(\bdw) \setminus \msupp|\ccE|(\mLex-{\nIh}) \subseteq \set{\ccN - 1}.
\]
Since
\[
 \msupp|\ccE|(\mLex-{\nIh}) \setminus \msupp|\ccE|(\mLex-{\nIh - 1}) = \Cdn[\mGen{\nIh}](\mEvec{\nIh}),
\]
it follows that
\[
  \msupp|\ccE|(\mLex-{\nIh - 1}[\bdw]) \setminus \msupp|\ccE|(\mLex-{\nIh - 1}) 
  \subseteq \Cdn[\mGen{\nIh}](\mEvec{\nIh}) \cup \set{\ccN - 1}.
\]
 
\end{step}

 \begin{step}
 \comment{Step.}
\label{sec:org289eb78}

Define \(\bdw\) as follows:
\begin{equation}
\label{equ:linear-breaking-def-w}
 \bdw = \begin{cases}
 \bdy & \tif \ithComp{\ccy}<\ccN> = 0, \\
 \bdz & \tif \ithComp{\ccy}<\ccN> \neq 0. \\
 \end{cases}
\end{equation}
We prove that
\begin{align*}
  \Size{\msupp|\ccE|(\mLex-{\nIh - 1}[\bdw]) \setminus \msupp|\ccE|(\mLex-{\nIh - 1})}  < 3^{\nIk - \nIh - 1} \ccd'.
\end{align*}
We first show that \(\ithComp{\ccw}<\xi> = \ithComp{\ccw}<\eta> = 0\).
Since \(\eta > \xi = \ccN\),
it follows from \eqref{equ:linear-breaking-overN} and Lemma \ref{org45cd32d} that
\[\ithComp{\ccw}[\eta] = \begin{cases}
 \word[\eta](3^{\nIh}) + \word[\eta](3^{\nIl}) = 0 & \tif \bdw = \bdy \\
 \word[\eta](3^{\nIl}) = 0 & \tif \bdw = \bdz. \\
\end{cases}
\]
Hence \(\ithComp{\ccw}<\eta> = \ithComp{\ccw}<\xi>\).
If \(\bdw = \bdy\), then \(\ithComp{\ccw}<\xi> = \ithComp{\ccy}<\ccN> = 0\)
by \eqref{equ:linear-breaking-def-w}.
Suppose that \(\bdw = \bdz\). Then \(\ithComp{\ccy}<\xi> = \ithComp{\ccy}<\ccN>  \neq 0\).
Since \(\ithComp{\ccy}<\xi> < \word<\xi>(3^\nIh) + \word<\xi>(3^\nIl) = 1 + 1\),
it follows that \(\ithComp{\ccy}<\xi> = 1\),
and hence \(\ithComp{\ccz}<\xi> = 0\).
Therefore \(\ithComp{\ccw}<\xi> = \ithComp{\ccw}<\eta> = 0\).
By \eqref{equ:linear-breaking-supp-w},
\[
  \msupp|\ccE|(\mLex-{\nIh - 1}[\bdw]) \setminus \msupp|\ccE|(\mLex-{\nIh - 1}) 
  \subseteq \Cdn[\mGen{\nIh}](\mEvec{\nIh}) \cup \set{\ccN - 1} \setminus \set{\xi, \eta}.
\]
From \eqref{equ:linear-breaking-N},
we see that \(\xi, \eta \in \Cdn[\mGen-{\nIh}](\mEvec{\nIh})\), and hence
\[
  \Size{\msupp|\ccE|(\mLex-{\nIh - 1}[\bdw]) \setminus \msupp|\ccE|(\mLex-{\nIh - 1})}
  \le 3^{\nIk - \nIh - 1} \ccd' + 1 - 2 < 3^{\nIk - \nIh - 1} \ccd',
\]
contrary to \eqref{equ:linear-breaking-bound}. 
Therefore we conclude that \(\msupp|\ccE|(\bdz) \subseteq \msupp|\ccE|(\mLex-{\nIl})\) and \(\msupp|\ccE|(\bdy) \subseteq \msupp|\ccE|(\mLex-{\nIh})\). 
\qedhere
 
\end{step} 
 
\end{mycase} 
\end{proof}

\begin{corollary}
 \comment{Cor.}
\label{sec:orgd1bf050}
\label{org7f53679}
If \(\nIh < \nIk\), then \(\mRes(\mLex-{\nIh - 1}[\bdy])\) and \(\mRes(\mLex-{\nIl - 1}[\bdz])\) meet the Griesmer bound.
 
\end{corollary}

\begin{proof}
 \comment{Proof.}
\label{sec:org55ce329}
By Lemma \ref{orgc29f204}, 
the minimum distances of
\(\mLex-{\nIh - 1}[\bdy]\) and \(\mLex-{\nIl - 1}[\bdz]\) are at least \(\ccd\).
Lemma \ref{orgab95765} shows that
\[
 \msupp|\ccE|(\mLex-{\nIh - 1}[\bdy]) \subseteq \msupp|\ccE|(\mLex-{\nIh})
\tand
\msupp|\ccE|(\mLex-{\nIl - 1}[\bdz]) \subseteq \msupp|\ccE|(\mLex-{\nIl}).
\]
Since \(\mRes(\mLex-{\nIh})\) and \(\mRes(\mLex-{\nIl})\) 
meet the Griesmer bound by Theorem \ref{org56f1397},
it follows that \(\mRes(\mLex-{\nIh - 1}[\bdy])\) and
\(\mRes(\mLex-{\nIl - 1}[\bdz])\) also attain the Griesmer bound.
\end{proof}

\begin{lemma}
 \comment{Lem.}
\label{sec:orgf17e8d6}
\label{org3ed1c0d}
Suppose that \(\nIh < \nIk\).
Let
\[\mVct{\nIh - 1}[\bdy] =
 \begin{cases}
 \mVct{\nIh} & \tif \nIl < \nIh,\\
 \mVct{\nIh} \cup \set{2 \mEvec{\nIh}} \setminus \set{\mEvec{\nIh}} & \tif \nIl = \nIh, \\
 \end{cases}
\]
and
\[
 \mVct*{\nIh - 1}[\bdy] = \mVct{\nIh - 1}[\bdy] \cup \set{\allzero^{\marray{\nIh}}}.
\]
If \(\mvecB[\alpha] \in \FF_\ccp^{\marray{\nIh}} \setminus \mVct*{\nIh - 1}[\bdy]\),
then \(\Cdn[\mGen{\nIh - 1}[\bdy]](\mvecB[\alpha]) = \emptyset\).
In particular, 
\(\cdn[\mGen{\nIh - 1}[\bdy]](\mvecA) = 3^{\nIk - \nIh - 1} \ccd'\)
for \(\mvecA \in \mVct{\nIh - 1}[\bdy]\).
 
\end{lemma}

\begin{proof}
 \comment{Proof.}
\label{sec:org10edf21}
If \(\mvecB \not \in \mVct*{\nIh - 1}\), then
\(\Cdn[\mGen{\nIh - 1}[\bdy]](\mvecB[\alpha]) \subseteq \Cdn[\mGen{\nIh - 1}](\mvecB) = \emptyset\).
Suppose that \(\mvecB = \allzero \in \FF_{\ccp}^{\marray{\nIh - 1}}\).
Assume that there exists \(\ccL \in \Cdn[\mGen{\nIh - 1}[\bdy]](\allzero^{\alpha})\).
Since \(\mvecB[\alpha] \neq \allzero\), it follows that \(\ccL \in \msupp|\ccE|(\bdy) \subseteq \msupp|\ccE|(\mLex-{\nIh})\).
Therefore \(\ccL \in \Cdn[\mGen{\nIh}](\mEvec{\nIh})\) since \(\ccL \in \Cdn[\mGen{\nIh - 1}](\allzero)\).
Moreover,
\[
 \ithComp{\ccz}<\ccL> = \ithComp{\ccy}<\ccL> - \word<\ccL>(\ccp^\nIh) = \alpha - 1.
\]

\resetmycase
\begin{mycase}[\(\nIl < \nIh\)]
 \comment{Case. [\(\nIl < \nIh\)]}
\label{sec:org8540104}

Since \(\allzero^\alpha \not \in \mVct*{\nIh}\), we see that \(\alpha = 2\).
Hence \(\ithComp{\ccz}<\ccL> = \alpha - 1 = 1\) and \(\ccL \in \msupp|\ccE|(\bdz)\).
Since \(\ccL \in \Cdn[\mGen{\nIh - 1}](\allzero)\) and \(\nIl < \nIh\), it follows that \(\ccL \not \in \msupp|\ccE|(\mLex-{\nIl})\),
however, \(\msupp|\ccE|(\bdz) \subseteq \msupp|\ccE|(\mLex-{\nIl})\), a contradiction.
 
\end{mycase} 

\begin{mycase}[\(\nIl = \nIh\)]
 \comment{Case. [\(\nIl = \nIh\)]}
\label{sec:org5247225}

Since \(\allzero^\alpha \not \in \mVct*{\nIh - 1}[\bdy]\), we see that \(\alpha = 1\).
Hence \(\ithComp{\ccz}<\ccL> = \alpha - 1 = 0\)  and \(\ccL \not \in \msupp|\ccE|(\bdz)\).
Since \(\nIl = \nIh\) and \(\ccL \in \Cdn[\mGen{\nIh}](\mEvec{\nIh})\), it follows that
\(\ccL \not \in \msupp|\ccE|(\mLex-{\nIh - 1}[\bdz])\)
and \(\ccL  \in \msupp|\ccE|(\mLex-{\nIh})\), 
contrary to \(\msupp|\ccE|(\mLex-{\nIh - 1}[\bdz]) = \msupp|\ccE|(\mLex-{\nIh})\).

Therefore 
\(\Cdn[\mGen{\nIh - 1}[\bdy]](\mvecB[\alpha]) = \emptyset\)
for \(\mvecB[\alpha] \in \FF_\ccp^{\marray{\nIh}} \setminus \mVct*{\nIh - 1}[\bdy]\).
Since \(\Size{\mVct{\nIh - 1}[\bdy]} = \Size{\mVct{\nIh}}\),
it follows from Corollary \ref{org09650b1} that 
\(\cdn[\mGen{\nIh - 1}[\bdy]](\mvecA) = 3^{\nIk - \nIh - 1} \ccd'\) for \(\mvecA \in \mVct{\nIh - 1}[\bdy]\).
\hspace*{\fill} \(\qedhere\)
 
\end{mycase} 
\end{proof}

\comment{connect}
\label{sec:org328cf82}
We will examine the components of \(\bdy\) and \(\bdz\), ultimately leading to a contradiction.

\begin{lemma}
 \comment{Lem.}
\label{sec:orgf976a91}
\label{org85cba49}
Suppose that \(\nIh < \nIk\).

\begin{enumerate}
\item Let \(0 \le \nIm \le \nIh\) and \(\mvecB \in \mVct*{\nIm}\). 
If \(\Cdn[\mGen{\nIm}](\mvecB) \subseteq \ZZ_{> \ccN}\), then
\(\ithComp{\ccy}<\nIi> = \word<\nIi>(3^{\nIh}) + \word<\nIi>(3^{\nIl})\)
for \(\nIi \in \Cdn[\mGen{\nIm}](\mvecB)\).
\item If \(\Cdn[\mGen{\nIh}](\mvecA[\alpha]) \subseteq \ZZ_{> \ccN}\), then
\(\Cdn[\mGen{\nIh}](\mvecA[\alpha]) = \Cdn[\mGen{\nIh - 1}[\bdy]](\mvecA[\alpha + \gamma])\),
where \(\gamma\) is the \(\nIl\)-th component of \(\mvecA[\alpha]\); in particular,
\(\ithComp{\ccy}<\nIi> \neq \alpha +  \gamma\) for \(\nIi \in \Cdn[\mGen{\nIh - 1}](\mvecA) \cap \ZZ_{\le \ccN}\).
\item If \(\Cdn[\mGen{\nIl}](\mvecA[\alpha]) \subseteq \ZZ_{> \ccN}\), then
\(\Cdn[\mGen{\nIl}](\mvecA[\alpha]) = \Cdn[\mGen{\nIl - 1}[\bdz]](\mvecA[\alpha])\);
in particular, \(\ithComp{\ccz}<\nIi> \neq \alpha\) for \(\nIi \in \Cdn[\mGen{\nIl - 1}](\mvecA) \cap \ZZ_{\le \ccN}\).
\end{enumerate}
 
\end{lemma}

\begin{proof}
 \comment{Proof.}
\label{sec:orgeda3ef0}

(1) This is obvious when \(\nIi \neq \eta\).
Suppose that \(\nIi = \eta\).
We first show that \(\xi > \ccN\).
If \(\xi > \eta\), then \(\xi > \eta = \nIi > \ccN\).
Suppose that \(\xi < \eta\).
Recall that \(\mGen{\nIh} = \mGen-{\nIh}\).
It follows from Theorem \ref{org56f1397} and Corollary \ref{org09650b1} that
\(\cdn[\mGen{\nIm}](\mvecB) = 3^{\nIk - \nIm - 1} \ccd' \ge 2\).
Hence \(\mGen{\nIm}<\xi> = \mGen{\nIm}<\eta> = \mvecB\) by Lemma \ref{org45cd32d}.
Thus \(\xi \in \Cdn[\mGen{\nIm}](\mvecB) \subseteq \ZZ_{> \ccN}\) and \(\xi > \ccN\).
Therefore
\begin{align*}
 \ithComp{\ccy}<\eta> &= \ithComp{\ccy}[\eta] + \ithComp{\ccy}[\xi]\\
&= \word[\eta](3^{\nIh}) + \word[\eta](3^{\nIl}) + 
 \word[\xi](3^{\nIh}) + \word[\xi](3^{\nIl})\\
&= \word<\eta>(3^{\nIh}) + \word<\eta>(3^{\nIl}).
\end{align*}

\noindent
(2) Corollary \ref{org7f53679} shows that \(\mRes(\mLex-{\nIh - 1}[\bdy])\) meets the Griesmer bound.
By (1), we see that \(\ithComp{\ccy}<\nIi> = \word<\nIi>(3^{\nIh}) + \word<\nIi>(3^{\nIl})
= \alpha + \gamma\)  for \(\nIi \in \Cdn[\mGen{\nIh}](\mvecA[\alpha])\).
Therefore \(\nIi \in \Cdn[\mGen{\nIh - 1}[\bdy]](\mvecA[\alpha + \gamma])\), that is,
\(\Cdn[\mGen{\nIh}](\mvecA[\alpha]) \subseteq \Cdn[\mGen{\nIh - 1}[\bdy]](\mvecA[\alpha + \gamma])\).
Since \(\mRes(\mLex-{\nIh - 1}[\bdy])\) meets the Griesmer bound,
it follows from Corollary \ref{orgb141d90} that
\(\Cdn[\mGen{\nIh}](\mvecA[\alpha]) = \Cdn[\mGen{\nIh - 1}[\bdy]](\mvecA[\alpha + \gamma])\).
In particular, if 
\(\nIi \in \Cdn[\mGen{\nIh - 1}](\mvecA) \cap \ZZ_{\le \ccN}\),
then \(\nIi \not \in \Cdn[\mGen{\nIh - 1}[\bdy]](\mvecA[\alpha + \gamma])\),
so \(\ithComp{\ccy}<\nIi> \neq \alpha +  \gamma\).

\noindent
(3) Corollary \ref{org7f53679} shows that \(\mRes(\mLex-{\nIl - 1}[\bdz])\) meets the Griesmer bound.
By (1), we see that \(\ithComp{\ccz}<\nIi> = \word<\nIi>(3^{\nIl}) = \alpha\) for  \(\nIi \in \Cdn[\mGen{\nIl}](\mvecA[\alpha])\).
Therefore \(\Cdn[\mGen{\nIl}](\mvecA[\alpha]) \subseteq \Cdn[\mGen{\nIl - 1}[\bdz]](\mvecA[\alpha])\).
Since \(\mRes(\mLex-{\nIl - 1}[\bdz])\) meets the Griesmer bound, 
it follows that \(\Cdn[\mGen{\nIl}](\mvecA[\alpha]) = \Cdn[\mGen{\nIl - 1}[\bdz]](\mvecA[\alpha])\).
\end{proof}

\begin{lemma}
 \comment{Lem.}
\label{sec:org3e06dd7}
\label{orgdc2a2bc}
Suppose that \(\nIh < \nIk\).

\begin{enumerate}
\item \(\ithComp{\ccy}<\ccN> \neq \word<\ccN>(3^{\nIh}) + \word<\ccN>(3^{\nIl})\).
\item If \(N \in \Cdn[\mGen{\nIh}](\mvecA[\beta])\), then there exist \(\ccM \in \NN\) and \(\alpha \in \FF_\ccp\) satisfying the following conditions:  \(\ccM \in \Cdn[\mGen{\nIh}](\mvecA[\alpha])\), \(\alpha > \beta\), \(\ccM < \ccN\), and \(\ithComp{\ccy}<\ccM> = \word<\ccN>(3^{\nIh}) + \word<\ccN>(3^{\nIl})\).
\end{enumerate}
 
\end{lemma}

\begin{proof}
 \comment{Proof.}
\label{sec:orgae8cf6a}

(1) This is obvious when \(\ccN \neq \eta\)
since \(\ithComp{\ccy}[\ccN] \neq \word[\ccN](3^{\nIh}) + \word[\ccN](3^{\nIl})\).
Suppose that \(\ccN = \eta\).
Lemma \ref{org5e6c6e3} shows that \(\ccN = \eta < \xi\),
and hence \(\ithComp{\ccy}<\xi> = \word[\xi](3^{\nIh}) + \word[\xi](3^{\nIl})\).
Therefore
\begin{align*}
 \ithComp{\ccy}<\ccN> &= \ithComp{\ccy}[\eta] + \ithComp{\ccy}[\xi]\\
& = \ithComp{\ccy}[\eta] + \word[\xi](3^{\nIh}) + \word[\xi](3^{\nIl})\\
&\neq \word[\eta](3^{\nIh}) + \word[\eta](3^{\nIl}) + \word[\xi](3^{\nIh}) + \word[\xi](3^{\nIl})  = \word<\eta>(3^{\nIh}) + \word<\eta>(3^{\nIl}).
\end{align*}

\noindent
(2) Let \(\gamma = \word<\ccN>(3^{\nIl})\); then \(\ithComp{\ccy}<\ccN> \neq \beta + \gamma\) by (1).
Note that
\begin{align*}
\Cdn[\mGen{\nIh - 1}](\mvecA) &= \bigcup_{\delta \in \FF_3} \Cdn[\mGen{\nIh}](\mvecA[\delta])\\
&= \bigcup_{\delta \in \FF_3} \Cdn[\mGen{\nIh - 1}[\bdy]](\mvecA[\delta]).
\end{align*}

\resetstep
 
 \begin{step}
 \comment{Step.}
\label{sec:org0eec552}
We show that
\(\mvecA[\beta + \gamma] \in \mVct{\nIh - 1}[\bdy]\) and
\(\cdn[\mGen{\nIh - 1}[\bdy]](\mvecA[\beta + \gamma]) = 3^{\nIk - \nIh - 1} \ccd'\).
Suppose that \(\mvecA \neq \allzero\).
Since \(\ccN \in \Cdn[\mGen{\nIh}](\mvecA[\beta])\),
it follows that \(\mvecA \in \mVct{\nIh - 1}\), and hence \(\mvecA[\beta + \gamma] \in \mVct{\nIh - 1}[\bdy]\).
Suppose that \(\mvecA = \allzero\).
Then \(\ccN \in \Cdn[\mGen{\nIh}](\allzero^\beta)\). 
Since \(\word[\ccN](3^{\nIl}) \in \set{1, 2}\) by \eqref{equ:linear-breaking-support-1},
it follows that \(\ccN \in \msupp|\ccF|(\mLex-{\nIl}) \subseteq \msupp|\ccE|(\mLex-{\nIl})\),
and hence \(\beta = 1\) and \(\nIl = \nIh\).
Therefore \(\gamma = \word<\ccN>(3^\nIl) = 1\) and \(\mvecA[\beta + \gamma] = 2 \mEvec{\nIh} \in \mVct{\nIh - 1}[\bdy]\).
Lemma \ref{org3ed1c0d} shows that \(\cdn[\mGen{\nIh - 1}[\bdy]](\mvecA[\beta + \gamma]) = 3^{\nIk - \nIh - 1} \ccd'\).
 
\end{step}

 \begin{step}
 \comment{Step.}
\label{sec:org1c98560}
We show that there exist \(\ccM \in \NN\) and \(\alpha \in \FF_\ccp\) satisfying the following three conditions.

\begin{enumerate}[label=(\roman*)]
  \item $\alpha \neq \beta$.  
  \item $\ccM \in \Cdn[\mGen{\nIh}](\mvecA[\alpha])$.
  \item $\ccM \in \Cdn[\mGen{\nIh - 1}[\bdy]](\mvecA[\beta + \gamma])$, that is, $\ithComp{\ccy}<\ccM> = \beta + \gamma$.
\end{enumerate}

By (1), we see that \(\ccN \in \Cdn[\mGen{\nIh}](\mvecA[\beta]) \setminus \Cdn[\mGen{\nIh - 1}[\bdy]](\mvecA[\beta + \gamma])\).
Thus there exists \(\ccM \in \Cdn[\mGen{\nIh - 1}[\bdy]](\mvecA[\beta + \gamma]) \setminus \Cdn[\mGen{\nIh}](\mvecA[\beta])\)
since \(\cdn[\mGen{\nIh - 1}[\bdy]](\mvecA[\beta + \gamma]) = \cdn[\mGen{\nIh}](\mvecA[\beta]) = 3^{\nIk - \nIh - 1} \ccd'\) by Step 1.
Let \(\mvecA[\alpha] = \mGen{\nIh}<\ccM>\).
Since \(\ccM \not \in \Cdn[\mGen{\nIh}](\mvecA[\beta])\),
we see that \(\alpha \neq \beta\). Thus \(\ccM\) and \(\alpha\) satisfy the conditions (i)--(iii).

Note that
\begin{equation}
\label{equ:lem-existence-1}
\word<\ccM>(3^{\nIh}) + \word<\ccM>(3^{\nIl}) \neq \beta + \gamma.
\end{equation}
Indeed, if \(\nIl < \nIh\), then \(\word<\ccM>(3^{\nIl}) = \word<\ccN>(3^{\nIl}) = \gamma\)
since \(\ccM, \ccN \in \Cdn[\mGen{\nIh - 1}](\mvecA)\), and hence
\(\word<\ccM>(3^{\nIh}) + \word<\ccM>(3^{\nIl}) = \alpha + \gamma \neq \beta + \gamma\).
If \(\nIl = \nIh\), then \(\gamma = \word<\ccN>(3^{\nIl}) = \word<\ccN>(3^{\nIh}) = \beta\), and hence
\(\word<\ccM>(3^{\nIh}) + \word<\ccM>(3^{\nIl}) = 2 \alpha \neq 2 \beta = \beta + \gamma\).
Therefore \eqref{equ:lem-existence-1} holds.
 
\end{step}

 \begin{step}
 \comment{Step.}
\label{sec:org5088b3c}
We show that if \(\ccM\) and \(\alpha\) satisfy (i)--(iii) and \(\alpha < \beta\), then \(\ccN = \xi\), \(\xi < \eta\), and \(\ccM = \ccN - 1\).

Assume that \(\Cdn[\mGen{\nIh}](\mvecA[\alpha]) \subseteq \ZZ_{> \ccN}\).
Since \(\ccM \in \Cdn[\mGen{\nIh}](\mvecA[\alpha])\),
it follows from Lemma \ref{org85cba49} that
\(\ithComp{\ccy}<\ccM> = \word<\ccM>(3^{\nIh}) + \word<\ccM>(3^{\nIl})\).
However, \(\ithComp{\ccy}<\ccM> = \beta + \gamma\) by (iii), contrary to \eqref{equ:lem-existence-1}.
Hence \(\Cdn[\mGen{\nIh}](\mvecA[\alpha]) \not \subseteq \ZZ_{> \ccN}\).
Thus there exists \(\ccL \in \Cdn[\mGen{\nIh}](\mvecA[\alpha])\) such that \(\ccL < \ccN\).
Since \(\alpha < \beta\), it follows from Lemma \ref{org69880d5} that \(\ccN \in \Theta\) and \(\word[\eta](3^{\nIh}) = 0\).
Since \(\word[\ccN](3^{\nIh}) \in \set{1, 2}\), we see that \(\ccN \neq \eta\), and hence \(\ccN = \xi\).
Lemmas \ref{org69880d5} and \ref{org45cd32d} imply that \(\xi < \eta\) and \(\word<\eta>(3^{\nIh}) = \word<\xi>(3^{\nIh}) = \beta\).

We show that \(\ccM < \ccN\).
Since \(\word<\ccM>(3^{\nIh}) = \alpha\) and \(\word<\eta>(3^{\nIh}) = \beta\),
it follows that \(\ccM \neq \eta\).
If \(\ccM > \ccN\), then 
\[
\ithComp{\ccy}<\ccM> = \ithComp{\ccy}[\ccM] = \word[\ccM](3^{\nIh}) + \word[\ccM](3^{\nIl})
= \word<\ccM>(3^{\nIh}) + \word<\ccM>(3^{\nIl}),
\]
contrary to \eqref{equ:lem-existence-1} since \(\ithComp{\ccy}<\ccM> = \beta + \gamma\).
Therefore \(\ccM < \ccN\).
Lemma \ref{org69880d5} shows that \(\ccM = \ccN - 1\).
 
\end{step}

 \begin{step}
 \comment{Step.}
\label{sec:orga823a49}
We show that if \(\ccM\) and \(\alpha\) satisfy (i)--(iii) and \(\alpha < \beta\), then 
there exist \(\ccM'\) and \(\alpha'\) satisfying (i)--(iii) and \(\ccM' \neq \ccM\); in particular \(\alpha' > \beta\).

By Step 3, we see that \(\ccN = \xi\) and \(\xi < \eta\).
Lemma \ref{org45cd32d} implies that
\(\mGen{\nIh}<\eta> = \mGen{\nIh}<\xi> = \mvecA[\beta]\).
Since \(\eta > \xi = \ccN\), it follows that
\(\ithComp{\ccy}[\eta] = \word[\eta](3^{\nIh}) + \word[\eta](3^{\nIl}) = 0 + 0\),
and that \(\ithComp{\ccy}<\eta> = \ithComp{\ccy}<\xi> = \ithComp{\ccy}<\ccN> \neq \beta + \gamma\) by (1).
Hence
\(\mGen{\nIh - 1}[\bdy]<\eta> = \mGen{\nIh - 1}[\bdy]<\xi> \neq \mvecA[\beta + \gamma]\).
Therefore
\(\xi, \eta \in \Cdn[\mGen{\nIh}](\mvecA[\beta])\) and
\(\xi, \eta \not \in \Cdn[\mGen{\nIh - 1}[\bdy]](\mvecA[\beta + \gamma])\).
This implies that the size of \(\Cdn[\mGen{\nIh - 1}[\bdy]](\mvecA[\beta + \gamma]) \setminus \Cdn[\mGen{\nIh}](\mvecA[\beta])\) 
is at least two.
Hence there exists \(\ccM' \in \Cdn[\mGen{\nIh - 1}[\bdy]](\mvecA[\beta + \gamma]) \setminus \Cdn[\mGen{\nIh}](\mvecA[\beta])\) 
such that \(\ccM' \neq \ccM\).
Let \(\mvecA[\alpha'] = \mGen{\nIh}<\ccM'>\).
It follows from Step 3 that \(\alpha' > \beta\) since otherwise \(\ccM' = \ccN - 1 = \ccM\).

From Steps 2--4, 
we conclude that there exist \(\alpha\) and \(\ccM\) satisfying (i)--(iii) and \(\alpha > \beta\).
 
\end{step}

 \begin{step}
 \comment{Step.}
\label{sec:org14312cf}
If \(\ccM < \ccN\), then \(\ccM\) and \(\alpha\) satisfy the conditions.
Suppose that \(\ccM > \ccN\).
We show that \(\xi\) and \(\alpha\) satisfy the conditions.
Since \(\ccM > \ccN\), we see that
\(\ithComp{\ccy}[\ccM] = \word[\ccM](3^\nIh) + \word[\ccM](3^\nIl)\).
Hence \(\ccM = \eta\) because \(\ithComp{\ccy}<\ccM> = \beta + \gamma \neq \word<\ccM>(3^\nIh) + \word<\ccM>(3^\nIl)\).
Since \(\ccM \in \Cdn[\mGen{\nIh}](\mvecA[\alpha])\),
\(\ccN \in \Cdn[\mGen{\nIh}](\mvecA[\beta])\), 
\(\alpha > \beta\), and \(\ccM > \ccN\), it follows from Lemma \ref{org69880d5} that
\(\word[\ccM](3^{\nIh}) = \word[\eta](3^{\nIh}) = 0\).
If \(\xi > \ccN\), then
\begin{align*}
\ithComp{\ccy}<\ccM> = \ithComp{\ccy}<\eta> = \ithComp{\ccy}[\eta] + \ithComp{\ccy}[\xi]
&= \word[\eta](3^\nIh) + \word[\eta](3^\nIl) + \word[\xi](3^\nIh) + \word[\xi](3^\nIl)\\
&= \word<\eta>(3^\nIh) + \word<\eta>(3^\nIl)\\
&= \word<\ccM>(3^\nIh) + \word<\ccM>(3^\nIl) \neq \beta + \gamma,
\end{align*}
a contradiction.
Thus \(\xi < \ccN < \ccM = \eta\).
Lemma \ref{org45cd32d} shows that \(\mGen{\nIh}<\xi> = \mGen{\nIh}<\eta> = \mvecA[\alpha]\).
Hence \(\xi \in \Cdn[\mGen{\nIh}](\mvecA[\alpha])\).
Since \(\ithComp{\ccy}[\eta] = \word[\eta](3^{\nIh}) + \word[\eta](3^{\nIl}) = 0\),
it follows that
\(\ithComp{\ccy}<\xi> = \ithComp{\ccy}<\eta> = \ithComp{\ccy}<\ccM> = \beta + \gamma\).
Therefore \(\xi\) and \(\alpha\) satisfy the conditions.
\hspace*{\fill} \(\qedhere\)
 
\end{step} 
\end{proof}

\begin{corollary}
 \comment{Cor.}
\label{sec:org5fcb25a}
\label{org6d4e5d6}
Suppose that \(\nIh < \nIk\).
Let \(\alpha\) and \(\beta\) be as in Lemma \ref{orgdc2a2bc}.
Then \(\word[\ccN](3^{\nIh}) = \word[\ccN](3^{\nIl}) = 1\), \(\alpha = 2\), and \(\beta = 1\).
 
\end{corollary}

\begin{proof}
 \comment{Proof.}
\label{sec:orgab6460a}

We first show that \(\word[\ccN](3^{\nIh}) = \word[\ccN](3^{\nIl}) = 1\).
By \eqref{equ:linear-breaking-support-1},
\((\word[\ccN](3^\nIh),\ \word[\ccN](3^\nIl)) = (1, 1)\) or \((2, 2)\).
Assume that \((\word[\ccN](3^\nIh), \word[\ccN](3^\nIl)) = (2, 2)\).
Since \(\word<\ccN>(3^\nIh) = \beta < \alpha\), it follows that
\(\beta \neq 2\). Hence \(\word<\ccN>(3^\nIh) \neq 2 = \word[\ccN](3^{\nIh})\).
This implies that \(\ccN = \eta\) and \(\word<\xi>(3^{\nIh}) \neq 0\).
By Lemma \ref{org45cd32d}, we see that \(\word[\eta](3^\nIh) = 0\) when \(\xi < \eta\),
and \(\word[\eta](3^\nIh) \neq \ccp - 1 = 2\) when \(\xi > \eta\), a contradiction.
Therefore \((\word[\ccN](3^\nIh), \word[\ccN](3^\nIl)) = (1, 1)\).

We next show that \(\alpha = 2\) and \(\beta = 1\).
If \(\ccN \neq \eta\), then
\(\beta = \word<\ccN>(3^{\nIh}) = \word[\ccN](3^{\nIh}) = 1\),
and hence \(\alpha = 2\) since \(\alpha > \beta\).
Suppose that \(\ccN = \eta\).
Assume that \(\beta = \word<\eta>(3^\nIh) \neq 1\).
Since \(\beta < \alpha \le 2\), we see that \(\word<\eta>(3^\nIh) = 0\).
Since \(\word[\eta](3^\nIh) = 1 \neq 0\),
it follows from Lemma \ref{org45cd32d} that \(\xi > \eta\).
This implies that \(\word<\eta>(3^\nIh) = 0 < \word<\xi>(3^\nIh)\), contrary to Lemma \ref{org45cd32d}.
Therefore \(\word<\eta>(3^\nIh) = \beta = 1\) and \(\alpha = 2\).
\end{proof}

\subsection{Ternary linear lexicographic codes}
\label{sec:org48f7a86}

\begin{theorem} \comment{Thm.}
\label{sec:orgbda3305}
\label{orga034f76}
Let \(\ccp = 3\), \(\ccd = 3^{\nIk - 1} \ccd'\), and \(\ccd' \ge 2\).
Then \(\mLex{\nIk - 1}\) is a linear code.

\end{theorem}

\begin{proof}
 \comment{Proof.}
\label{sec:orgcadec69}
Assume that \(\mLex{\nIk - 1}\) is not linear.
Let \(\nIh\), \(\nIl\), \(\bdy\), \(\bdz\), \(\ccN\), \(\ccM\), \(\alpha\), and \(\beta\) be as in Lemma \ref{orgdc2a2bc}.
Then \(\word[\ccN](3^{\nIh}) = \word[\ccN](3^{\nIl}) = 1\), \(\alpha = 2\), and \(\beta = 1\).
Let \(\mvecC = \mGen{\nIl - 1}<\ccN>\) and \(\mvecA = \mGen{\nIh - 1}<\ccN>\).

\resetmycase

\begin{mycase}[\(\ccN \not \in \Theta\)]
 \comment{Case. [\(\ccN \not \in \Theta\)]}
\label{sec:org955a1e9}

Note that \(\ithComp{\ccy}<\ccM> = \word<\ccN>(3^\nIh) + \word<\ccN>(3^\nIl) = 2\) and
\(\ccM \in \Cdn[\mGen{\nIh}](\mvecA[2])\).
This implies that \(\ithComp{\ccz}<\ccM> = \ithComp{\ccy}<\ccM> - \word<\ccM>(3^\nIh) = 2 - 2 = 0\).
Since \(\ccN \not \in \Theta\) and \(\ccN \in \Cdn[\mGen{\nIl}](\mvecC[1])\),
it follows from Lemma \ref{org06db361} that \(\Cdn[\mGen{\nIl}](\mvecC[0]) \subseteq \ZZ_{> \ccN}\).
Since \(\ccM \in \Cdn[\mGen{\nIl - 1}](\mvecC) \cap \ZZ_{< \ccN}\),
it follows from Lemma \ref{org85cba49} that \(\ithComp{\ccz}<\ccM> \neq 0\), a contradiction.
 
\end{mycase}

\begin{mycase}[\(\ccN = \eta\)]
 \comment{Case. [\(\ccN = \eta\)]}
\label{sec:org1f04256}

Since \(\word<\ccN>(3^{\nIh}) = \word[\ccN](3^{\nIh}) = 1\), we see that \(\word<\xi>(3^{\nIh}) = 0\).
By Lemma \ref{org45cd32d}, we see that \(\eta < \xi\).
Let \(\gamma = \word<\xi>(3^{\nIl})\).
Then \(\word<\ccN>(3^{\nIl}) = \word[\eta](3^{\nIl}) + \word<\xi>(3^{\nIl}) = 1 + \gamma\).
Lemma \ref{org45cd32d} implies that \(\gamma + 1 > \gamma\).

Since \(\ccN = \eta\), \(\ccN \in \Cdn[\mGen{\nIl}](\mvecC[\gamma + 1])\),
\(\gamma < \gamma + 1\),
and \(\word[\ccN](3^\nIl) = \word<\ccN>(3^\nIl) = 1\),
it follows from Lemma \ref{org06db361} that
\(\Cdn[\mGen{\nIl}](\mvecC[\gamma]) \subseteq \ZZ_{> \ccN}\).
Lemma \ref{org85cba49} shows that 
\begin{equation}
\label{equ:thm-5-linear-case-2}
 \ithComp{\ccz}<\ccL> \neq \gamma \tfor \ccL \in \Cdn[\mGen{\nIl - 1}](\mvecC) \cap \ZZ_{\le \ccN}.
\end{equation}
Since \(\ithComp{\ccy}<\ccM> = \word<\ccN>(3^\nIh) + \word<\ccN>(3^\nIl) = 1 + \gamma + 1 = \gamma + 2\),
it follows that \(\ithComp{\ccz}<\ccM> = \ithComp{\ccy}<\ccM> - \word<\ccM>(3^\nIh) = \gamma + 2 - 2 = \gamma\).
However \(\ccM \in \Cdn[\mGen{\nIl - 1}](\mvecC)\) and \(\ccM < \ccN\), contrary to \eqref{equ:thm-5-linear-case-2}.
 
\end{mycase}

\begin{mycase}[\(\ccN = \xi\)]
 \comment{Case. [\(\ccN = \xi\)]}
\label{sec:org0a15ab8}

Since \(\ithComp{\ccy}<\ccN> = \beta < \alpha\), we see that \(\ithComp{\ccy}<\ccN> = 0, 1\).

\noindent
\textbf{Case 3.1 (\(\ithComp{\ccy}<\ccN> = 1\)).}

We first show that
\begin{equation}
\label{equ:equ-thm-5-linear-case3-1-1}
 \Cdn[\mGen{\nIl}](\mvecC[0]) \setminus \set{\ccN - 1} \subseteq \Cdn[\mGen{\nIl - 1}[\bdz]](\mvecC[0]).
\end{equation}
Note that \(\ithComp{\ccz}<\ccN> = \ithComp{\ccy}<\ccN> - \word<\ccN>(3^{\nIh}) = 1 - 1 = 0\).
By Lemma \ref{org85cba49}, we see that
\(\Cdn[\mGen{\nIl}](\mvecC[0]) \not \subseteq \ZZ_{> \ccN}\).
Since \(\xi = N \in \Cdn[\mGen{\nIl}](\mvecC[1])\), it follows from Lemma \ref{org69880d5} that
\(\xi < \eta\) and \(\Cdn[\mGen{\nIl}](\mvecC[0]) \setminus \set{\ccN - 1} \subseteq \ZZ_{> \ccN}\).
This implies that if \(\cci \in \Cdn[\mGen{\nIl}](\mvecC[0]) \setminus \set{\ccN - 1}\),
then \(\cci > \ccN\), and hence \(\ithComp{\ccz}<\cci> = \word<\cci>(\ccp^{\nIl}) = 0\).
Therefore \eqref{equ:equ-thm-5-linear-case3-1-1} holds.

We next show that
\begin{equation}
\label{equ:equ-thm-5-linear-case3-1-2}
 \set{\xi, \eta} \subseteq \Cdn[\mGen{\nIl - 1}[\bdz]](\mvecC[0]).
\end{equation}
Since \(\eta > \xi\),
we see that \(\mGen{\nIl}<\eta> = \mGen{\nIl}<\xi> = \mvecC[1]\), that is,
\begin{equation}
\label{equ:equ-thm-5-linear-case3-1-3}
 \eta, \xi \in \Cdn[\mGen{\nIl}](\mvecC[1]).
\end{equation}
Since \(\eta > \xi = \ccN\), it follows that that \(\ithComp{\ccz}[\eta] = \word[\eta](3^{\nIl}) = 0\).
Hence \(\ithComp{\ccz}<\eta> = \ithComp{\ccz}<\xi> + \ithComp{\ccz}[\eta] = 0 + 0 = 0\).
Therefore \eqref{equ:equ-thm-5-linear-case3-1-2} holds.

From \eqref{equ:equ-thm-5-linear-case3-1-1} -- \eqref{equ:equ-thm-5-linear-case3-1-3},
we conclude that 
\[
\cdn[\mGen{\nIl - 1}[\bdz]](\mvecC[0]) \ge 2 + \cdn[\mGen{\nIl}](\mvecC[0]) - 1 = 3^{\nIk - \nIl - 1} \ccd' + 1,
\]
contrary to Corollaries \ref{org7f53679} and \ref{orgb141d90}.

\noindent
\textbf{Case 3.2} (\(\ithComp{\ccy}<\ccN> = 0\)).

\resetstep
 
 \begin{step}
 \comment{Step.}
\label{sec:org3a4a08d}

We show that \(\xi < \eta\) and
\begin{equation}
\label{equ:thm-ternarly-linear-z}
\Cdn[\mGen{\nIl}](\mvecC[0]) \setminus \set{\ccN - 1} \subseteq \ccZ_{> \ccN}.
\end{equation}
Let \(\gamma = \word<\ccM>(3^{\nIl})\).
Since \(\ccM \in \Cdn[\mGen{\nIh}](\mvecA[2])\), 
we see that
\(\gamma = 2\) when
if \(\nIl = \nIh\).
Moreover,
since \(\ccN \in \Cdn[\mGen{\nIh}](\mvecA[2])\), 
it follows that
if \(\nIl < \nIh\), then
\(\gamma = \word<\ccM>(3^{\nIl}) = \word<\ccN>(3^{\nIl}) = 1\).
Therefore
\[
 \gamma = \begin{cases} 2 & \tif \nIl = \nIh \\ 1 & \tif \nIl < \nIh. \end{cases}
\]
Note that
\begin{align*}
 \ithComp{\ccz}<\ccM> &= \ithComp{\ccy}<\ccM> - \word<\ccM>(3^\nIh) \\
&= \word<\ccN>(3^\nIh) + \word<\ccN>(3^\nIl) - 2 \\
&= 1 + 1 - 2 = 0.
\end{align*}
Thus \(\ccM \in \Cdn[\mGen{\nIl - 1}[\bdz]](\mvecC[0])\).
It follows that \(\Cdn[\mGen{\nIl - 1}[\bdz]](\mvecC[0]) \neq \Cdn[\mGen{\nIl}](\mvecC[0])\)
since \(\ccM \not \in \Cdn[\mGen{\nIl}](\mvecC[0])\).
By Lemma \ref{org85cba49}, we see that
\(\Cdn[\mGen{\nIl}](\mvecC[0]) \not \subseteq \ZZ_{> \ccN}\).
Thus there exists \(\ccL \in \Cdn[\mGen{\nIl}](\mvecC[0])\) such that \(\ccL < \ccN\).
Lemma \ref{org69880d5} shows that \(\ccL = \ccN - 1\) and \(\xi < \eta\). 
Moreover, \eqref{equ:thm-ternarly-linear-z} holds.
 
\end{step} 
 
\end{mycase}

 \begin{step}
 \comment{Step.}
\label{sec:orgcc99721}

We show that (1) \(\xi, \eta \in \Cdn[\mGen{\nIh - 1}[\bdy]](\mvecA[0])\) and (2) \(\mvecA \neq \allzero^{\marray{\nIh - 1}}\)

(1) Since \(\ccN = \xi < \eta\), we see that \(\mGen{\nIh}<\xi> = \mGen{\nIh}<\eta> = \mvecA[1]\).
Moreover, \(\ithComp{\ccy}[\eta] = \word[\eta](3^{\nIh}) + \word[\eta](3^{\nIl}) = 0 + 0\),
and hence \(\ithComp{\ccy}<\eta> = \ithComp{\ccy}<\xi> = 0\).
Therefore \(\xi, \eta \in \Cdn[\mGen{\nIh - 1}[\bdy]](\mvecA[0])\).

(2) If \(\mvecA = \allzero^{\marray{\nIh - 1}}\), then \(\ccN \in \Cdn[\mGen{\nIh - 1}[\bdy]](\allzero{\nIh})\) and
\(\ccN \not \in \msupp|\ccE|(\mLex-{\nIh - 1}[\bdy])\), contrary to \(\ccN \in \Cdn[\mGen{\nIh}](\mEvec{\nIh}) \subseteq \msupp|\ccE|(\mLex-{\nIh}) = \msupp|\ccE|(\mLex-{\nIh - 1}[\bdy])\).
Therefore \(\mvecA \neq \allzero^{\marray{\nIh - 1}}\).
 
\end{step}

 \begin{step}
 \comment{Step.}
\label{sec:org04304bc}
We show that \(\Cdn[\mGen{\nIh}](\mvecA[0]) \subseteq \ZZ_{> \ccN}\).

Assume that there exists \(\ccL' \in \Cdn[\mGen{\nIh}](\mvecA[0]) \cap \ZZ_{\le \ccN}\).
Since \(\ccN \in \Cdn[\mGen{\nIh}](\mvecA[1])\), we see that \(\ccL' < \ccN\).
Lemma \ref{org69880d5} shows that \(\ccL' = \ccN - 1\), and hence \(\ccL' = \ccL \in \Cdn[\mGen{\nIl}](\mvecC[0])\).
Assume that \(\nIl < \nIh\). Then
\(\word<\ccL>(3^\nIl) = \mvecA<\nIl> = \word<\ccN>(3^\nIl) = 1\).
However, since \(\mGen{\nIl}<\ccL>  = \mvecC[0]\), we see that \(\word<\ccL>(3^\nIl) = 0\), a contradiction.
Thus \(\nIl = \nIh\).
Therefore \(\Cdn[\mGen{\nIh}](\mvecA[0]) \setminus \set{\ccN - 1} \cup \set{\xi, \eta} \subseteq \Cdn[\mGen{\nIh - 1}[\bdy]](\mvecA[0])\),
contrary to Corollaries \ref{org7f53679} and \ref{orgb141d90}.
Therefore \(\Cdn[\mGen{\nIh}](\mvecA[0]) \subseteq \ZZ_{> \ccN}\).
 
\end{step}

 \begin{step}
 \comment{Step.}
\label{sec:org90f75df}

We show that \(\nIl < \nIh\) and
\begin{equation}
\label{equ:thm-ternarly-linear-lex-size-over2}
\Size{\Cdn[\mGen{\nIh - 1}[\bdy]](\mvecA[2]) \cap \Cdn[\mGen{\nIh}](\mvecA[2])} \ge 2.
\end{equation}

From Step 3 and Lemma \ref{org85cba49}, we see that
\(\ithComp{\ccy}<\ccN> \neq \gamma\), where \(\gamma\) is the \(\nIl\)-th component of \(\mvecA[0]\).
Since \(\ithComp{\ccy}<\ccN> = 0\) and the \(\nIh\)-th component of \(\mvecA[0]\) is 0, we see that
\(\nIl < \nIh\) and \(\gamma = \word<\ccN>(3^{\nIl}) = 1\).
Lemma \ref{org85cba49} shows that
\(\Cdn[\mGen{\nIh - 1}[\bdy]](\mvecA[1]) = \Cdn[\mGen{\nIh}](\mvecA[0])\).
Therefore
\begin{equation}
\label{equ:thm-ternarly-linear-lex-0112}
\begin{array}{ccc}
\Cdn[\mGen{\nIh - 1}](\mvecA) \setminus \Cdn[\mGen{\nIh - 1}[\bdy]](\mvecA[1]) & = &
\Cdn[\mGen{\nIh - 1}](\mvecA) \setminus\Cdn[\mGen{\nIh}](\mvecA[0])\\
\text{\rotatebox{90}{$=$}}& & \text{\rotatebox{90}{$=$}} \\
\Cdn[\mGen{\nIh - 1}[\bdy]](\mvecA[0]) \cup \Cdn[\mGen{\nIh - 1}[\bdy]](\mvecA[2]) & & \Cdn[\mGen{\nIh}](\mvecA[1]) \cup \Cdn[\mGen{\nIh}](\mvecA[2]).
\end{array}
\end{equation}
From Step 2, we know that \(\eta, \xi \in \Cdn[\mGen{\nIh - 1}[\bdy]](\mvecA[0])\).
Since \(\eta, \xi \in \Cdn[\mGen{\nIh}](\mvecA[1])\),
it follows from 
\eqref{equ:thm-ternarly-linear-lex-0112} that
\eqref{equ:thm-ternarly-linear-lex-size-over2} holds.
 
\end{step}

 \begin{step}
 \comment{Step.}
\label{sec:org6246034}
We show that \(\mLex{\nIk - 1}\) is a linear code.

Since \(\mGen{\nIh}<\ccM> = \mvecA[2]\)
and \(\ithComp{\bdy}<\ccM> = \word<\ccN>(3^{\nIh}) + \word<\ccN>(3^{\nIl}) = 2\), it follows that
\[\ccM \in \Cdn[\mGen{\nIh - 1}[\bdy]](\mvecA[2]) \cap \Cdn[\mGen{\nIh}](\mvecA[2]).\]
By \eqref{equ:thm-ternarly-linear-lex-size-over2},
we see that there exists \(\ccM'\) such that \(\ccM' \neq \ccM\) and 
\[\ccM' \in \Cdn[\mGen{\nIh}](\mvecA[2]) \cap \Cdn[\mGen{\nIh - 1}[\bdy]](\mvecA[2]).
\]
Note that
\begin{align*}
 \ithComp{\ccz}<\ccM> &= \ithComp{\ccy}<\ccM> - \word<\ccM>(3^{\nIh}) = 2 - 2 = 0,\\
 \ithComp{\ccz}<\ccM'> &= \ithComp{\ccy}<\ccM'> - \word<\ccM'>(3^{\nIh}) = 2 - 2 = 0.
\end{align*}
This implies that \(\ccM, \ccM' \in \Cdn[\mGen{\nIl - 1}[\bdz]](\mvecC[0])\).

We show that \(\Cdn[\mGen{\nIl}](\mvecC[0]) \setminus \set{\ccN - 1} \subseteq \Cdn[\mGen{\nIl - 1}[\bdz]](\mvecC[0])\).
Let \(\ccL \in \Cdn[\mGen{\nIl}](\mvecC[0]) \setminus \set{\ccN - 1}\).
By Step 1,  \(\ccL > \ccN\).
Since \(\mGen{\nIl}<\eta> = \mGen{\nIl}<\xi> = \mGen{\nIl}<\ccN> = \mvecC[1]\), it follows that \(\ccL \neq \eta\).
Thus
\(\ithComp{\ccz}<\ccL> = \ithComp{\ccz}[\ccL] = \ithComp{\ccy}[\ccL] - \word[\ccL](3^{\nIh}) =
\word[\ccL](3^{\nIh}) + \word[\ccL](3^{\nIl}) - \word[\ccL](3^{\nIh}) = \word[\ccL](3^{\nIl})\).
Since \(\mGen{\nIl}<\ccL> = \mvecC[0]\), it follows that \(\word[\ccL](3^{\nIl}) = 0\), and that
\(\ithComp{\ccz}<\ccL> = \word[\ccL](3^{\nIl}) = 0\).
Therefore 
\[\Bigl(\Cdn[\mGen{\nIl}](\mvecC[0]) \setminus \set{\ccN - 1} \Bigr) \cup \set{\ccM, \ccM'} \subseteq \Cdn[\mGen{\nIl - 1}[\bdz]](\mvecC[0]).\]
Since \(\ccM, \ccM' \in \Cdn[\mGen{\nIl}](\mvecC[1])\),
it follows that \(\cdn[\mGen{\nIl - 1}[\bdz]](\mvecC[0]) \ge \cdn[\mGen{\nIl}](\mvecC[0]) - 1 + 2 = 3^{\nIk - \nIl - 1} + 1\) , a contradiction.
Therefore we conclude that \(\mLex{\nIk - 1}\) is a linear code.\qedhere
 
\end{step} 
\end{proof}

\section{Solomon-Stiffler codes}
\label{sec:orgcff722a}
\label{org67955c0}
Let \(\ccp = 3\) and \(\ccd = 3^{\nIk - 1} \ccd' = 3^{\nIk - 1}(3 \mQuo + \mRem)\),
where \(0 < \mRem < 3\).
In the next section,
we prove that \(\word(3^{\nIk} + 1) \neq \word(3^{\nIk}) + \word(1)\) when \(\mQuo + \mRem \ge 2\); in particular,
\(\mLex{\nIk}\) is not a linear code.
To this end, we study \(\word(3^{\nIk})\) in this section. Note that \(\word(\cca) = \word-(\cca)\) for \(0 \le a \le 3^{\nIk}\).

\subsection{\texorpdfstring{\(\mDu{\ccd}\)-distributed matrices}{pi-d-distributed matrices}}
\label{sec:org42a1662}

\begin{example}
 \comment{Exm.}
\label{sec:org7ba1c31}
\label{orga22878d}
Let \(\ccp = 3\), \(\ccF = \ccE\), and \(\ccd = 15\). Then
\[
 \mGen{2} = \mGen-{2} =  \begin{bmatrix}
 1  1  1  1  1  &
 1  1  1  1  1  &
 1  1  1  1  1  &
 0  0  0  0  0  &
 0  0  0  0  0 \cdots\\
 2  2  2  2  2  &
 1  1  1  1  1  &
 0  0  0  0  0  &
 1  1  1  1  1  &
 0  0  0  0  0 \cdots\\
 2  2  1  1  0  &
 2  2  1  0  0  &
 2  1  1  0  0  &
 2  1  1  0  0  &
 1  1  0  0  0 \cdots\\
 \end{bmatrix}.
\]
As we have shown in Section \ref{orgd92423c},
\(\cdn[\mGen{1}](\mvecA) = 5\) for \(\mvecA \in \mVct{1}\).
Observe that
\[
 \cdn[\mGen{2}](\mVec{1 \\ 2 \\ 2}) = 
 \cdn[\mGen{2}](\mVec{1 \\ 2 \\ 1}) = 2 \ \tand \
 \cdn[\mGen{2}](\mVec{1 \\ 2 \\ 0}) = 1.
\]
Actually, for \(\mvecA \in \mVct{2}\),
\[
\cdn[\mGen{2}](\mvecA) = 
\begin{cases}
 1 & \tif  \innerproduct{\mvecA}{\allone} = 0,\\
 2 & \tif  \innerproduct{\mvecA}{\allone} \neq 0.\\
\end{cases}
\]
 
\end{example}

\comment{Connect}
\label{sec:org0ed26f3}
Based on the above example, we introduce the following notation.
Let \(\ccd = \ccp^{\nIk - 1} \ccd' = \ccp^{\nIk - 1}(\ccp \mQuo + \mRem)\),
where \(0 \le \mRem < \ccp\). 
For \(0 \le \nIl \le \nIk\) and 
\(\mvecA \in \FF_p^{\marray{\nIl}} \setminus \set{\allzero}\),
we define
\(\mDu{\ccd}(\mvecA)\) as follows:
if \(\mvecA \not \in \mVct{\nIl}\), then \(\mDu{\ccd}(\mvecA) = 0\);
if \(\mvecA \in \mVct{\nIl}\), then
\begin{align*}
 \mDu{\ccd}(\mvecA) = 
 \begin{cases}
 \left \lceil \dfrac{\ccd}{\ccp^\nIl} \right \rceil & \tif  \nIl < \nIk\ \text{or}\ \innerproduct{\mvecA}{\allone} \neq 0,\\
  \mQuo + \mRem - \ccp + 1 & \tif \nIl = \nIk \text{\ and\ } \innerproduct{\mvecA}{\allone}  = 0. \\
 \end{cases}
\end{align*}
Note that if \(\mvecA \in \mVct{\nIk - 1}\), then
\[
 \sum_{\alpha \in \FF_\ccp} \mDu{\ccd}(\mvecA[\alpha]) = \mQuo + \mRem - \ccp + 1 + (\ccp - 1) \left \lceil \dfrac{\ccd}{\ccp^\nIk} \right \rceil = pq + r = \mDu{\ccd}(\mvecA). 
\]
Let \(\bdx_i \in \FF_\ccp^\NN\) for \(0 \le \cci \le \nIl\) and
\(\mnGen = \mVec{\bdx_0 \\ \bdx_1 \\ \svdots \\ \bdx_\nIl} \in \FF_{\ccp}^{\marray{\nIl} \times \NN}\).
For \(\ccW \subseteq \FF_\ccp^{\marray{\nIl}} \setminus \set{\allzero}\),
the matrix \(\mnGen\) is said to be \emph{\(\mDu{\ccd}\)-distributed} on \(\ccW\) if
\[
 \cdn[\mnGen](\mvecA) = \mDu{\ccd}(\mvecA) \tfor
\mvecA \in \ccW.
\]
If \(\mnGen\) is \(\mDu{\ccd}\)-distributed on \(\FF_\ccp^{\marray{\nIl}} \setminus \set{\allzero}\),
then \(\mnGen\) is said to be \(\mDu{\ccd}\)-distributed.

\begin{example}
 \comment{Exm.}
\label{sec:orga83ee08}
\label{org73b0799}
Let \(\ccp = 3\) and \(\ccd = 15 = 3 \cdot (3 \cdot 1 + 2)\). Then \(\mQuo = 1\) and \(\mRem = 2\),
so \(\mQuo + 1 = 2\) and \(\mQuo + \mRem - \ccp + 1 = 1\).
Thus \(\mGen{2}\) is \(\mDu{15}\)-distributed.

\comment{Connect}
\label{sec:org253d78d}
Our goal is to show the next result.
 
\end{example}

\begin{proposition}
 \comment{Prop.}
\label{sec:org252dd97}
\label{org6d2e6f0}

Let \(\ccd = \ccp^{\nIk - 1} (\ccp \mQuo + \mRem)\).
If \(\mQuo + \mRem - \ccp + 1 \ge 2\), then \(\mGen-{\nIk}\) is \(\mDu{\ccd}\)-distributed.
 
\end{proposition}

\comment{connect}
\label{sec:orgd1e3c73}
Note that a \(\mDu{\ccd}\)-distributed matrix generates a code with minimum distance \(\ccd\)
because this code is a Solomon-Stiffler code\cite{solomon-algebraically-1965}. Thus the following lemma holds.

\begin{lemma}
 \comment{Lem.}
\label{sec:orgec97e95}
\label{org65ab687}
Let \(\ccd = \ccp^{\nIk - 1} \ccd' = \ccp^{\nIk - 1}(\ccp \mQuo + \mRem)\),
\(0 \le \nIl \le \nIk\), 
and \(\mnGen = \mVec{\bdx_0 \\ \bdx_1 \\ \svdots \\ \bdx_\nIl} \in \FF_{\ccp}^{\marray{\nIl} \times \NN}\).
Let \(\cC\) be the code generated by the rows of \(\mnGen\).
If \(\mnGen\) is \(\mDu{\ccd}\)-distributed, then the minimum distance of \(\cC\) equals \(\ccd\)
and \(\wt(\bdx_\cci) = \ccd\) for \(0 \le \cci \le \nIl\).
 
\end{lemma}

\subsection{Restriction}
\label{sec:orge450481}
\comment{Connect}
\label{sec:orgc85a656}
For \(\Phi \subseteq \NN\) and \(\bdy \in \FF_\ccp^\NN\), let
\[
\restWord{\bdy}[\Phi; \ccF] = \sum_{\nIi \in \Phi} \ithComp{\ccy}[\nIi; \ccF] \restWord{(\bdf_\nIi)}<\Phi> \in \FF_\ccp^\Phi.
\]
For simplicity, we write \(\restWord{\bdy}[\Phi]\) instead of \(\restWord{\bdy}[\Phi; \ccF]\).
For \(\nIh \in \NN\), let
\begin{equation}
\label{equ:def-Phi_h}
 \mPhi_{\nIh} = \NN \setminus \msupp|\ccE|(\word(\ccp^{\nIh})) = \Set{\nIi \in \NN \ : \ \word-<\nIi>(\ccp^{\nIh}) = 0}.
\end{equation}
The next lemma allows an inductive approach.

\begin{lemma}
 \comment{Lem.}
\label{sec:orge1ce561}
\label{org5683742}
Let \(\nIh \in \NN\).
If \(\cdn[\mGen-{\nIh}](\mvecA) \ge 2\) for \(\mvecA \in \mVct{\nIh}\), 
then \(\restWord{\bdy}<\mPhi_{\nIh}> = \restWord{\bdy}[\mPhi_{\nIh}]\) for \(\bdy \in \FF_\ccp^\NN\).
 
\end{lemma}

\begin{proof}
 \comment{Proof.}
\label{sec:org8357fb6}
Let \(\Phi = \mPhi_{\nIh}\).
It suffices to show that
\(\ithComp{(\restWord{\bdy}<\Phi>)}<\nIi> = \ithComp{(\restWord{\bdy}[\Phi])}<\nIi>\)
for \(\nIi \in \Phi\).
By definition,
\begin{align*}
\ithComp{(\restWord{\bdy}<\Phi>)}<\nIi>
 = \ithComp{\ccy}<\nIi> &= \begin{cases}
 \ithComp{\ccy}[\eta] + \ithComp{\ccy}[\xi] & \tif \nIi = \eta, \\
 \ithComp{\ccy}[\nIi] & \totherwise, \\
 \end{cases}\\
\ithComp{(\restWord{\bdy}[\Phi])}<\nIi>
 &= \begin{cases}
 \ithComp{\ccy}[\eta] + \ithComp{\ccy}[\xi] & \tif \nIi = \eta \tand \xi \in \Phi, \\
 \ithComp{\ccy}[\nIi] & \totherwise. \\
 \end{cases}
\end{align*}
Thus we may assume that \(\nIi = \eta\) and \(\xi \not \in \Phi\).
Since \(\nIi \in \Phi\) and \(\xi \not \in \Phi\), it follows that \(\word-<\eta>(\ccp^{\nIh}) = 0\) and
\(\word-<\xi>(\ccp^{\nIh}) \neq 0\).
In particular, \(\word-<\eta>(\ccp^{\nIh}) < \word-<\xi>(\ccp^{\nIh})\), which contradicts Lemma~\ref{org45cd32d}.
\end{proof}

\comment{connect}
\label{sec:orge651416}
For \(\Phi \subseteq \NN\), let
\[
 \restWord{\ccF}<\Phi> = (\restWord{\bdf_\cci}<\Phi>)_{\cci \in \Phi}. 
\]
Then \(\restWord{\ccF}<\Phi>\) is a basis of \(\FF_\ccp^\Phi\) since, for \(\cci \in \Phi\),
\[
\restWord{\bdf_\nIi}<\Phi> =
\begin{cases}
\restWord{\bde_\eta}<\Phi> + \restWord{\bde_\xi}<\Phi> & \tif \cci = \eta,\\
\restWord{\bde_\nIi}<\Phi> & \totherwise.
\end{cases}
\]
When \(\Phi = \{i_j \ : \ j \in \NN\}\) with \(i_0 < i_1 < \cdots\),
identifying \(i_j\) with \(j\) allows us to treat \(\restWord{\ccF}<\Phi>\) as an element of \(\cF\).

\subsection{Outline of proof of Proposition \texorpdfstring{\ref{org6d2e6f0}}{6.3}}
\label{sec:org086b474}
We will prove the proposition by induction on \(\nIk\). By the induction hypothesis,

\begin{description}
  \item[($\text{D}_{\nIk}$)] for $0 \le \nIl \le \nIk - 1$ and $\tilde{\ccF} \in \cF$, 
if $\ord_{\ccp}(\tilde{\ccd}) \ge \nIl - 1$, then the matrix $\mGen-|\tilde{\ccF}, \tilde{\ccd}|{\nIl}$ is \(\mDu{\tilde{\ccd}}\)-distributed.
\end{description}

\comment{Connect}
\label{sec:org1d08593}
For \(\nIm \le \nIl\) and \(\mvecB \in \FF_p^{\marray{\nIm}}\),
let
\[
 \mVct{\nIl}(\mvecB) = \set{\mvecA \in \mVct{\nIl} : \mvecA<\nIi> = \mvecB<\nIi> \tfor 0 \le \nIi \le \nIm}.
\]
For example, \(\mVct{1}(0) = \mVct{1}([0]) = \Set{\begin{bsmallmatrix} \mvecA<0> \\ \mvecA<1> \end{bsmallmatrix} \in \mVct{1}  :  \mvecA<0> = 0} = \Set{\begin{bsmallmatrix} 0 \\ 1 \end{bsmallmatrix}}\).
Let 
\begin{align*}
\ccF^* &= \restWord{\ccF}<\Phi_0>, \ \ \ccE^* = \restWord{\ccE}<\Phi_0>,\ \ \ccd^* = \left \lceil\frac{\ccd}{\ccp} \right \rceil.\\
\word*(\cca) &= \word|\ccF^*, \ccd^*|(\cca) \ \ \tand \ \ 
\word-*(\cca) = \word-|\ccF^*, \ccd^*|(\cca).
\end{align*}
We first show that
\[
 \restWord{\word-(\ccp^{\nIl})}<\Phi_0> = \word-*(\ccp^{\nIl - 1})
 \tfor 0 \le \nIl \le \nIk
\]
in Lemma \ref{org7ac4c8a}. This implies that \(\mGen-{\nIk}\) is \(\mDu{\ccd}\)-distributed on
\(\mVct{\nIk}(0)\).
By using the restriction to \(\Phi_1\),
we next show that \(\mGen-{\nIk}\) is \(\mDu{\ccd}\)-distributed on
\(\mVct{\nIk}(\begin{bsmallmatrix} 1 \\ 0 \end{bsmallmatrix})\).
Using these results, we finally prove that 
\(\mGen-{\nIk}\) is \(\mDu{\ccd}\)-distributed when \(\ccp = 3\).

\subsection{\texorpdfstring{\(\mGen-{\nIk}\) is \(\mDu{\ccd}\)-distributed on \(\mVct{\nIk}(0)\)}{Lambda-k is pi-d distributed on Vk(0)}}
\label{sec:org7f7e76f}

For \(\ccW \subseteq \FF_\ccp^{\marray{\nIm}}\),
let 
\[
 \mVct{\nIl}(\ccW) = \bigcup_{\mvecB \in \ccW} \mVct{\nIl}(\mvecB)
\quad \tand \quad
\Cdn[\mnGen](\ccW) = \bigcup_{\mvecB \in \ccW} \Cdn[\mnGen](\mvecB).
\]

\begin{lemma}
 \comment{Lem.}
\label{sec:org17fe70f}
\label{org113beae}
Let \(1 \le \nIl \le \cck\) and \(\mnGen \in  \FF_{\ccp}^{\marray{\nIl} \times \NN}\).
Suppose that \(\mQuo + \mRem - \ccp + 1 \ge 2\).
Let \(\ccW\) be a subset of \(\mVct*{\nIl - 1}\) such that \(\allzero \in \ccW\).
If \(\mnGen\) is \(\mDu{\ccd}\)-distributed and
\(\mnGen[\bdy]\) is \(\mDu{\ccd}\)-distributed on \(\mVct{\nIl}(\ccW)\), then
there exists \(\tilde{\bdy}\) satisfying the following two conditions.

\begin{enumerate}
\item \(\mnGen[\tilde{\bdy}]\) is \(\mDu{\ccd}\)-distributed.
\item \(\supp_{\ccF}(\tilde{\bdy} - \bdy) = \supp_{\ccE}(\tilde{\bdy} - \bdy) \subseteq \NN \setminus (\Theta \cup \Cdn[\mnGen](\ccW))\);
in particular, \(\restWord{\tilde{\bdy}}<\Cdn[\mnGen](\ccW)> = \restWord{\bdy}<\Cdn[\mnGen](\ccW)>\).
\end{enumerate}
 
\end{lemma}

\begin{proof}
 \comment{Proof.}
\label{sec:orgd899396}
Let
\[
 \ccn = \# \Set{\mvecA \in \FF_\ccp^{\marray{\nIl - 1}} : \begin{matrix*}[l] \mvecA[\alpha] \neq \allzero \tand\\  \cdn[\mnGen[\bdy]](\mvecA[\alpha]) \neq \mDu{\ccd}(\mvecA[\alpha]) \text{\ for some\ } \alpha \in \FF_p \end{matrix*}}.
\]
We show by induction on \(\ccn\).
If \(\ccn = 0\), then \(\bdy\) satisfies the condition.
Suppose that \(\ccn > 0\) and
\(\cdn[\mnGen[\bdy]](\mvecA[\alpha]) \neq \mDu{\ccd}(\mvecA[\alpha])\).

We first show that \(\mvecA \in \mVct{\ccl - 1}\).
Since \(\mvecA \not \in \ccW\), we see that \(\mvecA \neq \allzero\).
Moreover, \(\cdn[\mnGen](\mvecB) = 0\) for \(\mvecB \not \in \mVct*{\ccl - 1}\) since \(\mnGen\) is \(\mDu{\ccd}\)-distributed,
and hence \(\cdn[\mnGen[\bdy]](\mvecB[\alpha]) = 0 =  \mDu{\ccd}(\mvecB[\alpha])\).
Therefore \(\mvecA \in \mVct{\ccl - 1}\).

Since \(\mnGen\) is \(\mDu{\ccd}\)-distributed,
it follows that \(\cdn[\mnGen](\mvecA) = \ccp^{\cck - \ccl} (\ccp \mQuo + \mRem)\).
Note that
\(\mQuo + \mRem - \ccp + 1 \ge 2\) and \(\mQuo + 1 \ge 2\).
It follows from Lemma \ref{org2f35cc7} that there exists \(\bdz\) satisfying the following two conditions:

\begin{itemize}
\item \(\cdn[\mnGen[\bdz]](\mvecA[\alpha])\ = \mDu{\ccd}(\mvecA[\alpha])\) for \(\alpha \in \FF_p\).
\item \(\supp_{\ccF}(\bdz - \bdy) = \supp_{\ccE}(\bdz - \bdy) \subseteq \Cdn[\mnGen](\mvecA) \setminus \Theta\).
\end{itemize}

\noindent
Since \(\mvecA \not \in \ccW\),
we see that 
\(\restWord{\bdz}<\Cdn[\mnGen](\mvecB)> = \restWord{\bdy}<\Cdn[\mnGen](\mvecB)>\)
for \(\mvecB \in \ccW\);
in particular, \(\cdn[\mnGen[\bdz]](\mvecB[\alpha]) = \cdn[\mnGen[\bdy]](\mvecB[\alpha]) = \mDu{\ccd}(\mvecB[\alpha])\).
Hence \(\mnGen[\bdz]\) is \(\mDu{\ccd}\)-distributed on  \(\mVct{\nIl}(\ccW \cup \set{\mvecA})\).
By the induction hypothesis, there exists \(\tilde{\bdz}\) such that
\(\mnGen[\tilde{\bdz}]\) is \(\mDu{\ccd}\)-distributed and
\begin{align*}
\supp_{\ccF}(\tilde{\bdz} - \bdz) &= \supp_{\ccE}(\tilde{\bdz} - \bdz) \\
 &\subseteq \NN \setminus (\Theta \cup \Cdn[\mnGen](\ccW \cup \set{\mvecA})) \\
 &\subseteq \NN \setminus (\Theta \cup \Cdn[\mnGen](\ccW)).
\end{align*}
Since \(\mvecA \not \in \ccW\), it follows that \(\Cdn[\mnGen](\mvecA) \subseteq \NN \setminus \Cdn[\mnGen](\ccW)\), and that
\begin{align*}
 \supp_{\ccF}(\tilde{\bdz} - \bdy) &\subseteq \supp_{\ccF}(\bdz - \bdy) \cup \supp_{\ccF}(\tilde{\bdz} - \bdz) \\
 &\subseteq \bigl(\Cdn[\mnGen](\mvecA) \setminus \Theta \bigr) \cup \bigl(\NN \setminus (\Theta \cup \Cdn[\mnGen](\ccW)) \bigr)\\
 &= \NN \setminus (\Theta \cup \Cdn[\mnGen](\ccW)).
\end{align*}
Thus \(\supp_{\ccF}(\tilde{\bdz} - \bdy) = \supp_{\ccE}(\tilde{\bdz} - \bdy)\).
Therefore \(\tilde{\bdz}\) satisfies the conditions.
\end{proof}

\begin{lemma}
 \comment{Lem.}
\label{sec:org9fe66cd}
\label{org3a86abe}
Let \(1 \le \nIl \le \nIk\) and  \(\mvecA \in \FF_\ccp^{\marray{\nIl - 1}} \setminus \set{\allzero}\).
Suppose that
\(\mvecB\) is a vector in \(\FF_\ccp^{\marray{\nIl}}\) obtained from 
\(\mvecA\) by inserting 0, that is,
\(\mvecB = \transpose{\begin{bsmallmatrix}\mvecA<0> & \cdots &  \mvecA<\cci - 1> &  0 &  \mvecA<\cci> & \cdots & \mvecA<\nIl - 1>\end{bsmallmatrix}}\)
for some \(\cci\).
Then
\[
 \mDu{\ccd^*}(\mvecA) = \mDu{\ccd}(\mvecB).
\]
 
\end{lemma}

\begin{proof}
 \comment{Proof.}
\label{sec:orge2b2a55}
If \(\mvecA \not \in \mVct{\nIl - 1}\),
then \(\mvecB \not \in \mVct{\nIl}\),
and hence \(\mDu{\ccd^*}(\mvecA) = 0 = \mDu{\ccd}(\mvecB)\).
Suppose that \(\mvecA  \in \mVct{\nIl - 1}\).
Then \(\mvecB \in \mVct{\nIl}\).
Since \(\Bigl\lceil \frac{\lceil \frac{\ccd}{\ccp} \rceil }{\ccp^{\nIl - 1}} \Bigr\rceil
= \Bigl\lceil \frac{\ccd}{\ccp^{\nIl}} \Bigr\rceil\), it follows that
\begin{align*}
 \mDu{\ccd^*}(\mvecA) &= 
 \begin{cases}
  \Bigl\lceil \dfrac{\ccd}{\ccp^{\nIl}} \Bigr\rceil & \tif \nIl - 1 < \nIk - 1 \tor \innerproduct{\mvecA}{\allone} \neq 0\\
  \mQuo + \mRem - \ccp + 1 & \totherwise
 \end{cases}\\
 &= \mDu{\ccd}(\mvecB).
\end{align*}
\end{proof}

\begin{lemma}[\(\Phi_0\) lemma]
 \comment{Lem. [\(\Phi_0\) lemma]}
\label{sec:org24ed299}
\label{org7ac4c8a}
Suppose that $\mQuo + \mRem - \ccp + 1 \ge 2$. 
If ($\text{D}_{\nIk}$) holds, then
\[
 \restWord{\word-(\ccp^{\nIl})}<\Phi_0> = \word-*(\ccp^{\nIl - 1})
 \tfor 0 \le \nIl \le \nIk,
\]
where $\word-*(\ccp^{-1}) = \word-*(0)$. Moreover, $\mGen-{\nIk}$ is \(\mDu{\ccd}\)-distributed on $\mVct{\nIk}(0)$.
 
\end{lemma}

\begin{proof}
 \comment{Proof.}
\label{sec:orga33140b}

We show the lemma by induction on $\nIl$.
If $\nIl = 0$, then $\restWord{\word-(\ccp^0)}<\Phi_0> = \allzero = \word-*(0)$.
Suppose that $1 \le \nIl \le \nIk$.

\smallskip
\noindent
\resetstep
 
 \begin{step}
 \comment{Step.}
\label{sec:org3a33333}

We show \(\restWord{\word-(\ccp^{\nIl})}<\Phi_0> \ge_{\ccF^*}  \word-*(\ccp^{\nIl - 1})\).

Since \(\mGen-{\nIk - 1}\) is \(\mDu{\ccd}\)-distributed,
we see that \(\cdn[\mGen-{0}](1) = \ccd\),
so the weight of \(\word-(1)\) equals \(\ccd\). 
This shows that the code generated by
\(\restWord{\word-(\ccp^1)}<\Phi_0>, \ldots, 
\restWord{\word-(\ccp^{\nIl})}<\Phi_0>\),
which is the residual code of \(\mGen-{\nIl}\)
with respect to \(\word-(1)\),
has minimum distance at least \(\ccd^*\)
since 
\[\ccd^* = \left \lceil \frac{\ccd}{\ccp} \right \rceil \le \ccd - \left\lfloor \frac{(\ccp - 1) \ccd}{\ccp} \right \rfloor.
\]
By the induction hypothesis,
\(\restWord{\word-(\ccp^\nIi)}<\Phi_0> = \word-*(\ccp^{\nIi - 1})\) for \(0 \le \nIi \le \nIl - 1\).
Since
\(\mRes(\mLex-{\nIl - 1})\) meets the Griesmer bound,
it follows that
\(\cdn[\mGen-{\nIl}](\mEvec{\nIl}) = \size{\msupp|\ccE|(\mLex-{\nIl})} - \size{\msupp|\ccE|(\mLex-{\nIl - 1})} \ge 1\),
and hence \(\cdn[\mGen-{\nIl}](\mEvec{\nIl}) \ge 1\).
This implies that \(\restWord{\word-(\ccp^{\nIl + 1})}<\Phi_0>\) 
is not generated by \(\word-*(\ccp^\nIi)\) for \(0 \le \nIi \le \nIl - 2\).
Thus \(\restWord{\word-(\ccp^{\nIl})}<\Phi_0> \neq \word-*(\cca)\) (\(0 \le \cca < \ccp^{\nIl - 1}\)).
Therefore \(\restWord{\word-(\ccp^{\nIl})}<\Phi_0> \ge_{\ccF^*}  \word-*(\ccp^{\nIl - 1})\).

\smallskip
Let
\[\bdy =  \sum_{\cci \in \Phi_0} \word-*[\cci](\ccp^{\nIl - 1}) \bdf_{\cci}.\]
Lemma \ref{org5683742} shows that
\[
\restWord{\bdy}<\Phi_0> = \restWord{\bdy}[\Phi_0] = \word-*(\ccp^{\nIl - 1}).
\]
 
\end{step}

 \begin{step}
 \comment{Step.}
\label{sec:org2ef3526}

We show that \(\mGen-{\nIl - 1}[\bdy]\) is \(\mDu{\ccd}\)-distributed on \(\mVct{\nIl}(0)\);
note that this implies that we can apply Lemma \ref{org113beae} to \(\mGen-{\nIl - 1}[\bdy]\)
since \(\mGen-{\nIl - 1}\) is \(\mDu{\ccd}\)-distributed.

Let \(\mGen-*{\nIl - 1} = \mGen-|\ccF^*, \ccd^*|{\nIl - 1}\).
It follows from (\(\text{D}_{\nIk}\)) that
\(\mGen-*{\nIl - 1}\) is \(\ccd^*\)-distributed.
Thus for \(\mvecA \in \mVct{\nIl - 1}\),
\[
  \cdn[\mGen-*{\nIl - 1}](\mvecA) =  \mDu{\ccd^*}(\mvecA).
\]
We see that
\[
 \Cdn[\mGen-*{\nIl - 1}](\mvecA) = 
\Cdn[\mGen-{\nIl - 1}[\bdy]](\mVec{0 \\ \mvecA}).
\]
Indeed,
\begin{align*}
\Cdn[\mGen-{\nIl - 1}[\bdy]](\mVec{0 \\ \mvecA})
&= \Set{\nIi \in \NN : 
  \begin{bsmallmatrix} \word-<\nIi>(1) & \word-<\nIi>(\ccp) & \cdots & \word-<\nIi>(\ccp^{\nIl - 1}) & \ithComp{\ccy}<\nIi>\end{bsmallmatrix} = 
\begin{bsmallmatrix} 0 & \mvecA<0> &  \cdots & \mvecA<\nIl - 2> &  \mvecA<\nIl - 1>\end{bsmallmatrix}}\\
&= \Set{
\nIi \in \Phi_0 :
\begin{bsmallmatrix} \word-*<\nIi>(1) &  \cdots & \word-*<\nIi>(\ccp^{\nIl - 2}) & \word-*<\nIi>(\ccp^{\nIl - 1}) \end{bsmallmatrix} = 
\begin{bsmallmatrix} \mvecA<0> & \cdots & \mvecA<\nIl - 2> & \mvecA<\nIl - 1> \end{bsmallmatrix} }\\
&= \Cdn[\mGen-*{\nIl - 1}](\mvecA).
\end{align*}
Lemma \ref{org3a86abe} yields
\[
 \cdn[\mGen-{\nIl - 1}[\bdy]](\mVec{0 \\ \mvecA}) = \cdn[\mGen-*{\nIl - 1}](\mvecA) =  \mDu{\ccd^*}(\mvecA)= \mDu{\ccd}(\mVec{0 \\ \mvecA}).
\]
Therefore \(\mGen-{\nIl - 1}[\bdy]\) is \(\mDu{\ccd}\)-distributed on \(\mVct{\nIl}(0)\).
 
\end{step}

 \begin{step}
 \comment{Step.}
\label{sec:org795e398}

We show that \(\restWord{\word-(\ccp^{\nIl})}<\Phi_0> = \word-*(\ccp^{\nIl - 1})\).
By Step 1, \(\restWord{\word-(\ccp^{\nIl})}<\Phi_0> \ge_{\ccF^*} \word-*(\ccp^{\nIl - 1})\).
Assume that \(\restWord{\word-(\ccp^{\nIl})}<\Phi_0> >_{\ccF^*} \word-*(\ccp^{\nIl - 1})\).

It follows from Step 2 and Lemma \ref{org113beae} that there exists \(\bdz\) satisfying the following two conditions:
\begin{enumerate}
\item \(\mGen-{\nIl - 1}[\bdz]\) is \(\mDu{\ccd}\)-distributed.
\item \(\supp_{E}(\bdz - \bdy) = \supp_{F}(\bdz - \bdy) 
 \subseteq \NN \setminus (\Phi_0 \cup \Theta) = 
 \Cdn[\mGen-{0}](1) \setminus \Theta\). 
In particular, \(\restWord{\bdz}<\Phi_0> = \restWord{\bdy}<\Phi_0> = \word-*(\ccp^{\nIl - 1})\).
\end{enumerate}
It follows from Lemma \ref{org65ab687} that
\[\ccd(\bdz, \word-(\cca)) \ge \ccd.
 \quad (0 \le \cca < \ccp^{\nIl})
\]
Since
\(\restWord{\bdz}<\Phi_0> = 
 \word-*(\ccp^{\nIl - 1}) 
 \neq \restWord{\word-(\ccp^{\nIl})}<\Phi_0>\)
it follows that \(\bdz \neq  \word-(\ccp^{\nIl})\).
Thus \(\bdz >_{\ccF} \word-(\ccp^{\nIl})\).
Let
\[
 \ccN = \max \set{\cci \in \NN : \bdz_{[\cci]} \neq \word-[\cci](\ccp^{\nIl})}
\]
and
\[
 \ccM = \max \set{\cci \in \Phi_0 : \ithComp{\ccz}[\cci] \neq \word-[\cci](\ccp^{\nIl})}.
\]
Note that \(\ithComp{\ccz}[\ccN] > \word-[\ccN](\ccp^{\nIl})\).
Moreover, \(\ithComp{\ccz}[\ccM]\relax < \word-[\ccM](\ccp^{\nIl})\) since
\[
 \restWord{\bdz}<\Phi_0>  
 = \word-*(\ccp^{\nIl - 1}) <_{\ccF^*} 
 \restWord{\word-(\ccp^{\nIl})}<\Phi_0>.
\]
In particular, \(\ccM < \ccN\) and \(\ccN \not \in \Phi_0\).
Therefore \(\ccN \in \Cdn[\mGen-{0}](1)\).
Moreover, \(\ccN \in \Cdn[\mGen-{0}](1) \setminus \Theta\).
Indeed, since \(\supp_{\ccF}(\bdy) \subseteq \Phi_0\),
it follows that \(\ithComp{\ccy}[\ccN] = 0\).
Since \(\ithComp{\ccz}[\ccN] \neq 0\),
we see that \(\ccN \in \supp_{\ccF}(\bdz - \bdy) \subseteq \Cdn[\mGen-{0}](1) \setminus \Theta\).
Therefore
\(\ccN \in \Cdn[\mGen-{0}](1) \setminus \Theta\) and
\(\ccM \in \Phi_0\).
It follows from Lemma \ref{org06db361} that \(\ccM > \ccN\), which is impossible.
Therefore \(\restWord{\word-(\ccp^{\nIl})}<\Phi_0> = \word-*(\ccp^{\nIl - 1})\).
 
\end{step} 

\comment{connect}
\label{sec:org7cd2ea4}
Finally, since \(\restWord{\bdy}<\Phi_0> = \word-*(\ccp^{\nIl - 1}) = \restWord{\word-(\ccp^{\nIl})}<\Phi_0>\),
it follows from Step 2 that \(\mGen-{\nIl}\) is \(\mDu{\ccd}\)-distributed on \(\mVct{\nIl}(0)\). 
\end{proof}

\subsection{\texorpdfstring{\(\mGen-{\nIk}\) is \(\mDu{\ccd}\)-distributed on \(\mVct{\nIk}(\mVec{1 \\ 0})\)}{Lambda-k is pi-d distributed on Vk(1 0)}}
\label{sec:org59a5233}
\comment{Connect}
\label{sec:orgb40721f}
We introduce the following notation.
\begin{align*}
 &\ccF^{\sharp} = \restWord{\ccF}<\mPhi_1>,\ \ccE^{\sharp} = \restWord{\ccE}<\mPhi_1>,\ \ccd^{\sharp} = \ccd^* = \left \lceil\frac{\ccd}{\ccp} \right \rceil,\\
 &\word!(\cca) = \word|\ccF^{\sharp}, \ccd^{\sharp}|(\cca),\ \word-!(\cca) = \word-|\ccF^{\sharp}, \ccd^{\sharp}|(\cca), \, \tand\\
 & \mGen!{\nIl} = \mGen|\ccF^{\sharp}, \ccd^{\sharp}|{\nIl}.
\end{align*}
For simplicity, we write \(\Cdn[\mnGen-](\mvecA)\) instead of \(\Cdn[\mGen-{\nIl}](\mvecA)\) for \(\mvecA \in \FF_\ccp^{\marray{\nIl}}\)

\begin{lemma}
 \comment{Lem.}
\label{sec:orgdd18843}
\label{org16d4190}
Let \(\cck \ge 2\). Suppose that \(\mQuo + \mRem - \ccp + 1 \ge 2\). 
If (\(\text{D}_{\nIk}\)) holds, then
\[
 \word-(\ccp^{0})|_{\mPhi_1} = \word-!(\ccp^{0}).
\]

\resetstep
 
\end{lemma}
\begin{proof}
 \comment{Proof.}
\label{sec:orga51dd79}
The proof will be divided into four steps.

 \begin{step}
 \comment{Step.}
\label{sec:orgb831c31}

We show that \(\wt(\restWord{\word-(\ccp^0)}<\mPhi_1>) = \ccd^{\sharp}\); in particular, \(\restWord{\word-(\ccp^0)}<\mPhi_1> \ge_{\ccF^{\sharp}} \word-!(\ccp^{0})\).

Note that \(\msupp|\ccE|(\restWord{\word-(\ccp^0)}<\mPhi_1>) = \Cdn[\mnGen-](\mVec{1 \\ 0})\).
It follows from (\(\text{D}_k\)) that
\(\mGen-{1}\) is \(\mDu{\ccd}\)-distributed, and hence \(\cdn[\mnGen-](\mVec{1 \\ 0}) = \ccd^{\sharp}\).
Therefore \(\wt(\restWord{\word-(\ccp^0)}<\mPhi_1>) = \ccd^{\sharp}\).
This yields \(\restWord{\word-(\ccp^0)}<\mPhi_1> \ge_{\ccF^{\sharp}} \word-(\ccp^{0})\).

Assume that \(\restWord{\word-(\ccp^0)}<\mPhi_1> >_{\ccF^{\sharp}} \word-(\ccp^{0})\). Let
\[
 \bdy = \sum_{\cci \in \NN \setminus \mPhi_1} \word-[\cci](\ccp^0) \bdf_{\cci} + \sum_{\cci \in \mPhi_1} \word-![\cci](\ccp^0) \bdf_{\cci}.
\]
 
\end{step}

 \begin{step}
 \comment{Step.}
\label{sec:org5598968}
We show that \(\wt(\bdy) < \ccd\).
By Lemma \ref{org5683742},
\[
\restWord{\bdy}<\mPhi_1> = \restWord{\bdy}[\mPhi_1] = \sum_{\cci \in \mPhi_1} \word-![\cci](\ccp^0) \restWord{\bdf_{\cci}}<\mPhi_1> = \word-!(\ccp^0).
\]
Note that
\[
 \supp_{\ccF}(\bdy - \word-(\ccp^0)) \subseteq \mPhi_1.
\]
Since \(\restWord{\bdy}<\mPhi_1> = \word-!(\ccp^0) <_{\ccF^{\sharp}} \restWord{\word-(\ccp^0)}<\mPhi_1>\),
it follows that \(\bdy <_{\ccF} \word-(\ccp^0)\). Hence \(\wt(\bdy) < \ccd\). 
 
\end{step}

 \begin{step}
 \comment{Step.}
\label{sec:org9f062e5}

We show that \(\xi \in \mPhi_1\), \(\eta \in \NN \setminus \mPhi_1\),
\(\word-[\xi](\ccp^0) = 1\), and \(\word-![\xi](\ccp^0) = 0\).

Note that \(\wt(\bdy) < \ccd\) and \(\wt(\word-(\ccp^0)) = \ccd\).
From Step 1, \(\wt(\restWord{\word-(\ccp^0)}<\mPhi_1>) = \ccd^{\sharp}\).
Moreover,
\[
 \wt(\restWord{\bdy}<\mPhi_1>) = \wt(\word-!(\ccp^0)) = \ccd^{\sharp}.
\]
Hence \(\restWord{\bdy}<\NN \setminus \mPhi_1> \neq \restWord{\word-(\ccp^0)}<\NN \setminus \mPhi_1>\). 
Since
\(\restWord{\bdy}[\NN \setminus \mPhi_1] = \restWord{\word-(\ccp^0)}[\NN \setminus \mPhi_1]\), it follows that
\[
\restWord{\bdy}<\NN \setminus (\mPhi_1 \cup \set{\eta})> = \restWord{\word-(\ccp^0)}<\NN \setminus (\mPhi_1 \cup \set{\eta})>.
\]
Therefore \(\eta \in \NN \setminus \mPhi_1\), \(\xi \in \mPhi_1\),
\(\ithComp{\ccy}<\eta> = 0\), and \(\word-<\eta>(\ccp^0) = 1\).
We see that 
\[
\ithComp{\ccy}<\eta> = \ithComp{\ccy}[\eta] + \ithComp{\ccy}[\xi] = \word-[\eta](\ccp^0) + \word-![\xi](\ccp^0) = 0 
\]
and
\[
\word-<\eta>(\ccp^0) = \word-[\eta](\ccp^0) + \word-[\xi](\ccp^0) = 1.
\]
Since \(\ithComp{\ccy}[\eta] = \word-[\eta](\ccp^0) \in \set{0, 1}\) and
\(\word-![\xi](\ccp^0), \word-[\xi](\ccp^0) \in \set{0, 1}\),
it follows that \(\word-[\eta](\ccp^0) = \word-![\xi](\ccp^0) = 0\) and \(\word-[\xi](\ccp^0) = 1\).
 
\end{step}

 \begin{step}
 \comment{Step.}
\label{sec:org6c9efa4}

We finally show that \(\restWord{\word-(\ccp^0)}<\mPhi_1> = \word-!(\ccp^{0})\).
Since \(\word-<\xi>(\ccp^0) = 1\) and \(\xi \in \mPhi_1\), it follows that
\(\xi \in \Cdn[\mnGen-](1) \cap \mPhi_1 = \Cdn[\mnGen-](\mVec{1 \\ 0 })\).
Thus \(\word-<\xi>(\ccp) = 0\).
Since \(\word-<\eta>(\ccp^0) = 1\) and \(\eta \not \in \mPhi_1\), we see that
\(\eta \in \Cdn[\mnGen-](\Set{\mVec{1 \\ 1}, \ldots \mVec{1 \\ \ccp - 1}})\).
Note that \(\cdn[\mnGen-](\mvecA) = \ccp^{\nIk - 2} (\ccp \mQuo + \mRem) \ge 2\) for \(\mvecA \in \mVct{1}\) since \(\mGen-{1}\) is \(\mDu{\ccd}\)-distributed.
If \(\xi < \eta\), then Lemma \ref{org45cd32d} shows that \(\word-<\eta>(\ccp) = \word-<\xi>(\ccp) = 0\),
a contradiction. Thus \(\eta < \xi\).
Let
\[\bdz = \bdy + \bdf_{\eta}.
\]
Then \(\bdz <_{\ccF} \word-(\ccp^0)\) 
because \(\ithComp{\ccz}[\xi] = \ithComp{\ccy}[\xi] = 0\), \(\word-[\xi](\ccp^0) = 1\), and \(\xi > \eta\).
Moreover,
\(\restWord{\bdz}<\NN \setminus \mPhi_1> = \restWord{\word-(\ccp^0)}<\NN \setminus \mPhi_1>\)
since \(\restWord{\bdz}<\NN \setminus (\mPhi_1\cup \set{\eta})> =\restWord{\bdy}<\NN \setminus (\mPhi_1 \cup \set{\eta})> =
\restWord{\word-(\ccp^0)}<\NN \setminus (\mPhi_1 \cup \set{\eta})>\)
and \(\ithComp{\ccz}<\eta> = 1 = \word-<\eta>(\ccp^0)\). It follows that
\begin{align*}
 \wt(\bdz) &= \wt(\restWord{\bdz}<\NN \setminus \mPhi_1>) + \wt(\restWord{\bdz}<\mPhi_1>)\\
 &= \wt(\restWord{\word-(\ccp^0)}<\NN \setminus \mPhi_1>) + \wt(\restWord{\bdy}<\mPhi_1>) = \ccd - \ccd^{\sharp} + \ccd^{\sharp} = \ccd.
\end{align*}
Hence \(\bdz \ge_{\ccF} \word(\ccp^0)\), a contradiction.
Therefore \(\restWord{\word-(\ccp^0)}<\mPhi_1> = \word-!(\ccp^{0})\). \(\hspace*{\fill} \qedhere\)
 
\end{step} 
\end{proof}

\comment{connect}
\label{sec:orgac6a6fb}
Let
\[
 \longword-|\ccF, \ccd|(\cca) = \word-|\ccF, \ccd|(\cca).
\]

\begin{lemma}
 \comment{Lem.}
\label{sec:org616cf08}
\label{org2014e3d}
Let \(\cck \ge 2\).
Suppose that \(\mQuo + \mRem - \ccp + 1 \ge 2\). 
If (\(\text{D}_{\nIk}\)) holds, then
\[
 \restWord{\word-(\ccp^{\nIl})}<\mPhi_1> = \word-!(\ccp^{\nIl^-})
\tfor 0 \le \nIl \le \nIk,
\]
where 
\[\nIl^- = \begin{cases} 
 0 & \nIl = 0, \\
 -1 & \nIl = 1, \\
 \nIl - 1 & \nIl \ge 2.
 \end{cases}
\]
 
\end{lemma}

\begin{proof}
 \comment{Proof.}
\label{sec:org34ec5bf}

We show the lemma by induction on \(\nIl\).
If \(\nIl = 0\), then this follows from the previous lemma.
If \(\nIl = 1\), then
\[\restWord{\word-(\ccp^1)}<\mPhi_1> = \allzero = \word-!(0).\]
Suppose that \(2 \le \nIl \le \nIk\).
Then \(\nIl^- = \nIl - 1\).

\resetstep
 
 \begin{step}
 \comment{Step.}
\label{sec:orgc9cf2f6}

We show that \(\restWord{\word-(\ccp^{\nIl})}<\mPhi_1> \ge_{\ccF^{\sharp}} \word-!(\ccp^{\nIl - 1})\).
By Lemma \ref{org5683742}, 
\[
 \restWord{\word-(\ccp^{\nIl})}<\mPhi_1> = \restWord{\word-(\ccp^{\nIl})}[\mPhi_1] = \sum_{\cci \in \mPhi_1} \word-[\cci](\ccp^{\nIl}) \restWord{\bdf_\cci}<\mPhi_1>.
\]
Note that \(\NN \setminus \mPhi_1  = \set{\nIi \in \NN : \word-<\nIi>(\ccp) \neq 0 } = \msupp|\ccE|(\word-(\ccp))\).
Since \(\mGen-{1}\) is \(\mDu{\ccd}\)-distributed, it follows from Lemma \ref{org65ab687} that \(\wt(\word-(\ccp)) = \ccd\).
Therefore the code generated by \(\restWord{\word-(\ccp^0)}<\mPhi_1>\), \(\restWord{\word-(\ccp^2)}<\mPhi_1>\), 
\(\restWord{\word-(\ccp^3)}<\mPhi_1>\), 
\(\ldots\), 
\(\restWord{\word-(\ccp^{\nIl - 1})}<\mPhi_1>\), 
\(\restWord{\word-(\ccp^{\nIl})}<\mPhi_1>\)
has minimum distance at least  \(\ccd^{\sharp}\)
since \(\ccd^\sharp =  \frac{\ccd}{\ccp}  \le \ccd - \left\lfloor \frac{(\ccp - 1) \ccd}{\ccp} \right \rfloor\).
By the induction hypothesis,
we see that
\(\restWord{\word-(\ccp^0)}<\mPhi_1> = \word-!(\ccp^0)\) and \(\restWord{\word-(\ccp^\nIl)}<\mPhi_1> = \word-!(\ccp^{\nIi - 1})\) for \(2 \le \nIi \le \nIl - 1\).
Since \(\cdn[\mnGen-](\mEvec{\nIl}) \ge 1\) 
and \(\Cdn[\mnGen-](\mEvec{\nIl}) \subseteq \mPhi_1\),
it follows that \(\restWord{\word-(\ccp^{\nIl})}<\mPhi_1> \neq \word-!(\cca)\) for \(0 \le \cca < \ccp^{\nIl - 1}\).
Therefore \(\restWord{\word-(\ccp^{\nIl})}<\mPhi_1> \ge_{\ccF^{\sharp}} \word-!(\ccp^{\nIl - 1})\). 
 
\end{step}

 \begin{step}
 \comment{Step.}
\label{sec:org5545a98}
We show that
\begin{equation}
\label{equ:lem-6-Step2-1}
\restWord{\word-!(\ccp^{\nIl - 1})}<\Cdn[\mnGen-!](0)> = \restWord{\word-*(\ccp^{\nIl - 1})}<\Cdn[\mnGen-*](0)>
\end{equation}
and 
\begin{equation}
\label{equ:lem-6-Step2-2}
 \word-[\cci](\ccp^{\nIl}) =  \word-*[\cci](\ccp^{\nIl-1}) = \word-![\cci](\ccp^{\nIl-1})
\ \tfor \cci \in \Cdn[\mnGen-](\mVec{0 \\0}).
\end{equation}
Note that \(\ccd^\mysecondsymbol = \ccp^{\nIk - 2}(\ccp \mQuo + \mRem)\).
Let
\[
 \ccF^{\myfirstsecondsymbols} = \restWord{\ccF}<\Phi_0 \cap \Phi_1> = \restWord{\ccF}<\Cdn[\mnGen-](\mVec{0 \\ 0})>
\quad \tand \quad
\ccd^{\myfirstsecondsymbols} = \left \lceil \frac{\ccd}{\ccp^2} \right \rceil.
\]
Since (\(\text{D}_{k}\)) holds, we see that (\(\text{D}_{k - 1}\)) holds. 
By Lemma \ref{org7ac4c8a},
\begin{align*}
 \restWord{\word-!(\ccp^{\nIl - 1})}<\Cdn[\mnGen-!](0)>  
 &= \restWord{\longword-|\ccF^{\sharp}, \ccd^\sharp|(\ccp^{\nIl - 1})}<\Cdn[\mnGen-!](0)>\\
&= \longword-|\restWord{\ccF^{\sharp}}<\Cdn[\mnGen-!](0)>,\ \ccd^{\myfirstsecondsymbols}|(\ccp^{\nIl - 2}).
\end{align*}
We claim that \(\restWord{\ccF^{\sharp}}<\Cdn[\mnGen-!](0)> = \ccF^{\myfirstsecondsymbols}\).
Indeed, it follows from Lemma \ref{org16d4190} that
\begin{align*}
 \Cdn[\mnGen-!](0) &= \set{\cci \in \mPhi_1 : \word-!<\cci>(\ccp^0) = 0}\\
&= \set{\cci \in \NN : \word-<\cci>(\ccp) = \word-<\cci>(\ccp^0) = 0} 
 = \Cdn[\mnGen-](\mVec{0 \\0 }).
\end{align*}
Hence
\begin{align*}
\restWord{\ccF^{\sharp}}<\Cdn[\mnGen-!](0)> &= \restWord{\left(\restWord{\ccF}<\mPhi_1> \right)}<\Cdn[\mnGen-](\mVec{0 \\ 0 })> 
=\restWord{\ccF}<\Cdn[\mnGen-](\mVec{0 \\ 0})> = \ccF^{\myfirstsecondsymbols}.
\end{align*}
Therefore
\[
\restWord{\word-!(\ccp^{\nIl - 1})}<\Cdn[\mnGen-!](0)> = 
\longword-|\ccF^{\myfirstsecondsymbols}, \ccd^{\myfirstsecondsymbols}|(\ccp^{\nIl - 2}).
\]
We next show that
\[
\longword-|\ccF^{\myfirstsecondsymbols}, \ccd^{\myfirstsecondsymbols}|(\ccp^{\nIl - 2}) = \restWord{\word-*(\ccp^{\nIl - 1})}<\Cdn[\mnGen-*](0)>. 
\]
It follows from Lemma \ref{org7ac4c8a} that
\begin{align*}
\Cdn[\mnGen-*](0) &= \set{\cci \in \Phi_0 : \word-*<\cci>(\ccp^0) = 0}\\
&= \set{\cci \in \NN : \word-<\cci>(\ccp^0) = \word-<\cci>(\ccp^1) = 0} 
= \Cdn[\mnGen-](\mVec{0\\0}).
\end{align*}
Again, Lemma \ref{org7ac4c8a} implies that
\begin{align*}
\restWord{\word-*(\ccp^{\nIl - 1})}<\Cdn[\mnGen-*](0)>
&= \restWord{\longword-|\ccF^*, \ccd^*|(\ccp^{\nIl - 1})}<\Cdn[\mnGen-*](0)>\\
&= \longword-|\restWord{\ccF^*}<\Cdn[\mnGen-*](0)>,\ \ccd^{\myfirstsecondsymbols}|(\ccp^{\nIl - 2})\\
&= \longword-|\ccF^{\myfirstsecondsymbols},\ \ccd^{\myfirstsecondsymbols}|(\ccp^{\nIl - 2}).
\end{align*}
Therefore \eqref{equ:lem-6-Step2-1} holds.
Since \(\Cdn[\mnGen-*](0) = \Cdn[\mnGen-](\mVec{0 \\ 0})\),
it follows from Lemma \ref{org5683742} that
\begin{align*}
\restWord{\word-*(\ccp^{\nIl - 1})}<\Cdn[\mnGen-*](0)> &= \restWord{\word-*(\ccp^{\nIl - 1})}[\Cdn[\mnGen-](\mVec{0 \\ 0})]\\
&= \sum_{\cci \in \Cdn[\mnGen-](\mVec{0 \\ 0})} \word-*[\cci](\ccp^{\nIl - 1}) \restWord{\left(\restWord{\bdf_i}<\Cdn[\mnGen-](0)>\right)}<\Cdn[\mnGen-](\mVec{0\\0})>\\
&= \sum_{\cci \in \Cdn[\mnGen-](\mVec{0 \\ 0})} \word-*[\cci](\ccp^{\nIl - 1}) \restWord{\bdf_i}<\Cdn[\mnGen-](\mVec{0\\0})>
\end{align*}
and
\begin{align*}
\restWord{\word-!(\ccp^{\nIl - 1})}<\Cdn[\mnGen-!](0)>
&= \restWord{\word-!(\ccp^{\nIl - 1})}[\Cdn[\mnGen-](\mVec{0 \\ 0})]\\
&= \sum_{\cci \in \Cdn[\mnGen-](\mVec{0 \\ 0})} \word-![\cci](\ccp^{\nIl - 1}) \restWord{\bdf_i}<\Cdn[\mnGen-](\mVec{0\\0})>.
\end{align*}
By \eqref{equ:lem-6-Step2-1}, \(\word-![\cci](\ccp^{\nIl - 1}) = \word-*[\cci](\ccp^{\nIl - 1})\) for \(\cci \in \Cdn[\mnGen-](\mVec{0 \\ 0})\).
It follows from Lemmas \ref{org5683742} and \ref{org7ac4c8a} that
\(\restWord{\word-(\ccp^{\nIl})}[\Cdn[\mnGen-](0)] = \restWord{\word-(\ccp^{\nIl})}<\Cdn[\mnGen-](0)> = \word-*(\ccp^{\nIl - 1})\).
Therefore \eqref{equ:lem-6-Step2-2} holds.
 
\end{step}

 \begin{step}
 \comment{Step.}
\label{sec:org54f8ef2}

Let
\[
 \bdy = \sum_{\cci \in \mPhi_1} \word-![\cci](\ccp^{\nIl - 1}) \bdf_{\cci} + \sum_{\cci \in \Cdn[\mnGen-](\mVec{0\\1})} \word-[\cci](\ccp^{\nIl}) \bdf_{\cci}
\]
We show the following.
\begin{align*}
\supp_{\ccF}(\bdy) &\subseteq\mPhi_0 \cup \mPhi_1.\\
\restWord{\bdy}<\mPhi_1> &= \word-!(\ccp^{\nIl - 1}).\\
\restWord{\bdy}<\mPhi_0> &=  \word-*(\ccp^{\nIl - 1}).
\end{align*}
By the definition of \(\bdy\), we see that \(\supp_{\ccF}(\bdy) \subseteq \mPhi_0 \cup \mPhi_1\).
Lemma \ref{org5683742} shows that
\[
\restWord{\bdy}<\mPhi_1> = \restWord{\bdy}[\mPhi_1] 
 = \sum_{\cci \in \mPhi_1} \word-![\cci](\ccp^{\nIl - 1}) \restWord{\bdf_i}<\mPhi_1>
 = \word-!(\ccp^{\nIl - 1}).
\]
From Step 2,
\begin{align*}
 \bdy &= \sum_{\cci \in \mPhi_1} \word-![\cci](\ccp^{\nIl - 1}) \bdf_{\cci} + \sum_{\cci \in \Cdn[\mnGen-](\mVec{0 \\1})} \word-[\cci](\ccp^{\nIl}) \bdf_{\cci}\\
 &= \sum_{\cci \in \Cdn[\mnGen-](\mVec{1 \\ 0})} \word-![\cci](\ccp^{\nIl - 1}) \bdf_{\cci}
+ \sum_{\cci \in \Cdn[\mnGen-](\mVec{0 \\0})} \word-![\cci](\ccp^{\nIl - 1}) \bdf_{\cci}  
+ \sum_{\cci \in \Cdn[\mnGen-](\mVec{0 \\ 1})} \word-[\cci](\ccp^{\nIl}) \bdf_{\cci}\\
 &= \sum_{\cci \in \Cdn[\mnGen-](\mVec{1 \\ 0})} \word-![\cci](\ccp^{\nIl - 1}) \bdf_{\cci}
+ \sum_{\cci \in \Cdn[\mnGen-](\mVec{0 \\0})} \word-[\cci](\ccp^{\nIl}) \bdf_{\cci}  
+ \sum_{\cci \in \Cdn[\mnGen-](\mVec{0 \\ 1})} \word-[\cci](\ccp^{\nIl}) \bdf_{\cci}\\
 &= \sum_{\cci \in \Cdn[\mnGen-](\mVec{1 \\0})} \word-![\cci](\ccp^{\nIl - 1}) \bdf_{\cci}
  + \sum_{\cci \in \mPhi_0} \word-[\cci](\ccp^{\nIl}) \bdf_{\cci}.\\
\end{align*}
Therefore
\begin{align*}
\restWord{\bdy}<\mPhi_0> &= \restWord{\bdy}[\mPhi_0] \\
&=   \sum_{\cci \in \mPhi_0} \word-[\cci](\ccp^{\nIl}) \restWord{\bdf_i}<\mPhi_0>\\
& = \restWord{\word-(\ccp^{\nIl})}[\mPhi_0] = \restWord{\word-(\ccp^{\nIl})}<\mPhi_0> = \word-*(\ccp^{\nIl - 1}).
\end{align*}
 
\end{step}

 \begin{step}
 \comment{Step.}
\label{sec:orgc52ea7f}
We show that
\(\mGen-{\nIl - 1}[\bdy]\) is \(\mDu{\ccd}\)-distributed on \(\mVct{\nIl}(0) \cup \mVct{\nIl}(\mVec{1 \\0})\).
We first show that \(\mGen-{\nIl - 1}[\bdy]\) is \(\mDu{\ccd}\)-distributed on \(\mVct{\nIl}(0)\).
Lemma \ref{org7ac4c8a} implies that \(\mGen-{\nIl}\) is \(\mDu{\ccd}\)-distributed on \(\mVct{\nIl}(0)\).
Since
\[
\restWord{\bdy}<\Cdn[\mnGen-](0)> = \restWord{\word-(\ccp^{\nIl})}<\Cdn[\mnGen-](0)>,
\]
it follows that 
\[
\Cdn[\mGen-{\nIl - 1}[\bdy] ](\mVec{0 \\ \mvecA})
= \Cdn[\mGen-{\nIl}](\mVec{0 \\ \mvecA}) \tfor \mVec{0 \\ \mvecA} \in \mVct{\nIl}(0).
\]
Therefore \(\mGen-{\nIl - 1}[\bdy]\) is \(\mDu{\ccd}\)-distributed on \(\mVct{\nIl}(0)\).
We next show that 
\(\mGen-{\nIl - 1}[\bdy]\) is \(\mDu{\ccd}\)-distributed on \(\mVct{\nIl}(\mVec{1 \\0})\).
For \(\mvecB \in \FF_{\ccp}^{\marray{\nIl - 2}}\),
\[
 \cdn[\mGen-{\nIl - 1}[\bdy]](\mVec{1 \\ 0 \\ \mvecB}) = 
\cdn[\mGen-!{\nIl - 1}](\mVec{1 \\ \mvecB}).
\]
Indeed,
\begin{align*}
\Cdn[\mGen-{\nIl - 1}[\bdy] ](\mVec{1 \\  0 \\ \mvecB})
& = \Set{\cci \in \NN \ : \  \begin{bsmallmatrix} \word-<\cci>(\ccp^0) &  \word-<\cci>(\ccp^1) & \word-<\cci>(\ccp^2)& \cdots & \word-<\cci>(\ccp^{\nIl - 1}) &  \ithComp{\ccy}<\cci> \end{bsmallmatrix} = 
\begin{bsmallmatrix} 1 & 0 & \mvecB  \end{bsmallmatrix}} \\
&= \Set{ \cci \in \Cdn[\mnGen-](\mVec{1 \\ 0}) \ : \
\begin{bsmallmatrix} \word-<\cci>(\ccp^0) & \word-<\cci>(\ccp^2) &  \cdots & \word-<\cci>(\ccp^{\nIl - 1}) &  \ithComp{\ccy}<\cci>\end{bsmallmatrix} = 
\begin{bsmallmatrix}1 & \mvecB \end{bsmallmatrix}}\\
&= \Set{ \cci \in \Cdn[\mnGen-](\mVec{1 \\ 0}) \ : \
\begin{bsmallmatrix} \word-!<\cci>(\ccp^0) & \word-!<\cci>(\ccp^1) &  \cdots & \word-!<\cci>(\ccp^{\nIl - 2}) &  \word-!<\cci>(\ccp^{\nIl - 1})\end{bsmallmatrix}
= \begin{bsmallmatrix}1 & \mvecB  \end{bsmallmatrix}}\\
&= \Cdn[\mGen-!{\nIl - 1}](\mVec{1 \\ \mvecB}).
\end{align*}
Since \(\mGen-!{\nIl - 1}\) is \(\mDu{\ccd^\mysecondsymbol}\)-distributed,
it follows from  Lemma \ref{org3a86abe} that
\(\mGen-{\nIl - 1}[\bdy]\) is \(\mDu{\ccd}\)-distributed on \(\mVct{\nIl}(\mVec{1 \\ 0})\).
 
\end{step}

 \begin{step}
 \comment{Step.}
\label{sec:org78e8e63}

We finally show that \(\restWord{\word-(\ccp^{\nIl})}<\mPhi_1> = \word-!(\ccp^{\nIl - 1})\).
From Step 1, \(\restWord{\word-(\ccp^{\nIl})}<\mPhi_1> \ge_{\ccF^{\sharp}} \word-!(\ccp^{\nIl - 1})\). 
Assume that \(\restWord{\word-(\ccp^{\nIl})}<\mPhi_1> >_{\ccF^{\sharp}} \word-!(\ccp^{\nIl - 1})\).
From Step 4, 
\(\mGen-{\nIl - 1}[\bdy]\) is \(\mDu{\ccd}\)-distributed on \(\mVct{\nIl}(0) \cup \mVct{\nIl}(\mVec{1 \\0})\).
Let \(\ccW = \mVct*{\nIl - 1}(0) \cup \mVct{\nIl - 1}(\mVec{1 \\0})\). Then
\[
 \mVct{\nIl}(\ccW) = \mVct{\nIl}(0) \cup \mVct{\nIl}(\mVec{1 \\0})  \tand 
 \Cdn[\mnGen-](\ccW) =  \Cdn[\mnGen-](0) \cup \Cdn[\mnGen-](\mVec{1 \\0}).
\]
It follows Lemma \ref{org113beae} that there exists \(\bdz\) satisfying the following two conditions.

\begin{enumerate}
\item \(\mGen-{\nIl - 1}[\bdz]\) is \(\mDu{\ccd}\)-distributed.
\item \(\msupp|\ccE|(\bdz - \bdy) = \msupp|\ccF|(\bdz - \bdy) \subseteq \NN \setminus (\Cdn[\mnGen-](\ccW) \cup \Theta) = \Cdn[\mnGen-](\Set{\mVec{1 \\ 1}, \ldots,  \mVec{1 \\ \ccp - 1}}) \ \setminus \Theta\). In particular,
\end{enumerate}
\begin{align*}
 \restWord{\bdz}<\mPhi_1> &= \restWord{\bdy}<\mPhi_1> = \word-!(\ccp^{\nIl - 1}) \neq \restWord{\word-(\ccp^\nIl)}<\mPhi_1>,\\
 \restWord{\bdz}<\mPhi_0> &= \restWord{\bdy}<\mPhi_0> = \word-*(\ccp^{\nIl - 1}) = \restWord{\word-(\ccp^{\nIl})}<\mPhi_0>.
\end{align*}
Since \(\mGen-{\nIl - 1}[\bdz]\) is \(\mDu{\ccd}\)-distributed,
it follows that \(\ccd(\bdz, \word-(\cca)) \ge \ccd\) for \(0 \le \cca < \ccp^{\nIl}\).
Hence \(\bdz >_{\ccF} \word-(\ccp^{\nIl})\).
Let
\begin{align*}
 \ccN &= \max \set{\cci \in \NN : \ithComp{\ccz}[\cci] \neq \word-[\cci](\ccp^{\nIl})} \ \text{and}\\
 \ccM &= \max \set{\cci \in \mPhi_1 : \ithComp{\ccz}[\cci] \neq \word-[\cci](\ccp^{\nIl})}.
\end{align*}
Since \(\restWord{\bdz}[\mPhi_0] = \restWord{\bdz}<\mPhi_0> = \restWord{\word-(\ccp^\nIl)}<\mPhi_0> = \restWord{\word-(\ccp^\nIl)}[\mPhi_0]\),
it follows that \(\ccN, \ccM \not \in \mPhi_0\). In particular, \(\ccM \in \mPhi_1 \setminus \mPhi_0 = \Cdn[\mnGen-](\mVec{1 \\ 0})\).

We show that \(\ccM < \ccN\).
Since \(\bdz >_{\ccF} \word-(\ccp^{\nIl})\),
it follows that
\(\ithComp{\ccz}[\ccN] > \word-[\ccN](\ccp^{\nIl})\).
Recall that \(\restWord{\bdz}<\mPhi_1> = \word-!(\ccp^{\nIl - 1}) <_{\ccF^{\sharp}} \restWord{\word-(\ccp^{\nIl})}<\mPhi_1>\).
This implies that \(\ithComp{\ccz}[\ccM]\relax < \word-[\ccM](\ccp^{\nIl})\).
Hence \(\ccM < \ccN\).

We next show that \(\ccM > \ccN\).
Recall that \(\ccN \not \in \mPhi_0\).
Moreover, \(\ccN \not \in \mPhi_1\)
because if \(\ccN \in \mPhi_1\), then \(\ccN = \ccM\).
Since \(\supp_\ccF(\bdy) \subseteq \mPhi_0 \cup \mPhi_1\),
it follows that \(\ithComp{\ccy}[\ccN] = 0\).
Thus \(\ccN \in \msupp|\ccF|(\bdz - \bdy) \subseteq \Cdn[\mnGen-](\Set{ \mVec{1 \\ 1}, \ldots, \mVec{1 \\ \ccp - 1}}) \setminus \Theta\).
Since \(\ccM \in \Cdn[\mnGen-](\mVec{1 \\ 0})\),
it follows from Lemma \ref{org06db361} that \(\ccM > \ccN\), a contradiction.
Therefore \(\restWord{\word-(\ccp^{\nIl})}<\mPhi_1> = \word-!(\ccp^{\nIl - 1})\).
\hspace*{\fill}\(\qedhere\)
 
\end{step} 
\end{proof}

\begin{proof}[Proof of Proposition \ref{org6d2e6f0} ]
 \comment{Proof. [Proof of Proposition \ref{org6d2e6f0} ]}
\label{sec:org594c49d}

Theorem \ref{org56f1397} shows that
the code \(\mRes(\mLex-{\nIk - 1})\) meets the Griesmer bound.
It follows from Corollary \ref{org09650b1} that \(\mGen-{\nIk - 1}\) is \(\mDu{\ccd}\)-distributed.
We show that \(\mGen-{\nIk}\) is \(\mDu{\ccd}\)-distributed by induction on \(\nIk\).

\resetmycase

\begin{mycase}[\(\nIk = 1\)]
 \comment{Case. [\(\nIk = 1\)]}
\label{sec:orgf54a6a1}
Note that \(\ccd = \ccp \mQuo + \mRem\).
Let \(\ccn\) be the length of \(\mRes(\mLex-{1})\). Then
\[
 \ccn = \cdn[\mnGen-](1) + \cdn[\mnGen-](\mVec{0 \\ 1}) = \ccd + \cdn[\mnGen-](\mVec{0 \\ 1}) = \ccp \mQuo + \mRem + \cdn[\mnGen-](\mVec{0 \\ 1}).
\]

\resetstep
 
\end{mycase} 
 
 \begin{step}
 \comment{Step.}
\label{sec:org44a9cf8}

We first show that \(\cdn[\mnGen-](\mVec{0 \\ 1}) = \mQuo + 1\).
By the Griesmer bound, \(\ccn \ge \ccd + \lceil \frac{\ccd}{\ccp} \rceil = \ccp \mQuo + \mRem + \mQuo + 1\),
and hence \(\cdn[\mnGen-](\mVec{0 \\ 1}) \ge \mQuo + 1\).
Assume that \(\cdn[\mnGen-](\mVec{0 \\1}) > \mQuo + 1\).
Since 
\(\mQuo + 1 \ge 2\), there exists \(\ccM \in \Cdn[\mnGen-](\mVec{0 \\ 1 }) \setminus \Theta\).
Let
\[
\bdy = \word-(\ccp) - \bdf_{\ccM}  = \word-(\ccp) - \bde_{\ccM}.
\]
Note that
\(\ithComp{y}[\ccM] = \word-[\ccM](\ccp) - 1 = 0\)
and
\(\cdn[\mGen-{0}[\bdy]](\mVec{0 \\ 1}) = \cdn[\mnGen-](\mVec{0 \\1}) - 1 \ge \mQuo + 1\).
Since \(\mQuo + 1, \mQuo + \mRem - 2 \ge 2\),
it follows from Lemma \ref{org2f35cc7} that there exists \(\bdz\) satisfying the following conditions.
\begin{align*}
 &\cdn[\mGen-{0}[\bdz]](\mVec{1 \\ \alpha})
 = \mDu{\ccd}(\mVec{1 \\ \alpha}) =
\begin{cases}
\mQuo + 1 & \tif \alpha \neq \ccp - 1, \\
\mQuo + \mRem - \ccp + 1 & \tif \alpha = \ccp - 1. \\
\end{cases}\\
 &\supp_{\ccE}(\bdy - \bdz) = \supp_{\ccF}(\bdy - \bdz) \subseteq \Cdn[\mnGen-](1) \setminus \Theta.
\end{align*}
We claim that \(\bdz <_{\ccF} \word-(\ccp)\).
Since \(\ccM \in \Cdn[\mnGen-](\mVec{0 \\ 1})\), we see that \(\ithComp{z}[\ccM] = \ithComp{y}[\ccM] = 0 < 1 = \word-[\ccM](\ccp)\).
For \(\cci \in \msupp|\ccE|(\bdy - \bdz) \subseteq \Cdn[\mnGen-](1) \setminus \Theta\),
we see that \(\cci < \ccM\) by Lemma \ref{orgfc26eca}. Hence \(\bdz <_{\ccF} \word-(\ccp)\).
Moreover, since \(\cdn[\mGen-{0}[\bdz]](\mVec{0 \\ 1}) = \cdn[\mGen-{0}[\bdy]](\mVec{0 \\ 1}) \ge \mQuo + 1\) and \(\mGen-{0}[\bdz]\) is \(\mDu{\ccd}\)-distributed on \(\mVct{1}(1)\),
it follows from Lemma \ref{org65ab687} that \(\ccd(\bdz, \word-(\cca)) \ge \ccd\) for \(0 \le \cca < \ccp\).
Hence \(\bdz \ge_{\ccF} \word-(\ccp)\), which is impossible. 
Therefore  \(\cdn[\mnGen-](\mVec{0 \\ 1}) = \mQuo + 1\).
In particular, \(\mRes(\mLex-{1})\) meets the Griesmer bound.
 
\end{step}

 \begin{step}
 \comment{Step.}
\label{sec:org9833e29}
We next show that \(\cdn[\mnGen-](\mVec{1 \\ \ccp - 1}) = \mQuo + \mRem - \ccp + 1\).
Assume that \(\cdn[\mnGen-](\mVec{1 \\ \ccp - 1}) > \mQuo + \mRem - \ccp + 1\).
Since \(\mQuo + 1, \mQuo + \mRem - \ccp + 1 \ge 2\),
it follows from Lemma \ref{org2f35cc7} that there exist \(\bdy\) satisfying the following conditions.
\begin{align*}
 &\cdn[\mGen-{0}[\bdy]](\mVec{1 \\ \alpha})
 = \mDu{\ccd}(\mVec{1 \\ \alpha}) =
\begin{cases}
\mQuo + 1 & \tif \alpha \neq \ccp - 1, \\
\mQuo + \mRem - \ccp + 1 & \tif \alpha = \ccp - 1. \\
\end{cases}\\
 &\supp_{\ccE}(\bdy - \word-(\ccp)) = \supp_{\ccF}(\bdy - \word-(\ccp)) \subseteq \Cdn[\mnGen-](1) \setminus \Theta.
\end{align*}
Since \(\ccd = \ccp \mQuo + \mRem < \ccp(\mQuo + 1)\), it follows from 
Corollary \ref{orgb141d90} that 
\(\cdn[\mnGen-](\mVec{1 \\ 0}) \le \mQuo + 1\).
Hence we may assume that
\[
 \msupp|\ccE|(\bdy - \word-(\ccp)) = 
\msupp|\ccF|(\bdy - \word-(\ccp)) \subseteq \Cdn[\mnGen-](\mVec{1 \\ \ccp - 1}) \setminus \Theta.
\]
For \(\cci \in \msupp|\ccF|(\bdy - \word-(\ccp))\), we see that \(\word-[\cci](\ccp) = \ccp - 1\), so \(\bdy <_{\ccF} \word-(\ccp)\).
However, \(\ccd(\bdy, \word-(\cca)) \ge \ccd\) for \(0 \le \cca < \ccp\), which is impossible.
Therefore \(\cdn[\mnGen-](\mVec{1 \\ \ccp - 1}) = \mQuo + \mRem - \ccp + 1\) and
\(\cdn[\mnGen-](\mVec{1 \\ 0}) = \cdn[\mnGen-](\mVec{1 \\ 1}) = \mQuo + 1\)
because \(\sum \cdn[\mnGen-](\mVec{1 \\ \alpha}) = \ccp \mQuo + \mRem\) and \(\cdn[\mnGen-](\mVec{1 \\ \alpha}) \le \mQuo + 1\).
 
\end{step}

\begin{mycase}[\(\cck \ge 2\)]
 \comment{Case. [\(\cck \ge 2\)]}
\label{sec:org5db183b}

\resetstep
 
\end{mycase} 
 
 \begin{step}
 \comment{Step.}
\label{sec:orgb501c8e}

We first show that \(\mGen-{\nIk}\) is \(\mDu{\ccd}\)-distributed on  \(\mVct{\nIk}(0) \cup \mVct{\nIk}(\mVec{1 \\ 0 })\).

By the induction hypothesis, (\(\text{D}_{\nIk}\)) holds.
It follows from Lemmas \ref{org7ac4c8a} and  \ref{org2014e3d} that
\(\mGen-{\nIk}\) is \(\mDu{\ccd}\)-distributed on  \(\mVct{\nIk}(0) \cup \mVct{\nIk}(\mVec{1 \\ 0})\).
 
\end{step}

 \begin{step}
 \comment{Step.}
\label{sec:org3b59388}
For \(\alpha, \beta \in \FF_{\ccp} \setminus \set{0}\),
let
\[
\mConst{\alpha}{\beta} = \alpha + \beta \ccp + \ccp^{\nIk} - \ccp^2 = 
 \alpha + \beta \ccp + (\ccp - 1) \ccp^2 + \cdots + (\ccp - 1) \ccp^{\nIk - 1}.
\]
Let \(\Omega = \mVct{\nIk - 1}(\Set{\mVec{1 \\ 1}, \mVec{1 \\ 2}, \ldots, \mVec{1 \\ \ccp - 1}})\).
To prove 
\(\mGen-{\nIk}\) is \(\mDu{\ccd}\)-distributed on \(\mVct{\nIk}(\mVec{1 \\ 1})\),
we show that
\[
 \ccd(\word-(\ccp^\nIk),\ \word-(\mConst{\alpha}{\beta}) = 
\ccp^{\nIk - 2}\Bigl((\ccp^{2} + \ccp - 1) \mQuo + (\ccp + 1) \mRem - 1 \Bigr) 
- \smashoperator{\sum_{\mvecA \in \Omega}}\ \cdn[\mGen-{\nIk}](\mvecA[\innerproduct{\mvecA}{\mConst{\alpha}{\beta}}]).
\]
Let \(\cca = \mConst{\alpha}{\beta}\).
By Lemma \ref{org4a0db5b},
\begin{align*}
 &\ccd(\word-(\ccp^{\nIk}),\ \word-(\cca)) \\
 &= \sum_{\mvecA \in \mVct{\nIk}, \mvecA<\nIk> \neq \innerproduct{\mvecA}{\cca}} \cdn[\mGen-{\nIk}](\mvecA) \\
 &= \sum_{\mvecA \in \mVct{\nIk}} \cdn[\mGen-{\nIk}](\mvecA) - \sum_{\mvecA \in \mVct{\nIk - 1}} \cdn[\mGen-{\nIk}](\mvecA[\innerproduct{\mvecA}{\cca}]) \\
 &= (\ccp \mQuo + \mRem)\frac{\ccp^{\nIk} - 1}{\ccp - 1}+ \mQuo + 1 - \sum_{\mvecA \in \mVct{\nIk - 1}} \cdn[\mGen-{\nIk}](\mvecA[\innerproduct{\mvecA}{\cca}]) \\
 &= \frac{\ccp^{\nIk}}{\ccp - 1}(\ccp \mQuo + \mRem) - \frac{\mQuo  + \mRem - \ccp + 1}{\ccp - 1} - \sum_{\mvecA \in \mVct{\nIk - 1}} \cdn[\mGen-{\nIk}](\mvecA[\innerproduct{\mvecA}{\cca}]).
\end{align*}

For \(\mvecA \in \mVct{\nIk - 1}(0) \cup \mVct{\nIk - 1}(\mVec{1 \\ 0})\), we calculate 
\(\cdn[\mGen-{\nIk}](\mvecA[\innerproduct{\mvecA}{\cca}])\).
From Steps 1 and 2, \(\mGen-{\nIk}\) is \(\mDu{\ccd}\)-distributed on
\(\mVct{\nIk - 1}(0) \cup \mVct{\nIk - 1}(\mVec{1 \\ 0})\),
and hence
\(\cdn[\mGen-{\nIk}](\mvecA[\innerproduct{\mvecA}{\cca}])\) is determined by
\(\innerproduct{\mvecA[\innerproduct{\mvecA}{\cca}]}{\allone}\).
Since
\[
\innerproduct{\mvecA}{\cca}
= \mvecA<0> \alpha + \mvecA<1> \beta - \mvecA<2> - \cdots -  \mvecA<\nIk - 1>,
\]
it follows that
\[
\innerproduct{\mvecA[\innerproduct{\mvecA}{\cca}]}{\allone} = (1 + \alpha) \mvecA<0> + (1 + \beta) \mvecA<1>.
\]
Since
\((\mvecA<0>, \mvecA<1>) \in \set{(0, 0), (0, 1), (1, 0)}\)
and \(\alpha, \beta \in \FF_\ccp \setminus \ccp - 1\),
it follows that
\[
\innerproduct{\mvecA[\innerproduct{\mvecA}{\cca}]}{\allone} = 0
\iff  \mvecA<0> = \mvecA<1> = 0.
\]
Therefore
\begin{align*}
 & \smashoperator{\sum_{\mvecA \in \mVct{\nIk - 1}(0) \cup \mVct{\nIk - 1}(\mVec{1 \\ 0 })}} \cdn[\mGen-{\nIk}](\mvecA[\innerproduct{\mvecA}{\cca}]) \\
 &= (\mQuo + \mRem - \ccp - 1)  \size{\mVct{\nIk - 1}(\mVec{0 \\ 0 })}
  + (\mQuo + 1) \bigl (\size{\mVct{\nIk - 1}(\mVec{0 \\ 1}, \mVec{1 \\ 0})} \bigr) \\
 &= (\mQuo + \mRem - \ccp + 1) \frac{\ccp^{\nIk - 2} - 1}{\ccp - 1} + (\mQuo + 1) 2 \ \ccp^{\nIk - 2} \\
 &= \frac{\ccp^{\nIk - 2}}{\ccp - 1}\bigl(2 \ccp \mQuo + \ccp + \mRem - \mQuo - 1\bigr) - \frac{\mQuo + \mRem - \ccp + 1}{\ccp - 1}. \\
\end{align*}
It follows that
\begin{align*}
 \ccd(\word-(\ccp^{\nIk}),\ \word-(\cca)) 
 &= \frac{\ccp^{\nIk}}{\ccp - 1}(\ccp \mQuo + \mRem) - \frac{\mQuo  + \mRem - \ccp + 1}{\ccp - 1} - \smashoperator{\sum_{\mvecA \in \mVct{\nIk - 1}}}\ \cdn[\mGen-{\nIk}](\mvecA[\innerproduct{\mvecA}{\cca}]) \\
 &= 
\ccp^{\nIk - 2}\Bigl((\ccp^{2} + \ccp - 1) \mQuo + (\ccp + 1) \mRem - 1 \Bigr) - \smashoperator{\sum_{\mvecA \in \Omega}}\ \cdn[\mGen-{\nIk}](\mvecA[\innerproduct{\mvecA}{\cca}]). 
\end{align*}
 
\end{step}

 \begin{step}
 \comment{Step.}
\label{sec:orga0ce966}
Let \(\gamma \in \FF_\ccp \setminus \set{0}\) and
\begin{align*}
 \ccD =
 \sum_{\cca_0 = 0}^{\ccp - 2} \sum_{\cca_1 = 0}^{\ccp - 2}  \mBigDist{\word-(\ccp^{\nIk})}{\word-(\mConst{\cca_0}{\cca_1})}
 + \sum_{\cca_0 = 0}^{\ccp - 2} \mBigDist{\word-(\ccp^{\nIk})}{\word-(\mConst{\cca_0}{(\cca_0 + 1) \gamma - 1})}.
\end{align*}
By calculating \(\ccD\), we show that \(\mGen-{\nIk}\) is \(\mDu{\ccd}\)-distributed on \(\mVct{\nIk}(\mVec{1 \\ -\gamma^{-1}})\). 
To this end, it suffices to calculate
\(\cdn[\mGen-{\nIk}](\mvecA[\innerproduct{\mvecA}{\cca}])\) for \(\mvecA \in \Omega\) by Step 3.
Let \(\sigma = - \mvecA<2> - \cdots - \mvecA<\nIk - 1>\). Then
\[
 \innerproduct{\mvecA}{\mConst{\cca_0}{\cca_1}} = \cca_0 + \mvecA<1>\cca_1 + \sigma.
\]
Note that
\[
\Bigl\{\innerproduct{\mvecA}{\mConst{\cca_0}{\cca_1}} : \cca_1 \in \FF_\ccp \setminus \set{\ccp - 1} \Bigr\}
=  \FF_\ccp \setminus \set{\cca_0 - \mvecA<1> + \sigma}.
\]
Hence
\begin{align*}
 \sum_{\cca_0 = 0}^{\ccp - 2} \sum_{\cca_1 = 0}^{\ccp - 2} \cdn[\mGen-{\nIk}](\mvecA[\innerproduct{\mvecA}{\mConst{\cca_0}{\cca_1}}])
&= (\ccp - 1) \cdn[\mGen-{\nIk - 1}](\mvecA) -
 \sum_{\cca_0 \in \set{0,1,\ldots, \ccp - 2}} \cdn[\mGen-{\nIk}](\mvecA[\cca_0 - \mvecA<1> + \sigma ])\\
&= (\ccp - 2) \cdn[\mGen-{\nIk - 1}](\mvecA) +
 \cdn[\mGen-{\nIk}](\mvecA[-1 - \mvecA<1> + \sigma]).
\end{align*}
Moreover, 
\[
 \innerproduct{\mvecA}{\mBigConst{\cca_0}{(\cca_0 + 1) \gamma - 1}} = \cca_0 + \mvecA<1>\Bigl((\cca_0 + 1) \gamma - 1 \Bigr) + \sigma
= \cca_0 (1 + \mvecA<1> \gamma) + \mvecA<1> \gamma - \mvecA<1> + \sigma.
\]
Let \(\nu = - \gamma^{-1}\).
Note that if \(\mvecA<1> \neq \nu\), then \(1 + \mvecA<1> \gamma \neq 0\), and hence
\[
\Set{\innerproduct{\mvecA}{\mBigConst{\cca_0}{(\cca_0 + 1) \gamma - 1}} : \cca_0 \in \FF_\ccp \setminus \set{\ccp - 1} }
=  \FF_\ccp \setminus \set{-1 - \mvecA<1> + \sigma}.
\]
If \(\mvecA<1> = \nu\), then
\[
\Set{\innerproduct{\mvecA}{\mBigConst{\cca_0}{(\cca_0 + 1) \gamma - 1}} : \cca_0 \in \FF_\ccp \setminus \set{\ccp - 1}}
=  \set{-1 - \mvecA<1> + \sigma}.
\]
Hence
\begin{align*}
 \sum_{\cca_0 = 0}^{\ccp - 2} \cdn[\mGen-{\nIk}](\mvecA[\innerproduct{\mvecA}{\mConst{\cca_0}{(\cca_0 + 1) \gamma - 1}} ])
&= \begin{cases}
 \cdn[\mGen-{\nIk - 1}](\mvecA) - \cdn[\mGen-{\nIk}](\mvecA[-1 - \mvecA<1> + \sigma])
 & \tif \mvecA<1> \neq \nu, \\
 (\ccp - 1) \cdn[\mGen-{\nIk}](\mvecA[-1 - \mvecA<1> + \sigma])
 & \tif \mvecA<1> = \nu
\end{cases}\\
&= \begin{cases}
 (\ccp - 1) (\ccp \mQuo + \mRem)  - \cdn[\mGen-{\nIk}](\mvecA[-1 - \mvecA<1> + \sigma])
 & \tif \mvecA<1> \neq \nu, \\
 (\ccp - 1) \cdn[\mGen-{\nIk}](\mvecA[-1 - \mvecA<1> + \sigma])
 & \tif \mvecA<1> = \nu.
\end{cases}
\end{align*}
Therefore
\begin{align*}
 &\sum_{\cca_0 = 0}^{\ccp - 2} \sum_{\cca_1 = 0}^{\ccp - 2} \cdn[\mGen-{\nIk}](\mvecA[\innerproduct{\mvecA}{\mConst{\cca_0}{\cca_1}}])
+ \sum_{\cca_0 = 0}^{\ccp - 2} \cdn[\mGen-{\nIk}](\mvecA[\innerproduct{\mvecA}{\mConst{\cca_0}{(\cca_0 + 1) \gamma - 1}} ])\\
 &= \begin{cases}
 (\ccp - 1) (\ccp \mQuo + \mRem)
 & \tif  \mvecA<1> \neq \nu, \\
 (\ccp - 2) (\ccp \mQuo + \mRem)  + \ccp  \cdn[\mGen-{\nIk}](\mvecA[-1 - \mvecA<1> + \sigma])
 & \tif \mvecA<1> = \nu. \\
 \end{cases}
\end{align*}
Note that, for \(\mvecA \in \Omega\), 
\[
\innerproduct{\mvecA[-1 - \mvecA<1> + \sigma]}{\allone} = - 1 + \mvecA<1> + \mvecA<2> + \cdots + \mvecA<\nIk - 1> -1 - \mvecA<1> + \sigma  =  0.
\]
Hence \(\mvecA[-1 - \mvecA<1> + \sigma]\) is an element \(\mvecB\) in \(\mVct{\nIk}(\mVec{1 \\ \mvecA<1>})\) with
\(\innerproduct{\mvecB}{\allone} = 0\).
Therefore
\begin{align*}
 \ccD &=  \sum_{\cca_0 = 0}^{\ccp - 2} \sum_{\cca_1 = 0}^{\ccp - 2}  \mBigDist{\word-(\ccp^{\nIk})}{\word-(\mConst{\cca_0}{\cca_1})}
 + \sum_{\cca_0 = 0}^{\ccp - 2} \mBigDist{\word-(\ccp^{\nIk})}{\word-(\mConst{\cca_0}{(\cca_0 + 1) \gamma - 1})}\\
&= \ccp^{\nIk - 1} (\ccp - 1) \Bigl((\ccp^{2} + \ccp - 1) \mQuo + (\ccp + 1) \mRem - 1 \Bigr)\\
 &\quad - \sum_{\mvecA \in \Omega, \mvecA<1> \neq \nu} (\ccp - 1) (\ccp \mQuo + \mRem)
  - \sum_{\mvecA \in \Omega, \mvecA<1> = \nu} (\ccp - 2) (\ccp \mQuo + \mRem) + \ccp \cdn[\mGen-{\nIk}](\mvecA[-1 - \mvecA<1> + \sigma]) \\
&= \ccp^{\nIk - 1} (\ccp - 1) \Bigl((\ccp^{2} + \ccp - 1) \mQuo + (\ccp + 1) \mRem - 1 \Bigr) \\
 &\quad - \ccp^{\nIk - 2} (\ccp - 1)(\ccp - 2) (\ccp \mQuo + \mRem)
- \ccp^{\nIk - 2} (\ccp - 2)  (\ccp \mQuo + \mRem)
 - \ccp \smashoperator{\sum_{\mvecB \in \mVct{\nIk}(\mVec{1 \\ \nu}),\ \innerproduct{\mvecB}{\allone} = 0}} \cdn[\mGen-{\nIk}](\mvecB) \\
&= \ccp^{\nIk - 1} \Bigl((\ccp^3 - \ccp^2 + 1) \mQuo + (\ccp^2 - \ccp + 1) \mRem ^ \ccp + 1 \Bigr)
 - \ccp  \smashoperator{\sum_{\mvecB \in \mVct{\nIk}(\mVec{1 \\ \nu}),\ \innerproduct{\mvecB}{\allone} = 0}} \cdn[\mGen-{\nIk}](\mvecB) \\
 &\ge (\ccp^2 - \ccp) \ccd = (\ccp^2 - \ccp) (\ccp \mQuo + \mRem).
\end{align*}
 This implies that
\[
\smashoperator{\sum_{\mvecB \in \mVct{\nIk}(\mVec{1 \\ \nu}),\ \innerproduct{\mvecB}{\allone} = 0}} \cdn[\mGen-{\nIk}](\mvecB)
\le \ccp^{\nIk - 2}(\mQuo + \mRem - \ccp + 1).
\]
Since \(\cdn[\mGen-{\nIk}](\mvecB) \ge \mQuo + \mRem - \ccp + 1\), it follows that
\(\cdn[\mGen-{\nIk}](\mvecB) = \mQuo + \mRem - \ccp + 1\)
for \(\mvecB \in  \mVct{\nIk}(\mVec{1 \\  \nu})\) with \(\innerproduct{\mvecB}{\allone} = 0\).
Write \(\mvecB = \mvecA[\beta]\).
Then \(\cdn[\mGen-{\nIk}](\mvecA[\alpha]) = \mQuo + 1\) for \(\alpha \neq \beta\).
This implies that
\(\mGen-{\nIk}\) is \(\mDu{\ccd}\)-distributed on \(\mVct{\nIk}(\mVec{1 \\ \nu})\).
Therefore
\(\mGen-{\nIk}\) is \(\mDu{\ccd}\)-distributed. \hspace*{\fill}\(\qedhere\)
 
\end{step} 
\end{proof}

\section{Proof of Theorem 1.7}
\label{sec:org14c8298}
\label{org2aaef62}
\subsection{Case \texorpdfstring{\(\ccp = 3\)}{p = 3}}
\label{sec:org57d5a25}
We show that \(\word(3^{\nIk} + 1) \neq \word(3^{\nIk}) + \word(1)\).
Let \(\bdy = \word(3^{\nIk}) + \word(1)\).
Theorem \ref{orga034f76} implies that \(\word(\cca) = \word-(\cca)\) for \(0 \le \cca \le 3^{\nIk}\).

Let \(\mvecA \in \mVct{\nIk - 1}(1)\)
and \(\beta = - \mvecA<0> - \mvecA<1> - \cdots - \mvecA<\nIk - 1>\).
Note that
\[
 \cdn[\mGen{\nIk - 1}[\bdy]](\mvecA[\alpha]) = \begin{cases} \delta + \epsilon - 2 & \tif \alpha = \beta + 1,\\  \delta + 1 & \tif \alpha = \beta, \beta + 2. \end{cases}
\]
Since  \(\delta + 1, \delta + \epsilon - 2 \ge 2\),
it follows from Lemma \ref{org2f35cc7} that
there exists \(\tilde{\bdy}\) satisfying the following two conditions:

\begin{enumerate}
\item \(\cdn[\mGen{\nIk - 1}[\tilde{\bdy}]](\mvecA[\alpha]) = \begin{cases} \delta + \epsilon - 2 & \tif \alpha = \beta,\\  \delta + 1 & \tif \alpha = \beta + 1, \beta + 2.\end{cases}\)
\item \(\msupp|\ccF|(\tilde{\bdy} - \bdy) = \msupp|\ccE|(\tilde{\bdy} - \bdy) \subseteq \Cdn[\mGen{\nIk - 1}[\bdy]](\mvecA[\beta])  \setminus \Theta\).
\end{enumerate}

\noindent
Note that \(\Cdn[\mGen{\nIk - 1}[\bdy]](\mvecA[\beta])  = \Cdn[\mGen{\nIk}](\mvecA[\beta + 2])\).

\begin{center}
\begin{tabular}{l|c@{\quad}c@{\quad}c}
 & \(\Cdn[\mGen{\nIk}](\mvecA[\beta])\) & \(\Cdn[\mGen{\nIk}](\mvecA[\beta + 2])\) & \(\Cdn[\mGen{\nIk}](\mvecA[\beta + 1])\)\\
\hline
\(\word(3^{\nIk})\) & \(\beta^{\delta + \epsilon - 2}\) & \((\beta + 2)^{\delta + 1}\) & \((\beta + 1)^{\delta + 1}\)\\
\(\bdy\) & \((\beta + 1)^{\delta + \epsilon - 2}\) & \(\beta^{\delta + 1}\) & \((\beta + 2)^{\delta + 1}\)\\
\(\tilde{\bdy}\) & \((\beta + 1)^{\delta + \epsilon - 2}\) & \((\beta+1)^{3 - \epsilon}\beta^{\delta + \epsilon - 2}\) & \((\beta + 2)^{\delta + 1}\)\\
\end{tabular}
\end{center}

\noindent
By repeatedly using Lemma \ref{org2f35cc7}, we see that there exists \(\bdz\) satisfying the following two conditions:

\begin{enumerate}
\item \(\cdn[\mGen{\nIk - 1}[\bdz]](\mvecA[\alpha]) = \begin{cases} \delta + \epsilon - 2 & \tif \alpha = \beta,\\  \delta + 1 & \tif \alpha = \beta + 1, \beta + 2.\end{cases}\)
\item \(\msupp|\ccF|(\bdz - \bdy) = \msupp|\ccE|(\bdz - \bdy) \subseteq \bigcup_{\mvecA \in \mVct{\nIk - 1}(1)} \Cdn[\mGen{\nIk}](\mvecA[\beta + 2]) \setminus \Theta\)
\end{enumerate}

\noindent
It suffices to show that \(\bdz <_{\ccF} \bdy = \word(3^{\nIk}) + \word(1)\) and \(\mDist{\bdz}{\word(\cca)} \ge \ccd\) for \(0 \le \cca \le 3^{\nIk}\).
Since \(\mGen{\nIk - 1}[\bdz]\) is \(\mDu{\ccd}\)-distributed,
it follows that \(\mDist{\bdz}{\word(\cca)} \ge \ccd\) for \(0 \le \cca < 3^{\nIk}\).
Moreover, by the definition of \(\bdz\), we see that \(\mDist{\bdz}{\word(3^{\nIk})} = \ccd\).
Let
\[
 \ccN = \max \set{\cci \in \NN : \ithComp{\ccz}[\cci] \neq \word[\cci](3^{\nIk}) + \word[\cci](1)}.
\]
Then \(\ccN \in \Cdn[\mGen{\nIk - 1}](\transpose{\begin{bsmallmatrix}1 & 0 & \cdots & 0  \end{bsmallmatrix}}) \setminus \Theta\).

\begin{center}
\begin{tabular}{l|c@{\quad}c@{\quad}c}
 & \(\Cdn[\mGen{\nIk}](\mvecA[2])\) & \(\Cdn[\mGen{\nIk}](\mvecA[1])\) & \(\Cdn[\mGen{\nIk}](\mvecA[0])\)\\
\hline
\(\word(3^{\nIk})\) & \(2^{\delta + \epsilon - 2}\) & \(1^{\delta + 1}\) & \(0^{\delta + 1}\)\\
\(\word(3^{\nIk}) + \word(1)\) & \(0^{\delta + \epsilon - 2}\) & \(2^{\delta + 1}\) & \(1^{\delta + 1}\)\\
\(\bdz\) & \(0^{\delta + \epsilon - 2}\) & \(0^{3 - \epsilon}2^{\delta + \epsilon - 2}\) & \(1^{\delta + 1}\)\\
\end{tabular}
\end{center}

\noindent
Hence \(\ithComp{\ccz}<\ccN> = 0 < 2 = \word<\ccN>(3^{\nIk}) + \word<\ccN>(1)\)
and \(\bdz <_{\ccF} \word(3^{\nIk}) + \word(1)\).
Therefore \(\word(3^{\nIk} + 1) \neq \word(3^{\nIk}) + \word(1)\). \hspace*{\fill}\(\qedhere\)

\subsection{Case \texorpdfstring{\(\ccp \ge 5\), \(\ccd = 2\)}{p >= 5, d = 2}}
\label{sec:org60d8c4e}

\begin{proposition}
 \comment{Prop.}
\label{sec:org0ace86b}
If \(\ccd = 2\) and \(\xi \in \set{0, 1}\),
then \(\MaxDimLinearLex|\ccF, \ccd| = 2\).
 
\end{proposition}

\begin{proof}
 \comment{Proof.}
\label{sec:orgd30d8f1}
Let \(\set{\xi, \iota} = \set{0, 1}\)
and \(\ccF' = \transpose{\begin{bmatrix}  \transpose{\bdf_\xi} & \transpose{\bdf_\iota} & \transpose{\bdf_2}  & \transpose{\bdf_3}\end{bmatrix}}\).

\resetstep
 
 \begin{step}
 \comment{Step.}
\label{sec:orgd5c372e}
We first show that
\[
 \word(1) = \mWord{1  0  0  0} \ccF'.
\]
If \(\xi = 0\), then \(\word(1) = \bdf_\xi\).
Suppose that \(\xi = 1\).
Since \(\word(1) \neq \mWord{0,\alpha, 0, 0} \ccF'\),
we see that \(\word(1) = \mWord{1  0  0  0} \ccF' = \bdf_\xi\).
 
\end{step}

 \begin{step}
 \comment{Step.}
\label{sec:org8135dbc}
We next show that
\[
\word(a_0 + a_1 \ccp) = \mWord{a_0  a_1  a_1  0} \ccF'
\]
by induction on \(\cca_1\).

Let \(\cca_1 = 0\).
If \(a_0 = 1\), then it follows from Step 1 that
\(\word(1) = \mWord{1  0  0  0} \ccF'\).
Since the distance between \(\mWord{1, \alpha,  0, 0} \ccF'\) and
\(\word(1)\) equals one,
we see that
\(\word(2) = \mWord{2  0  0  0} \ccF'\).
In general, \(\word(a_0) = \mWord{a_0  0  0  0} \ccF'\).

Suppose that \(a_1 \ge 1\).
For \(\alpha < \cca_1\),
by induction hypothesis,
we see that
\(\word(\mu + \alpha \ccp) = 
\mWord{\mu, \alpha, \alpha, 0} F'\).
Note that the distance between
\(\mWord{\lambda, \mu, \alpha, 0} F'\) and
\(\mWord{\lambda, \alpha, \alpha, 0} F'\) 
is at most one.
Therefore
\(\word(a_0 + a_1 \ccp) = \mWord{\lambda, \mu,  a_1, 0} F'\).
For \(\alpha < a_1\),
the distance between
\(\mWord{\lambda, \alpha,  a_1, 0} F'\) and
\(\word(\alpha + \lambda \ccp) = \mWord{\lambda, \alpha, \alpha,  0} F'\) 
is at most one.
Therefore \(\word(a_1 \ccp) = \mWord{0, a_1, a_1, 0} F'\).
We see that the distance
between
\(\mWord{a_0  a_1  a_1  0} F'\) 
and \(\word(\alpha + \lambda \ccp)\) equals two.
Therefore \(\word(a_0 + a_1 \ccp ) = \mWord{a_0  a_1  a_1 0} F'\).
 
\end{step}

 \begin{step}
 \comment{Step.}
\label{sec:org34a860a}

We show that \(\word(\ccp^2) = \mWord{0  1  0  1} F'\).
The distance between \(\mWord{\lambda, \mu, \nu, 0} F'\) and
\(\word(\lambda + \mu \ccp) = \mWord{\lambda, \mu,  \mu, 0} F'\) is at most one.
The distance between
\(\mWord{\lambda, 0,  0,  1} F'\) and
\(\word(\lambda) = \mWord{\lambda, 0, 0, 0} F'\) is at most one
Therefore \(\word(\ccp^2) = \mWord{0  1  0  1} F'\).
 
\end{step}

 \begin{step}
 \comment{Step.}
\label{sec:orgfabc250}

We show that
if \(\word(\alpha + \ccp^2) = \word(\alpha) + \word(\ccp^2)\)
for \(1 \le \alpha < \ccp\),
then \(\word(\ccp + \ccp^2) \neq \word(\ccp) + \word(\ccp^2)\).
Note that \(\word(\ccp) + \word(\ccp^2) = \mWord{0  2  1  1} F'\).
We show that the distance between \(\mWord{0  0  1  1} F'\) 
and the founded words is at least two.
Indeed, the distance between
\(\mWord{0  0  1  1} F'\)
and
\(\word(\cca_0 + \cca_1 \ccp) = 
\mWord{\cca_0  \cca_1  \cca_1  0} F'\) 
is at least two for \(\cca_1 = 0, 1\).
Moreover,
the distance between
\(\mWord{0  0  1  1} F'\)
and \(\word(\alpha + \ccp^2) 
= \mWord{\alpha  1  0  1} F'\) 
is at least two.
Therefore
\(\word(\ccp + \ccp^2) \neq \word(\ccp) + \word(\ccp^2)\). 
 
\end{step} 
\end{proof}

\begin{proposition}
 \comment{Prop.}
\label{sec:orgfa3b75b}
If \(\ccd = 2\) and \(\xi = 2\),
then \(\MaxDimLinearLex|\ccF, \ccd| = 2\).
 
\end{proposition}

\begin{proof}
 \comment{Proof.}
\label{sec:org237d2e8}
Let
\[
 \ccF' = \transpose{\begin{bmatrix} \transpose{\bdf_0} &  \transpose{\bdf_1} & \transpose{\bdf_2} & \transpose{\bdf_3} \end{bmatrix}}.
\]
We first show that
\[
\word(a_1 \ccp + a_0) = \mWord{a_0  a_0  a_1 0 } F'.
\]
by induction on  \(\cca_1\).
Let \(\cca_1 = 0\).
We see that \(\word(1) = \mWord{1  1  0 0} F\).
Hence \(\word(\cca_0) = \mWord{\cca_0  \cca_0  0 0} F\).
Let \(\cca_1 \ge 1\).
For \(\alpha < \cca_1\),
the distance between
\(\mWord{\lambda, \mu, \alpha, 0} F\) and
\(\word(\lambda + \alpha \ccp) = \mWord{\lambda, \lambda, \alpha, 0} F\)
equals one.
Thus
\(\word(\cca_0 + \cca_1 \ccp) = \mWord{\lambda, \mu, \cca_1, 0} F\).
Note that
For \(\alpha < \cca_0\),
\[
 \mBigDist{\mWord{\lambda, \alpha, \cca_1, 0} F}{\mWord{\cca_0 \cca_0 \cca_1 0} F} = 1.
\]
\[
 \mBigDist{\mWord{\alpha, \mu, \cca_1, 0} F}{\mWord{\cca_0, \cca_0, \cca_1, 0} F} = 1.
\]
Hence \(\word(\cca_0 + \cca_1 \ccp) = \mWord{\cca_0 \cca_0  \cca_1 0} F\).

We next show that \(\word(\ccp^2) = \mWord{1  0  0  1} F\).
The distance between \(\mWord{\lambda, \mu, \nu, 0} F\) and
\(\word(\lambda + \nu \ccp) = \mWord{\lambda, \lambda, \nu, 0} F\) 
equals one.
Thus \(\word(\ccp^2) = \mWord{1  0  0  1} F\).

Finally, we show that \(\word(1 + \ccp^2) \neq \word(1) + \word(\ccp^2)\).
Note that
\(\word(1) + \word(\ccp^2) = \mWord{2  1  0  1} F\).
Since 
\[
 \mBigDist{\mWord{0  1  0  1} F}{\word(\cca)} \ge 2 \quad \tfor 0 \le \cca \le \ccp^2,
\]
it follows that \(\word(1 + \ccp^2) \neq \word(1) + \word(\ccp^2)\).
\end{proof}

\begin{proposition}
 \comment{Prop.}
\label{sec:org55d7895}
If \(\ccd = 2\) and \(\xi \ge 3\),
then \(\MaxDimLinearLex|\ccF, \ccd| = 1\).

\begin{proof}
 \comment{Proof.}
\label{sec:orgca2f247}
We see that \(\word(\alpha) = \alpha \bdf_0 + \alpha \bdf_1\) for \(\alpha \in \FF_\ccp\)
and \(\word(\ccp) = \bdf_0 + \bdf_2\).
Moreover, 
\(\bdf_1 + \bdf_2 <_{\ccF} \bdf_0 + \bdf_1 + \bdf_2 = \word(\ccp) + \word(1)\) and \(\mDist{\bdf_1 + \bdf_2}{\word(\cca)} \ge 2\) for \(0 \le \cca \le \ccp\).
Hence \(\word(\ccp + 1) \neq \word(\ccp) + \word(1)\).
\end{proof}
 
\end{proposition}

\subsection{Case \texorpdfstring{\(\ccp \ge 5\), \(\ccd \ge 3\)}{p >= 5, d >= 3}}
\label{sec:orga6afd58}

Let \(\ccF \in \sF\) and \(\ccd \ge 3\).

\begin{lemma}
 \comment{Lem.}
\label{sec:orgb3df05c}
\label{org042dd48}

\begin{enumerate}
\item \(\cdn[\mnGen](1) = \ccd\).

\item \(\cdn[\mnGen]({\mVec{0 \\ 1}}) \ge \max_{\alpha \in \FF_\ccp} \cdn[\mnGen]({\mVec{1 \\ \alpha}})\).
Moreover, if \(\Cdn[\mnGen]({\mVec{0 \\ 1}}) \setminus \Theta \neq \emptyset\),
then \(\cdn[\mnGen]({\mVec{0 \\ 1}}) = \max_{\alpha \in \FF_\ccp} \cdn[\mnGen]({\mVec{1 \\ \alpha}})\).
\end{enumerate}
 
\end{lemma}

\begin{proof}
 \comment{Proof.}
\label{sec:org53048bb}

By definition, \(\mGen{0} = \mGen-{0}\).
Hence Theorem \ref{org56f1397}
shows that \(\cdn[\mnGen](1) = \ccd\).
Since \(\mDist{\word(\ccp)}{\word(\alpha)} = \cdn[\mnGen](1) + \cdn[\mnGen](\mVec{0 \\ 1}) - \cdn[\mnGen](\mVec{1 \\ \alpha}) \ge \ccd\),
it follows that \(\cdn[\mnGen]({\mVec{0 \\ 1}})  \ge \cdn[\mnGen]({\mVec{1 \\ \alpha}})\).
Suppose that \(\Cdn[\mnGen]({\mVec{0 \\ 1}}) \setminus \Theta \neq \emptyset\),
and let \(\ccM \in \Cdn[\mnGen]({\mVec{0 \\ 1}}) \setminus \Theta\) and
\[
 \bdy = \word(\ccp) - \bdf_\ccM.
\]
Since \(\ithComp{\ccy}[\ccM] = \word[\ccM](\ccp) - 1 = 0\),
we see that \(\bdy <_{\ccF} \word(\ccp)\).
Assume that
\[
\cdn[\mnGen]({\mVec{0 \\ 1}}) > \max_{\alpha \in \FF_\ccp} \cdn[\mnGen]({\mVec{1 \\ \alpha}}).
\]
Then, for \(\alpha \in \FF_\ccp\),
\[
 \mDist{\bdy}{\word(\alpha)} =
 \cdn[\mnGen](1) + \cdn[\mnGen](\mVec{0 \\ 1}) -  \cdn[\mnGen](\mVec{\alpha \\ 1}) - 1
\ge \cdn[\mnGen](1) = \ccd.
\]
This implies that \(\bdy \ge_{\ccF} \word(\ccp)\), a contradiction.
\end{proof}

\begin{lemma}
 \comment{Lem.}
\label{sec:org28c63f1}
\label{orgfb723c7}
Let \(\alpha, \beta \in \FF_\ccp\) with \(\alpha < \beta\) and \(\Cdn[\mnGen](\mVec{1 \\ \beta}) \neq \emptyset\).
If there exist \(\gamma\) such that \(\alpha < \gamma \le \beta\) and \(\Cdn[\mnGen](\mVec{1 \\ \gamma}) \setminus \Theta \neq \emptyset\),
then \(\cdn[\mnGen](\mVec{1 \\ \alpha}) \ge \cdn[\mnGen](\mVec{1 \\ \beta})\).
 
\end{lemma}

\begin{proof}
 \comment{Proof.}
\label{sec:org4aff1c3}

Assume that \(\cdn[\mnGen](\mVec{1 \\ \alpha}) <  \cdn[\mnGen](\mVec{1 \\ \beta})\).
Let \(\ccM \in \Cdn[\mnGen](\mVec{1 \\ \gamma}) \setminus \Theta\) and
\[
 \bdy = \word(\ccp) + (\alpha - \gamma) \bdf_\ccM = \word(\ccp) + (\alpha - \gamma) \bde_\ccM .
\]
We see that
\[
\ithComp{\ccy}[\ccM] = \word[\ccM](\ccp) + \alpha - \gamma = \gamma + \alpha - \gamma = \alpha < \gamma = \word[\ccM](\ccp).
\]
This implies that \(\bdy <_{\ccF} \word(\ccp)\).
Note that \(\mDist{\bdy}{\word(\cca)} \ge \mDist{\word(\ccp)}{\word(\cca)} \ge \ccd\)
for \(0 \le \cca < \ccp\) with \(\cca \neq \alpha\).
Moreover, Lemma \ref{org042dd48} shows that
\(\cdn[\mnGen](\mVec{0 \\ 1}) \ge
\cdn[\mnGen](\mVec{1 \\ \beta})
> \cdn[\mnGen](\mVec{1 \\ \alpha})\),
and hence \(\ccd(\bdy, \word(\alpha)) \ge \ccd\).
Therefore \(\bdy \ge_{\ccF} \word(\ccp)\), which is impossible.
\end{proof}

\begin{lemma}
 \comment{Lem.}
\label{sec:org9bb1612}
\label{orga081bbf}
Let \(\alpha, \beta \in \FF_\ccp\) with \(\alpha < \beta\).
If \(\Cdn[\mnGen](\mVec{1 \\ \beta}) \neq \emptyset\),
then \(\Cdn[\mnGen](\mVec{1 \\ \alpha}) \neq \emptyset\).
 
\end{lemma}

\begin{proof}
 \comment{Proof.}
\label{sec:orgc290e6a}

When \(\Cdn[\mnGen](\mVec{1 \\ \beta}) \setminus \Theta \neq \emptyset\),
it follows from Lemma \ref{orgfb723c7} that
\[
\cdn[\mnGen](\mVec{1 \\ \alpha})
 \ge \cdn[\mnGen](\mVec{1 \\ \beta}) > 0.
\]
Suppose that \(\Cdn[\mnGen](\mVec{1 \\ \beta}) \setminus \Theta = \emptyset\),
and assume that \(\Cdn[\mnGen](\mVec{1 \\ \alpha}) = \emptyset\).

Let \(\ccM \in \Cdn[\mnGen](\mVec{1 \\ \beta})\) and
\[
 \bdy = \word(\ccp) + (\alpha - \beta) \bdf_\ccM.
\]
Then \(\ccM = \eta\) or \(\xi\). Note that the following statements hold.

\begin{enumerate}[label=(\roman*)]
  \item If $\ccM = \eta$, then  $\alpha - \word[\xi](\ccp)  > \beta - \word[\xi](\ccp)$.
  \item If $\ccM = \xi$, then $\cdn[\mnGen](\mVec{0 \\ 1}) = 1$.
\end{enumerate}

\noindent
Indeed, if \(\ccM = \eta\), then \(\ccd(\bdy, \word(\cca)) \ge \ccd\)  for \(0 \le \cca < \ccp\),
and hence \(\bdy >_{\ccF} \word(\ccp)\),
which implies that \(\ithComp{\ccy}[\ccM] = \alpha - \word[\xi](\ccp)  > \beta - \word[\xi](\ccp) = \word[\ccM](\ccp)\).
If \(\ccM = \xi\), then \(\bdy <_{\ccF} \word(\ccp)\),
and hence 
\[
\ccd(\bdy, \word(\alpha)) = \cdn[\mnGen](1) + \cdn[\mnGen](\mVec{0 \\ 1}) -  \cdn[\mnGen](\mVec{1 \\ \alpha}) - 2
= \ccd + \cdn[\mnGen](\mVec{0 \\ 1}) -  2 < \ccd,
\]
which yields \(\cdn[\mnGen](\mVec{0 \\ 1}) = 1\).

We show that \(\xi \in \Cdn[\mnGen](1)\).
Assume that \(\xi \not \in \Cdn[\mnGen](1)\). Then \(\ccM = \eta\).
By (i), \(\alpha - \word[\xi](\ccp)  > \beta - \word[\xi](\ccp)\).
It follows that \(\xi \in \Cdn[\mnGen](\mVec{0 \\ 1})\), and hence \(\alpha - 1  > \beta - 1\).
Hence \(\alpha = 0\) and \(\cdn[\mnGen](\mVec{1 \\ 0}) = 0\). 
It follows from Lemma \ref{orgfb723c7} that \(\Cdn[\mnGen](1) \subseteq \Theta\),
and hence \(\cdn[\mnGen](1) = \ccd = 1\), a contradiction.
Hence \(\xi \in \Cdn[\mnGen](1)\).
Lemma \ref{orgedc6ece} implies that \(\eta \in \Cdn[\mnGen](1)\).

Assume that \(\xi, \eta \in \Cdn[\mnGen](\mVec{1 \\ \beta})\).
Then \(\cdn[\mnGen](\mVec{0 \\ 1}) \ge \cdn[\mnGen](\mVec{1 \\ \beta}) \ge 2\),
contrary to (ii).
Therefore \(\Cdn[\mnGen](\mVec{1 \\ \beta}) =  \{\xi\}\) or \(\{\eta\}\).

\resetmycase
\newcommand{\tmpcaselabel}{{\Cdn[\mnGen](\mVec{1 \\ \beta}) = \{\xi\}}}
\newcommand{\tmpcaselabelsecond}{{\Cdn[\mnGen](\mVec{1 \\ \beta}) = \{\eta\}}}

\begin{mycase}[\(\tmpcaselabel\)]
 \comment{Case. [\(\tmpcaselabel\)]}
\label{sec:orgcd7cf30}

Let \(\eta \in \Cdn[\mnGen](\mVec{1 \\ \gamma})\).
Let
\[
 \bdz = \word(\ccp) + (\alpha - \beta) \bdf_\xi + (1 + \beta - \gamma) \bdf_\eta.
\]
Then \(\ithComp{\ccz}[\xi] = \word[\xi](\ccp) + \alpha - \beta = \alpha < \beta = \word[\xi](\ccp)\)
and \(\ithComp{\ccz}[\eta] = \word[\eta](\ccp) + 1 + \beta - \gamma = 1 \le \gamma - \beta = \word[\eta](\ccp)\),
and hence \(\bdz <_{\ccf} \word(\ccp)\).
Therefore \(\ccd(\bdz, \word(\cca)) < \ccd\) for some \(0 \le \cca < \ccp\).
Note that
\[
\begin{alignedat}{2}
 \word<\xi>(\ccp) &= \beta, &\quad  \word<\eta>(\ccp) &= \gamma,\\
 \ithComp{\ccz}<\xi> &= \alpha, &  \ithComp{\ccz}<\eta> &= \alpha + 1.\\
\end{alignedat}
\]
Since \(\cdn[\mnGen](\mVec{1 \\ \alpha}) < \cdn[\mnGen](\mVec{1 \\ \beta})\),
it follows that \(\cca = \alpha + 1 \neq \{\beta, \gamma\}\) and \(\cdn[\mnGen](\mVec{1 \\ \alpha + 1}) > 0\).
Thus \(\Cdn[\mnGen](\mVec{1 \\ \alpha + 1}) \setminus \Theta \neq \emptyset\).
It follows from Lemma \ref{orgfb723c7} that \(\cdn[\mnGen](\mVec{1 \\ \alpha}) \ge \cdn[\mnGen](\mVec{1 \\ \beta}) \ge 1\),
a contradiction.
 
\end{mycase}

\begin{mycase}[\(\tmpcaselabelsecond\)]
 \comment{Case. [\(\tmpcaselabelsecond\)]}
\label{sec:orgba38119}

Let \(\xi \in \Cdn[\mnGen](\mVec{1 \\ \gamma})\).
If \(\gamma > \alpha\), then the proof follows from Case 1.
Suppose that \(\gamma < \alpha < \beta\). 
Since \(\ccM = \eta\), it follows from (i) that
\(\alpha - \gamma > \beta - \gamma\), which is impossible. \hspace*{\fill}\(\qedhere\)
 
\end{mycase} 
\end{proof}

\begin{lemma}
 \comment{Lem.}
\label{sec:orgcf762df}
\label{orga704c68}
If there exists \(\alpha, \gamma \in \FF_\ccp\) satisfying the following two conditions,
then \(\MaxDimLinearLex|\ccF, \ccd| = 1\).

(1) \(\gamma \ge 1\), \(\Cdn[\mnGen](\begin{bsmallmatrix} 1 \\ \ccp - \gamma \end{bsmallmatrix}) \setminus \Theta \neq \emptyset\).

(2) \(1 \le \alpha \le \ccp - 2 \gamma - 1\), \(\Cdn[\mnGen](\mVec{1 \\ \alpha}) \setminus \Theta \neq \emptyset\).
 
\end{lemma}

\begin{proof}
 \comment{Proof.}
\label{sec:orgd10871c}
It suffices to show that
if \(\word(\ccp + \beta) = \word(\ccp) + \word(\beta)\) for \(1 \le \beta < \gamma\),
then \(\word(\ccp + \gamma) \neq \word(\ccp) + \word(\gamma)\).

Let
\(\ccN \in \Cdn[\mnGen](\begin{bsmallmatrix} 1 \\ \ccp - \gamma \end{bsmallmatrix}) \setminus \Theta\) and
\(\ccM \in \Cdn[\mnGen](\mVec{1 \\ \alpha}) \setminus \Theta\).
Since \(\alpha < \ccp - \gamma\), it follows from Lemma \ref{orgfc26eca} that \(\ccM > \ccN\).
Note that
\(\word[\ccN](\ccp) + \word[\ccN](\gamma) = (\ccp - \gamma) \oplus_\ccp \gamma = 0\) and
\(\word[\ccM](\ccp) + \word[\ccM](\gamma) = \alpha + \gamma\).
Let
\[
 \bdy = \word(\ccp) + \word(\gamma) - (\alpha + \gamma) \bdf_\ccM + (\alpha + \gamma) \bdf_\ccN.
\]
We see that \(\ccy_{[\ccM]} = 0 < \alpha + \gamma = \word[\ccM](\ccp) + \word[\ccM](\gamma)\),
and hence \(\bdy <_{\ccF} \word(\ccp) + \word(\gamma)\).
Since \(1 \le \alpha \le \ccp - 2 \gamma - 1\), it follows that
\(\alpha + \gamma \le \ccp - \gamma - 1\);
in particular, \(\alpha + \gamma \neq \ccp - \gamma, \ccp - \gamma + 1, \ldots, \ccp - 1, 0\).
Hence
\(\ccd(\bdy, \word(\cca)) \ge \ccd\) for \(0 \le \cca < \ccp + \gamma\).
Therefore \(\word(\ccp + \gamma) \neq \word(\ccp) + \word(\gamma)\).

\begin{center}
\begin{tabular}{lll}
 & \(\ccN\) & \(\ccM\)\\
\hline
\(\word[\cci](1)\) & \(1\) & \(1\)\\
\(\word[\cci](\ccp)\) & \(\ccp - \gamma\) & \(\alpha\)\\
\(\word[\cci](\ccp) + \word[\cci](1)\) & \(\ccp - \gamma + 1\) & \(\alpha + 1\)\\
\(\word[\cci](\ccp) + \word[\cci](2)\) & \(\ccp - \gamma + 2\) & \(\alpha + 2\)\\
\(\hphantom{\word[\cci](\ccp)}\ \, \vdots\) &  & \\
\(\word[\cci](\ccp) + \word[\cci](\gamma - 1)\) & \(\ccp - 1\) & \(\alpha + \gamma - 1\)\\
\(\word[\cci](\ccp) + \word[\cci](\gamma)\) & \(0\) & \(\alpha + \gamma\)\\
\(\ccy_{[\cci]}\) & \(\alpha + \gamma\) & \(0\)\\
\end{tabular}
\end{center}
\end{proof}

\begin{proposition}
 \comment{Prop.}
\label{sec:orgc367d2c}
\label{orgb060bb2}
Let \(\ccp \ge 7\).
If \(\Cdn[\mnGen](\begin{bsmallmatrix} 1 \\ \ccp - 1 \end{bsmallmatrix}) \setminus \Theta \neq \emptyset\),
then \(\MaxDimLinearLex|\ccF, \ccd| = 1\).
 
\end{proposition}

\begin{proof}
 \comment{Proof.}
\label{sec:org48a4b1c}

Lemma \ref{orga081bbf} implies that
\(\Cdn[\mnGen](\mVec{1 \\ \alpha}) \neq \emptyset\)
for \(0 \le \alpha \le \ccp - 2\).
Since \(\ccp \ge 7\), it follows that
\(\Cdn[\mnGen](\mVec{1 \\ \alpha}) \setminus \Theta \neq \emptyset\) for some \(\alpha\)
with \(1 \le \alpha \le \ccp - 3\).
Lemma \ref{orga704c68} shows that \(\MaxDimLinearLex|\ccF, \ccd| = 1\).
\end{proof}

\begin{proposition}
 \comment{Prop.}
\label{sec:org6cae5bc}
\label{org53f50d3}
Let \(\ccp \ge 7\). 
If \(\Cdn[\mnGen](\begin{bsmallmatrix} 1 \\ \ccp - 1 \end{bsmallmatrix}) = \Theta\), 
then \(\MaxDimLinearLex|\ccF, \ccd| = 1\).
 
\end{proposition}

\begin{proof}
 \comment{Proof.}
\label{sec:org58cbc6e}
Lemma \ref{orga081bbf} implies that
\(\Cdn[\mnGen](\mVec{1 \\ \alpha}) \neq \emptyset\)
for \(0 \le \alpha \le \ccp - 2\).
Thus \(\Cdn[\mnGen](\begin{bsmallmatrix} 1 \\ \ccp - 2 \end{bsmallmatrix}) \setminus \Theta\neq \emptyset\).
Since \(\ccp \ge 7\), we see that \(\ccp - 5 \ge 1\)
and \(\Cdn[\mnGen](\begin{bsmallmatrix} 1 \\ \ccp - 5 \end{bsmallmatrix}) \setminus \Theta \neq \emptyset\).
Lemma \ref{orga704c68} shows that \(\MaxDimLinearLex|\ccF, \ccd| = 1\).
\end{proof}

\begin{proposition}
 \comment{Prop.}
\label{sec:org5bcfaa5}
\label{org0734f0d}
Let \(\ccp \ge 11\).
If  \(\Cdn[\mnGen](\begin{bsmallmatrix} 1 \\ \ccp - 1 \end{bsmallmatrix}) \subseteq \Theta\)
and \(\cdn[\mnGen](\begin{bsmallmatrix} 1 \\ \ccp - 1 \end{bsmallmatrix}) = 1\),
then \(\MaxDimLinearLex|\ccF, \ccd| = 1\).
 
\end{proposition}

\begin{proof}
 \comment{Proof.}
\label{sec:org2b7f74a}

Suppose that
\(\Cdn[\mnGen](\begin{bsmallmatrix} 1 \\ \ccp - 2 \end{bsmallmatrix}) \setminus \Theta \neq \emptyset\). 
Since
\(\Cdn[\mnGen](\begin{bsmallmatrix} 1 \\ 1 \end{bsmallmatrix}) \setminus \Theta \neq \emptyset\)
or
\(\Cdn[\mnGen](\begin{bsmallmatrix} 1 \\ 2 \end{bsmallmatrix}) \setminus \Theta \neq \emptyset\),
it follows from Lemma \ref{orga704c68} that \(\MaxDimLinearLex|\ccF, \ccd| = 1\).

Suppose that \(\Cdn[\mnGen](\begin{bsmallmatrix} 1 \\ \ccp - 2 \end{bsmallmatrix}) \setminus \Theta = \emptyset\).
Then
\(\Cdn[\mnGen](\begin{bsmallmatrix} 1 \\ \ccp - 1 \end{bsmallmatrix}) \cup \Cdn[\mnGen](\begin{bsmallmatrix} 1 \\ \ccp - 2 \end{bsmallmatrix}) = \Theta\).
Since \(\ccp \ge 11\), we see that  \(\ccp - 7 \ge 1\).
Moreover,
\(\Cdn[\mnGen](\begin{bsmallmatrix} 1 \\ \ccp - 3 \end{bsmallmatrix}) \setminus \Theta \neq \emptyset\)
and
\(\Cdn[\mnGen](\begin{bsmallmatrix} 1 \\ \ccp - 7 \end{bsmallmatrix}) \setminus \Theta \neq \emptyset\).
Lemma \ref{orga704c68} shows that \(\MaxDimLinearLex|\ccF, \ccd| = 1\).
\end{proof}

\begin{lemma}
 \comment{Lem.}
\label{sec:org4f17579}
\label{orga67785b}
If one of the following two conditions holds,
then \(\MaxDimLinearLex|\ccF, \ccd| = 1\). 

\begin{enumerate}
\item \(\cdn[\mnGen](\mVec{1 \\ p - 1}) \le \cdn[\mnGen](\mVec{0 \\ 1}) - 1\) and  \(\Cdn[\mnGen](\mVec{1 \\ \alpha}) \setminus \Theta \neq \emptyset\) for some \(1 \le \alpha \le \ccp - 2\).
\item \(\cdn[\mnGen](\mVec{1 \\ p - 1}) \le \cdn[\mnGen](\mVec{0 \\ 1}) - 2\) and \(\Cdn[\mnGen](\mVec{1 \\ \alpha}) = \Theta\) for some \(1 \le \alpha \le \ccp - 2\).
\end{enumerate}
 
\end{lemma}

\begin{proof}
 \comment{Proof.}
\label{sec:org022b4ed}
When (1) holds, let \(\ccM \in \Cdn[\mnGen](\mVec{1 \\ \alpha}) \setminus \Theta\).
When (2) holds, let \(\ccM = \xi\).
Let
\[
 \bdy = \word(\ccp) + \word(1) - (1 + \alpha) \bdf_\ccM.
\]
Then \(\word[\ccM](\ccp) + \word[\ccM](1) = \alpha + 1 > 0\)
and \(\ithComp{\ccy}[\ccM] = 0\).
Thus \(\bdy <_\ccF \word(\ccp) + \word(1)\).
It suffices to show that \(\mbigDist{\bdy}{\word(\cca)} \ge \ccd\) for  \(0 \le \cca \le \ccp\).

First, suppose that (1) holds.
Since \(\ithComp{\ccy}<\ccM>= 0\) and \(\ccM \in \Cdn[\mnGen](\mVec{1 \\ \alpha})\),
it follows that 
\[
 \mbigDist{\bdy}{\word(\cca)} \ge \mbigDist{\word(\ccp) + \word(1)}{\word(\cca)} \ge \ccd
\tfor 1 \le \cca \le \ccp. 
\]
Moreover, since \(\cdn[\mnGen](\mVec{1 \\ p - 1}) \le \cdn[\mnGen](\mVec{0 \\ 1}) - 1\) and
\(\Cdn[\mGen{0}[\bdy]](\mVec{1 \\ 0}) = \Cdn[\mnGen](\mVec{1 \\ p - 1}) \cup \set{\ccM}\),
it follows that 
\[
 \mbigDist{\bdy}{\word(0)} = \cdn[\mnGen](1) + \cdn[\mnGen](\mVec{0 \\ 1}) - \cdn[\mnGen](\mVec{1 \\ p - 1}) - 1  \ge \ccd.
\]
Thus \(\mbigDist{\bdy}{\word(\cca)} \ge \ccd\) for \(0 \le \cca \le \ccp\).

Next, suppose that (2) holds.
Since \(\ithComp{\ccy}<\ccM> = \ithComp{\ccy}<\xi> = \ithComp{\ccy}<\eta> = 0\) and \(\ccM \in \Cdn[\mnGen](\mVec{1 \\ \alpha})\),
it follows that \(\mbigDist{\bdy}{\word(\cca)} \ge \ccd\)
for \(1 \le \cca \le \ccp\).
Moreover, since \(\cdn[\mnGen](\mVec{1 \\ p - 1}) \le \cdn[\mnGen](\mVec{0 \\ 1}) - 2\) and
\(\Cdn[\mGen{0}[\bdy]](\mVec{1 \\ 0}) = \Cdn[\mnGen](\mVec{1 \\ p - 1}) \cup \Theta\),
it follows that
\[
 \mbigDist{\bdy}{\word(0)} = \cdn[\mnGen](1) + \cdn[\mnGen](\mVec{0 \\ 1}) - \cdn[\mnGen](\mVec{1 \\ p - 1}) - 2  \ge \ccd.
\]
Therefore \(\mbigDist{\bdy}{\word(\cca)} \ge \ccd\) for \(0 \le \cca \le \ccp\).
\end{proof}

\begin{proposition}
 \comment{Prop.}
\label{sec:orgfb47e42}
\label{orgc4b4730}
Let \(\ccp \ge 5\).
If  \(\Cdn[\mnGen](\begin{bsmallmatrix} 1 \\ \ccp - 1 \end{bsmallmatrix}) = \emptyset\),
then \(\MaxDimLinearLex|\ccF, \ccd| = 1\). 
 
\end{proposition}

\begin{proof}
 \comment{Proof.}
\label{sec:org39f59a1}

If \(\Cdn[\mnGen](\mVec{1 \\ \alpha}) \setminus \Theta \neq \emptyset\) for some \(\alpha\) with \(1 \le \alpha \le \ccp - 2\),
then Lemma \ref{orga67785b} shows that \(\MaxDimLinearLex|\ccF, \ccd| = 1\). 
Suppose that \(\Cdn[\mnGen](\mVec{1 \\ \alpha}) \setminus \Theta= \emptyset\) for \(1 \le \alpha \le \ccp - 2\).
We first show that \(\cdn[\mnGen](\begin{bsmallmatrix} 1 \\ 0 \end{bsmallmatrix}) \le 2\).
Suppose that
\(\cdn[\mnGen](\begin{bsmallmatrix} 1 \\ 0 \end{bsmallmatrix}) \ge 3\).
Then there exists \(\ccM \in \Cdn[\mnGen](\begin{bsmallmatrix} 1 \\ 0 \end{bsmallmatrix}) \setminus \Theta\).
Since \(\ccp \ge 5\),
we see that there exists \(\beta\) such that
\(\Cdn[\mnGen](\mVec{1 \\ \beta}) = \emptyset\) and \(1 \le \beta \le \ccp - 2\).
We also see that there exists \(\ccN \in \Cdn[\mnGen](\mVec{0 \\ 1}) \setminus \Theta\) since
\(\cdn[\mnGen](\mVec{0 \\ 1}) \ge \Cdn[\mnGen](\mVec{1 \\ 0})\ge 3\).
Note that \(\ccN > \ccM\).
Let
\[
\bdy = \word(\ccp) + \beta \bdf_{\ccM} - \bdf_{\ccN}.
\]
Then \(\ithComp{\ccy}[\ccN] = \word[\ccN](\ccp) - 1 = 0\),
and hence \(\bdy <_{\ccF} \word(\ccp)\).
Moreover, 
\[
 \mbigDist{\bdy}{\word(\beta)} = \cdn[\mnGen](1) + \cdn[\mnGen](\mVec{0 \\ 1}) - \cdn[\mnGen](\mVec{1 \\ \beta}) - 1 \ge \ccd.
\]
Thus 
\[
 \ccd(\bdy, \word(\cca)) \ge \ccd \tfor 0 \le \cca < \ccp,
\]
a contradiction.
Hence \(\cdn[\mnGen](\begin{bsmallmatrix} 1 \\ 0 \end{bsmallmatrix})\le 2\).
Since \(\Cdn[\mnGen](\mVec{1 \\ \alpha}) \subseteq \Theta\) for \(1 \le \alpha \subseteq \ccp - 1\),
it follows that \(\ccd = \sum_{\alpha = 0}^{\ccp - 1} \cdn[\mnGen](\mVec{1 \\ \alpha}) \le 4\)
Thus \(\ccd = 3 \tor 4\).

\resetmycase
\begin{mycase}[\(\ccd = 4\)]
 \comment{Case. [\(\ccd = 4\)]}
\label{sec:org8a571ad}

Since \(\ccd =  \sum_{\alpha = 0}^{\ccp - 1} \cdn[\mnGen](\mVec{1 \\ \alpha})\)
and \(\Cdn[\mnGen](\mVec{1 \\ \alpha}) \subseteq \Theta\) for \(\alpha \neq 0\),
it follows that \(\cdn[\mnGen](\begin{bsmallmatrix} 1 \\ 0 \end{bsmallmatrix}) = 2\)
and \(\Cdn[\mnGen](\mVec{1 \\ 1}) \cup \Cdn[\mnGen](\mVec{1 \\ 2}) = \Theta\).
Let \(\cdn[\mnGen](\begin{bsmallmatrix} 1 \\ 0 \end{bsmallmatrix}) = \set{\ccN, \ccM}\)
and
\[
 \bdy = \bdf_{\ccM} + \bdf_{\ccN} + \bdf_{\xi}.
\]
Then \(\mDist{\bdy}{\word(0)} = 4\).
Since \(\xi \in \Cdn[\mnGen](1)\), it follows that \(\bdy = \word(1)\).
Recall that \(\Cdn[\mnGen](\mVec{1 \\ 1}) \cup \Cdn[\mnGen](\mVec{1 \\ 2}) = \Theta\).
Assume that \(\cdn[\mnGen](\mVec{1 \\ 2}) = 1\).
Let \(\Cdn[\mnGen](\mVec{0 \\ 1}) = \set{\cci_0, \cci_1}\) and 
\(\ccE' = \transpose{\begin{bmatrix} \transpose{\transpose{\bde_{\xi}} &  \transpose{\bde_{\eta}} & \bde_{\ccN}} & \transpose{\bde_{\ccM}} & \transpose{\bde_{\cci_0}} & \transpose{\bde_{\cci_1}} \end{bmatrix}}\).
Then we see that
\begin{align*}
 \word(1) &= \mWord{1 1 1 1 0 0} \ccE',\\
 \word(\ccp) &= \mWord{1 2 0 0 1 1} \ccE'.\\
\end{align*}
Let \(\bdy = \word(\ccp)  - \bdf_{i_1} + 3 \bdf_{\ccM}\). Then
\[
 \bdy = \mWord{1 2 3 0 1 0} \ccE'.\\
\]
Since \(i_1 > \ccM\), we see that \(\bdy <_{\ccF} \word(\ccp)\).
Moreover, \(\ccd(\bdy, \word(\cca)) \ge 4\) for \(0 \le \cca < \ccp\), a contradiction.
Hence \(\cdn[\mnGen](\mVec{1 \\ 2}) = 0\) and \(\Cdn[\mnGen](\mVec{1 \\ 1}) = \Theta\).
Lemma \ref{orga67785b} shows that \(\MaxDimLinearLex|\ccF, \ccd| = 1\). 
 
\end{mycase}

\begin{mycase}[\(\ccd = 3\)]
 \comment{Case. [\(\ccd = 3\)]}
\label{sec:org638aaee}

\quad

\noindent
\textbf{Case 2.1.} \(\cdn[\mnGen](\mVec{0 \\ 1}) = 2\).

Since \(\Cdn[\mnGen](1) \cap \Theta \neq \emptyset\),
we see that \(\Cdn[\mnGen](\mVec{0 \\ 1}) \setminus \Theta \neq \emptyset\).
It follows from  Lemma \ref{org042dd48} that
\(\cdn[\mnGen](\mVec{1 \\ \alpha}) = 2\) for some \(\alpha\).

Assume that \(\cdn[\mnGen](\begin{bsmallmatrix} 1 \\ 0 \end{bsmallmatrix}) = 2\).
Since \(\sum_{\alpha \in \FF_\ccp} \cdn[\mnGen](\begin{bsmallmatrix} 1 \\ \alpha \end{bsmallmatrix}) = 3\),
it follows from Lemma \ref{orga081bbf} that \(\cdn[\mnGen](\begin{bsmallmatrix} 1 \\ 1 \end{bsmallmatrix}) = 1\).
Moreover, there exists \(\ccN \in \Cdn[\mnGen](\begin{bsmallmatrix} 1 \\ 0 \end{bsmallmatrix}) \setminus \Theta\)
and 
\(\ccM \in \Cdn[\mnGen](\mVec{0 \\ 1}) \setminus \Theta\).
Note that \(\ccN < \ccM\).
Let \(\Cdn[\mnGen](\mVec{1 \\ 1}) = \set{i_0}\),
 \(\Cdn[\mnGen](\mVec{1 \\ 0}) = \set{\ccN, i_1}\),
 \(\Cdn[\mnGen](\mVec{0 \\ 1}) = \set{\ccM, i_2}\),
and \(\ccE' = \transpose{\begin{bmatrix}\transpose{\bde_{i_0}} & \transpose{\bde_{N}} & \transpose{\bde_{i_1}} & \transpose{\bde_{M}} & \transpose{\bde_{i_2}}\end{bmatrix}}\).
Then
\begin{align*}
 \word(1) &= \mWord{1  1  1  0  0} E'.\\
 \word(\ccp) &= \mWord{1  0  0  1  1} E'.\\
\end{align*}
Let \(\bdy = \word(\ccp) + 2 \bdf_{\ccN} - \bdf_{\ccM}\). Then
\begin{align*}
 \bdy &= \mWord{1  2  0  0  1} E'.
\end{align*}
Since \(\ithComp{\ccy}[\ccM] = 0\) and \(\ccN < \ccM\), it follows that \(\bdy <_{\ccF} \word(\ccp)\).
Moreover, \(\ccd(\bdy, \word(\cca)) \ge 3\) for \(0 \le \cca < \ccp\), a contradiction.
Therefore \(\cdn[\mnGen](\begin{bsmallmatrix} 1 \\ 0 \end{bsmallmatrix}) = 1\).
This implies that \(\cdn[\mnGen](\begin{bsmallmatrix} 1 \\ 1 \end{bsmallmatrix}) = 2\).
Thus \(\Cdn[\mnGen](\begin{bsmallmatrix} 1 \\ 1 \end{bsmallmatrix}) = \Theta\).
Therefore Lemma \ref{orga67785b} shows that \(\MaxDimLinearLex|\ccF, \ccd| = 1\).

\noindent
\textbf{Case 2.2.} \(\cdn[\mnGen](\mVec{0 \\ 1}) = 1\).

It suffices to show that
if \(\word(\ccp + 1) = \word(\ccp) + \word(1)\), 
then \(\word(\ccp + 2) \neq \word(\ccp) + \word(2)\).
Since \(\ccd = 3\), it follows that
\(\cdn[\mnGen](\begin{bsmallmatrix} 1 \\ 0 \end{bsmallmatrix}) = 
\cdn[\mnGen](\begin{bsmallmatrix} 1 \\ 1 \end{bsmallmatrix}) = 
\cdn[\mnGen](\begin{bsmallmatrix} 1 \\ 2 \end{bsmallmatrix}) = 1\).
Let \(\Cdn[\mnGen](\begin{bsmallmatrix} 1 \\ 0 \end{bsmallmatrix}) =  \{\ccN\}\) and
\(\Cdn[\mnGen](\mVec{0 \\ 1}) =  \{\ccM\}\).
Note that
\[
\bigl(\Cdn[\mnGen](\begin{bsmallmatrix} 1 \\ 1 \end{bsmallmatrix}),  
\Cdn[\mnGen](\begin{bsmallmatrix} 1 \\ 2 \end{bsmallmatrix}) \bigr)
= \bigl(\{\xi\}, \{\eta\}\bigr)\ \text{or}\ \bigl(\{\eta\}, \{\xi\} \bigr).
\]
We show that 
\(\bigl(\Cdn[\mnGen](\begin{bsmallmatrix} 1 \\ 1 \end{bsmallmatrix}),  
\Cdn[\mnGen](\begin{bsmallmatrix} 1 \\ 2 \end{bsmallmatrix}) \bigr)
= \bigl(\{\xi\}, \{\eta\}\bigr)\).
If \(\bigl(\Cdn[\mnGen](\begin{bsmallmatrix} 1 \\ 1 \end{bsmallmatrix}),  
\Cdn[\mnGen](\begin{bsmallmatrix} 1 \\ 2 \end{bsmallmatrix}) \bigr)
= \bigl(\{\xi\}, \{\eta\}\bigr)\),
then \(\word[\xi](\ccp) = 1\) and \(\word[\eta](\ccp) = 1\).
If 
\(\bigl(\Cdn[\mnGen](\begin{bsmallmatrix} 1 \\ 1 \end{bsmallmatrix}),  
\Cdn[\mnGen](\begin{bsmallmatrix} 1 \\ 2 \end{bsmallmatrix}) \bigr)
= \bigl(\{\eta\}, \{\xi\} \bigr)\),
then \(\word[\xi](\ccp) = 2\) and
\(\word[\eta](\ccp) = \ccp - 1\).
Thus \(\Cdn[\mnGen](\begin{bsmallmatrix} 1 \\ 1 \end{bsmallmatrix}) = \{\xi\}\) and
\(\Cdn[\mnGen](\begin{bsmallmatrix} 1 \\ 2 \end{bsmallmatrix}) = \{\eta\}\).

Let \(\ccF' = \transpose{\begin{bmatrix} \transpose{\bdf_\xi} & \transpose{\bdf_\eta } & \transpose{\bdf_{\ccN}} & \transpose{\bdf_{\ccM}} \end{bmatrix}}\),
and \(\ccE' = \transpose{\begin{bmatrix} \transpose{\bde_\xi} & \transpose{\bde_\eta } & \transpose{\bde_{\ccN}} & \transpose{\bde_{\ccM}} \end{bmatrix}}\).
Then
\begin{align*}
 \word(1) &= \mWord{1 0 1 0} \ccF' 
             = \mWord{1 1 1 0} \ccE'.\\
 \word(\ccp) &= \mWord{1 1 0 1} \ccF' 
             = \mWord{1 2 0 1} \ccE'.
\end{align*}
It follows that 
\begin{align*}
 \word(\ccp) + \word(1) &= \mWord{2 1 1 1} \ccF' 
             = \mWord{2  3  1  1} \ccE'.\\
 \word(\ccp) + \word(2) &= \mWord{3 1 2 1} \ccF' 
             = \mWord{3  4  2  1} \ccE'.
\end{align*}
Let
\begin{align*}
 \bdy = \mWord{0  1  2  1} \ccF' 
      = \mWord{0  1  2  1} \ccE'.
\end{align*}
Then \(\bdy <_{\ccF} \word(\ccp) + \word(2)\)
and 
\(\ccd(\bdy, \word(\cca)) \ge \ccd\)  for \(0 \le \cca \le \ccp + 1\).
Therefore \(\word(\ccp + 2) \neq \word(\ccp) + \word(2)\). \(\hspace*{\fill}\) \(\qedhere\)
 
\end{mycase} 
\end{proof}

\begin{proposition}
 \comment{Prop.}
\label{sec:org3685491}
\label{org230bb86}
If \(\ccp \ge 5\), then \(\MaxDimLinearLex|\ccF, \ccd| \le 2\). 
 
\end{proposition}

\begin{proof}
 \comment{Proof.}
\label{sec:orgaf60725}

From Propositions \ref{orgb060bb2}--\ref{orgc4b4730},
if one of the following conditions holds,
then \(\MaxDimLinearLex|\ccF, \ccd| = 1\). 

\begin{enumerate}
\item \(\ccp \ge 7\) and \(\Cdn[\mnGen](\begin{bsmallmatrix} 1 \\ \ccp - 1 \end{bsmallmatrix}) \setminus \Theta \neq \emptyset\).

\item \(\ccp \ge 7\) and \(\Cdn[\mnGen](\begin{bsmallmatrix} 1 \\ \ccp - 1 \end{bsmallmatrix}) = \Theta\).

\item \(\ccp \ge 11\) and \(\Cdn[\mnGen](\begin{bsmallmatrix} 1 \\ \ccp - 1 \end{bsmallmatrix}) = \set{\xi} \tor \set{\eta}\).

\item \(\ccp \ge 5\) and \(\Cdn[\mnGen](\begin{bsmallmatrix} 1 \\ \ccp - 1 \end{bsmallmatrix}) = \emptyset\).
\end{enumerate}

The remaining cases are as follows.

\noindent
\((1)'\)  \(\ccp = 5, 7\) and \(\Cdn[\mnGen](\begin{bsmallmatrix} 1 \\ \ccp - 1 \end{bsmallmatrix}) = \set{\xi} \tor \set{\eta}\).

\noindent
\((2)'\) \(\ccp = 5\) and \(\Cdn[\mnGen](\begin{bsmallmatrix} 1 \\ \ccp - 1 \end{bsmallmatrix}) \setminus \Theta \neq \emptyset\).

\noindent
\((3)'\) \(\ccp = 5\) and \(\Cdn[\mnGen](\begin{bsmallmatrix} 1 \\ \ccp - 1 \end{bsmallmatrix}) = \Theta\).

\medskip
\noindent
\((1)'\) \(\ccp = 5, 7\) and \(\Cdn[\mnGen](\begin{bsmallmatrix} 1 \\ \ccp - 1 \end{bsmallmatrix}) = \set{\xi}\) or \(\set{\eta}\).

Lemma \ref{orga081bbf} shows that
\(\Cdn[\mnGen](\mVec{1 \\ \alpha}) \neq \emptyset\) for \(0 \le \alpha \le \ccp - 1\).
Since \(\Cdn[\mnGen](\begin{bsmallmatrix} 1 \\ \ccp - 1 \end{bsmallmatrix}) \cap \Theta \neq \emptyset\),
it follows that
\(\Cdn[\mnGen](\mVec{1 \\ 1}) \setminus \Theta \neq \emptyset\) or
\(\Cdn[\mnGen](\mVec{1 \\ 2}) \setminus \Theta \neq \emptyset\).
If \(\cdn[\mnGen](\mVec{0 \\ 1}) \ge 2\),
then Lemma \ref{orga67785b} 
shows that \(\MaxDimLinearLex|\ccF, \ccd| = 1\). 
Hence we may assume that
\(\cdn[\mnGen](\mVec{0 \\ 1}) = 1\).
Therefore \(\ccd = \ccp\).
By computer, we see that \(\MaxDimLinearLex|\ccF, \ccd| = 1\).

\medskip
\noindent
\((2)'\) \(\ccp = 5\) and \(\Cdn[\mnGen](\begin{bsmallmatrix} 1 \\ \ccp - 1 \end{bsmallmatrix}) \setminus \Theta \neq \emptyset\).

If \(\Cdn[\mnGen](\mVec{1 \\ \alpha}) \setminus \Theta \neq \emptyset\) for some
\(1 \le \alpha \le \ccp - 3 = 2\),
then \(\MaxDimLinearLex|\ccF, \ccd| = 1\). 
Hence we may assume that
\(\Cdn[\mnGen](\begin{bsmallmatrix} 1 \\ 1 \end{bsmallmatrix}) \cup \Cdn[\mnGen](\begin{bsmallmatrix} 1 \\ 2 \end{bsmallmatrix}) = \Theta\).
Since
\(\Cdn[\mnGen](\begin{bsmallmatrix} 1 \\ 4 \end{bsmallmatrix}) \neq \emptyset\),
it follows from Lemma \ref{orga081bbf} that
\(\cdn[\mnGen](\mVec{1 \\ \alpha}) \neq 0\),
and hence \(\cdn[\mnGen](\mVec{1 \\ 1}) = \cdn[\mnGen](\mVec{1 \\ 2}) = 1\).
Thus \(\Cdn[\mnGen](\mVec{1 \\ 1}) = \set{\xi}\)
and \(\Cdn[\mnGen](\mVec{1 \\ 2}) = \set{\eta}\).
Note that \(\word[\eta](\ccp) = 1\) and \(\word<\eta>(\ccp) = 2\).
Assume that \(\cdn[\mnGen](\mVec{0 \\ 1}) \ge 2\).
Let \(\bdy = \word(\ccp) - \bdf_\eta\).
Then \(\ithComp{\ccy}[\eta] = 0\), and hence \(\bdy <_{\ccF} \word(\ccp)\).
Moreover,
\[
 \mbigDist{\bdy}{\word(1)} = \cdn[\mnGen](1) + \cdn[\mnGen](\mVec{0 \\ 1}) - 2 \ge \ccd.
\]
Thus \(\mbigDist{\bdy}{\word(\cca)} \ge \ccd\) for \(0 \le \cca < \ccp\),
a contradiction.
Hence \(\cdn[\mnGen](\mVec{0 \\ 1}) = 1\) and \(\ccd = 5\). 
By computer, we see that
\[
 \MaxDimLinearLex|\ccF, 5|
= \begin{cases}
 2 & \tif \xi = 2 \tand \eta \in \set{3,4,5},\\
 1 & \totherwise.
\end{cases}
\]

\medskip
\noindent
\((3)'\) \(\ccp = 5\) and \(\Cdn[\mnGen](\begin{bsmallmatrix} 1 \\ \ccp - 1 \end{bsmallmatrix}) = \Theta\).

Lemma \ref{orga081bbf} shows that
\(\Cdn[\mnGen](\mVec{1 \\ \alpha}) \neq \emptyset\) for \(0 \le \alpha \le 4\).
Note that
\(\Cdn[\mnGen](\mVec{1 \\ 4}) = \Theta\) and
\(\Cdn[\mnGen](\mVec{1 \\ 3}) \setminus \Theta \neq \emptyset\).
Lemma \ref{orgfb723c7} shows that \(\cdn[\mnGen](\begin{bsmallmatrix} 1 \\ 1 \end{bsmallmatrix}) \ge 2\).
Let \(\ccM , \ccM' \in \Cdn[\mnGen](\begin{bsmallmatrix} 1 \\ 1 \end{bsmallmatrix})\).
We show that  \(\max \{\ccM, \ccM'\} > \xi\).
Let
\[
 \bdy = 
\word(\ccp) - 3  \bdf_\xi + 3 (\bdf_\ccM + \bdf_\ccM').
\]
We see that \(\ccd(\bdy, \word(\cca)) \ge \ccd\) for \(0 \le \cca < \ccp\).
Hence \(\word(\ccp) <_{\ccF} \bdy\). 
This implies that \(\max \{\ccM, \ccM'\} > \xi\).

\begin{center}
\begin{tabular}{lrrrr}
 & \(\xi\) & \(\eta\) & \(\ccM\) & \(\ccM'\)\\
\hline
\(\word<\nIi>(1)\) & 1 & 1 & 1 & 1\\
\(\word<\nIi>(\ccp)\) & \(4\) & \(4\) & 1 & 1\\
\(\ithComp{\ccy}<\nIi>\) & 1 & 1 & \(4\) & \(4\)\\
\end{tabular}
\end{center}

Let
\[
 \bdz =  \word(\ccp) + \word(1) + 2 \bdf_{\xi} - 2 (\bdf_{\ccM} + \bdf_{\ccM'}).
\]
Then \(\bdz <_{\ccF} \word(1) + \word(\ccp)\) and
\(\ccd(\bdz, \word(\cca)) \ge \ccd\) for \(0 \le \cca \le \ccp\).
Therefore \(\word(\ccp + 1) \neq \word(\ccp) + \word(1)\).

\begin{center}
\begin{tabular}{lrrrr}
 & \(\xi\) & \(\eta\) & \(\ccM\) & \(\ccM'\)\\
\hline
\(\word<\nIi>(1)\) & 1 & 1 & 1 & 1\\
\(\word<\nIi>(\ccp)\) & \(4\) & \(4\) & 1 & 1\\
\(\word<\nIi>(\ccp) + \word<\nIi>(1)\) & 0 & 0 & 2 & 2\\
\(\ithComp{\ccz}<\nIi>\) & 2 & 2 & 0 & 0\\
\end{tabular}
\end{center}
\end{proof}

\comment{Bib}
\label{sec:org7f918d9}
\bibliographystyle{plain}

\end{document}